\documentclass[11pt, letterpaper]{article}

\usepackage[german, french, american]{babel}
\usepackage[latin1]{inputenc}
\usepackage[T1]{fontenc}
\usepackage{amsmath, amssymb, amsfonts}
\usepackage{amsthm}
\usepackage{mathtools}
\usepackage{mathrsfs}
\usepackage{enumerate}
\usepackage{enumitem}
\usepackage{graphicx}
\usepackage{color}
\usepackage{enumitem}
\usepackage{xfrac}
\usepackage{verbatim}
\usepackage{multirow}
\usepackage{dsfont}

\usepackage{cases}
\usepackage{comment}

\usepackage{epstopdf}

\usepackage{thumbpdf}
\usepackage{hyperref}

\usepackage[margin = 1in]{geometry}

\theoremstyle{plain}
\newtheorem{theorem}{Theorem}[section]

\newtheorem{lemma}[theorem]{Lemma}
\newtheorem{remark}[theorem]{Remark}

\newtheorem{definition}[theorem]{Definition}
\newtheorem{assumption}[theorem]{Assumption}

\newcommand{\DIV}{\mathop{\operatorname{div}}}

\numberwithin{equation}{section}

\newcommand{\inner}[1]{{\big\langle #1 \big\rangle}}

\newcommand{\CAL}[1]{{\mathcal #1}}

\newcommand{\BR}{\mathbb{R}}
\newcommand{\BN}{\mathbb{N}}

\newcommand{\ltwo}[1]{\big\| #1 \big\|_{2}}
\newcommand{\linfinity}[1]{\big\| #1 \big\|_{\infty}}

\newcommand{\divq}{\operatorname{div} \mathbf{q}}

\newcommand{\calZ}{\CAL{Z}}

\newcommand{\calT}{\CAL{T}}
\newcommand{\calP}{\CAL{P}}

\newcommand{\zHtwo}{\|z\|_{H^2}}
\newcommand{\ztHtwo}{\|z_t\|_{H^2}}

\allowdisplaybreaks

\begin{document}

\title{Long-Time Behavior of Quasilinear Thermoelastic \\ Kirchhoff-Love Plates with Second Sound}

\author{Irena Lasiecka\thanks{Department of Mathematical Sciences, University of Memphis, Memphis, TN
and IBS, Polish Academy of Sciences, Warsaw, Poland \hfill \url{lasiecka@memphis.edu}} \and
Michael Pokojovy\thanks{Department of Mathematical Sciences, The University of Texas at El Paso, El Paso, TX
\hfill \url{mpokojovy@utep.edu}} \and
Xiang Wan\thanks{Department of Mathematics, Wayne State University, Detroit, MI \hfill \url{xiangwan@wayne.edu}}}

\date{\today}

\maketitle

\begin{abstract}
    We consider an initial-boundary-value problem for a thermoelastic Kirchhoff \& Love plate, thermally insulated and simply supported on the boundary,
    incorporating rotational inertia and a quasilinear hypoelastic response,
    while the heat effects are modeled using the hyperbolic Maxwell--Cattaneo--Vernotte law giving rise to a `second sound' effect.
    We study the local well-posedness of the resulting quasilinear mixed-order hyperbolic system
    in a suitable solution class of smooth functions mapping into Sobolev $H^{k}$-spaces.
    Exploiting the sole source of energy dissipation entering the system through the hyperbolic heat flux moment,
    provided the initial data are small in a lower topology (basic energy level corresponding to weak solutions), we prove a  nonlinear stabilizability estimate furnishing global existence \& uniqueness 
    and exponential decay of classical solutions.
\end{abstract}

\begin{center}
\begin{tabular}{p{1.0in}p{5.0in}}
	\textbf{Key words:} & Kirchhoff-Love plates; nonlinear thermoelasticity; hyperbolic thermoelasticity;
			global well-posedness; classical solutions; exponential stability \\
	\textbf{MSC (2010):} & Primary
	35L57,  
	35Q74,  
	74B20,  
	74F05,  
	74K20;  
	\\ &
	Secondary
	35A01,  
	35A02,  
	35A09,  
	35B40,  
	35B65   
\end{tabular}
\end{center}

\section{Introduction}
\label{SECTION_INTRODUCTION}

Consider a PDE model of a prismatic thermoelastic plate of a uniform thickness $h > 0$.
Let the bounded domain $\Omega \subset \mathbb{R}^{d}$, $d \in \{1, 2, 3\}$, with a smooth boundary $\partial \Omega$ parametrize the mid-plane of the plate.
Further, let $K \colon \mathbb{R} \to \mathbb{R}$ with $K'(0) > 0$ and $K''(0) = 0$ be a smooth function
related to the strain-stress curve (see Appendix Section \ref{SECTION_MODEL_DESCRIPTION}) and the plate thickness $h$.
Continuing, let $\alpha, \beta, \eta > 0$ and $\gamma, \tau, \sigma \geq 0$ be constant.
With $w$, $\theta$, $\mathbf{q}$ denoting the vertical displacement, a properly scaled thermal moment and the $x_{3}$-moment of the heat flux, respectively,
the associated dynamics is governed by a quasilinear plate equation
\begin{align}
	\label{EQUATION_QUASILINEAR_PDE_IN_W_THETA_AND_Q_1}
	w_{tt} - \gamma \triangle w_{tt} + \triangle K(\triangle w) + \alpha \triangle \theta &= 0 &\text{ in } (0, \infty) \times \Omega, \\
	\label{EQUATION_QUASILINEAR_PDE_IN_W_THETA_AND_Q_2}
	\beta \theta_{t} + \DIV \mathbf{q} + \sigma \theta - \alpha \triangle w_{t} &= 0 &\text{ in } (0, \infty) \times \Omega, \\
	\label{EQUATION_QUASILINEAR_PDE_IN_W_THETA_AND_Q_3}
	\tau \mathbf{q}_{t} + \mathbf{q} + \eta \nabla \theta &= \mathbf{0} &\text{ in } (0, \infty) \times \Omega\phantom{,}
\end{align}
subject to hinged boundary conditions
\begin{align}
	\label{EQUATION_BC_IN_W_THETA_AND_Q}
	w &= \triangle w = \theta = 0 &\text{ on } (0, \infty) \times \partial \Omega
\end{align}
and usual initial conditions
\begin{align}
	\label{EQUATION_IC_IN_W_THETA_AND_Q}
	w(0, \cdot) &= w^{0}, \quad w_{t}(0, \cdot) = w^{1}, \quad
	\theta(0, \cdot) = \theta^{0}, \quad \mathbf{q}(0, \cdot) = \mathbf{q}^{0} &\text{ in } \Omega.
\end{align}

Later in the paper, we will restrict our attention to the case $\gamma > 0$ and $\tau > 0$.
Also, to be consistent with the overwhelming majority of mathematical publications in the area, we will assume $\sigma = 0$.
While the latter constant needs to be positive (cf. Appendix Section \ref{SECTION_MODEL_DESCRIPTION} or \cite[Chapter 6.1]{LaLi1988}) from the physical point of view,
from the mathematical point of view, discarding this lower-order perturbation neither changes the underlying topology nor the qualitative stability properties of the system. 
Moreover, it makes the stability analysis more challenging as a natural dissipativity source is eliminated keeping the only damping arising from $\mathbf{q}$.

Depending on the choice of the $\gamma$ and $\tau$ parameters in Equations (\ref{EQUATION_QUASILINEAR_PDE_IN_W_THETA_AND_Q_1})--(\ref{EQUATION_QUASILINEAR_PDE_IN_W_THETA_AND_Q_3}),
the system represents various types of thermoelastic Kirchhoff--Love plates (viz. \cite[p. 2]{RiRaSeVi2018}).
While the presence of the $\gamma \triangle w_{tt}$-term in Equation (\ref{EQUATION_QUASILINEAR_PDE_IN_W_THETA_AND_Q_1}) 
accounts for rotational inertia for $\gamma > 0$ or neglects the latter if $\gamma = 0$, 
the positive relaxation time $\tau > 0$ in Equation (\ref{EQUATION_QUASILINEAR_PDE_IN_W_THETA_AND_Q_3})
originates from the Maxwell--Cattaneo--Vernotte's (or, for short, Cattaneo's) law of relativistic heat conduction
vs. the classic Fourier's law of heat conduction for $\tau = 0$.
Hence, Equations (\ref{EQUATION_QUASILINEAR_PDE_IN_W_THETA_AND_Q_1})--(\ref{EQUATION_QUASILINEAR_PDE_IN_W_THETA_AND_Q_3}) embody four possible thermoelastic plate models.
Further distinctions are made based on the response $K(\cdot)$ being linear vs. nonlinear
and the domain $\Omega$ being the full space $\mathbb{R}^{d}$ or a domain with boundary
such as bounded domains, exterior domains, half-spaces or wave-guides, etc.
Finally, in case $\Omega$ is a domain with boundary, a set of boundary conditions selected from a wide range of combinations need to be adopted
(cf. \cite[Chapter 2]{Am1970}, \cite[Chapter 4]{Il2004}, \cite[Chapter 1]{LaLi1988}, \cite{LaTri1998.1, LaTri1998.2, LaTri1998.3, LaTri1998.4}).

We continue our discussion with a brief review of the vast body of literature 
on Equations (\ref{EQUATION_QUASILINEAR_PDE_IN_W_THETA_AND_Q_1})--(\ref{EQUATION_QUASILINEAR_PDE_IN_W_THETA_AND_Q_3}).
Table \ref{TABLE_SYMMARY_OF_RESULTS} summarizes some of these results for bounded domains $\Omega$.
\begin{table}[!h]
    \centering
    
    \begin{tabular}{c||c|c}
                                    & $\gamma = 0$                                             & $\gamma > 0$                           \\
        \hline \hline
        \multirow{2}{*}{$\tau = 0$} & exponential stability                                 & exponential stability                     \\
                                    & maximal $L^{p}$-regularity/analyticity                & \underline{no} maximal $L^{p}$-regularity/analyticity \\
        \hline
        \multirow{2}{*}{$\tau > 0$} & \underline{no} exponential stability                  & exponential stability                     \\
                                    & \underline{no} maximal $L^{p}$-regularity/analyticity & \underline{no} maximal $L^{p}$-regularity/analyticity \\
        \hline \hline
    \end{tabular}
    
    \caption{Summary of results on Equations (\ref{EQUATION_QUASILINEAR_PDE_IN_W_THETA_AND_Q_1})--(\ref{EQUATION_QUASILINEAR_PDE_IN_W_THETA_AND_Q_3}) in bounded domains
    \label{TABLE_SYMMARY_OF_RESULTS}}
\end{table}

In the linear situation, i.e., $K(z) = az$ for some $a > 0$, 
it is well known that the thermoelastic system (\ref{EQUATION_QUASILINEAR_PDE_IN_W_THETA_AND_Q_1})--(\ref{EQUATION_QUASILINEAR_PDE_IN_W_THETA_AND_Q_3})
comprising a Kirchhoff-Love plate equation without rotational inertia ($\gamma = 0$) coupled  with
the standard parabolic heat equation ($\tau = 0$) generates an analytic semigroup on a respective finite-energy space
for a wide range of boundary conditions \cite{LaTri1998.1, LaTri1998.2, LaTri1998.3, LaTri1998.4}.
These results were subsequently improved by showing the maximal $L^{p}$-regularity of the underlying semigroup \cite{DeShi2009, DeShi2017, Na2009, NaShi2009}.
It was further shown the associated energy decays exponentially as $t \to \infty$ \cite{AvLa1997, Ki1992, LiZh1997, Shi1994}.
The exponential stability extends to the quasilinear situation with $K(z) = z + z^{3}$ in the class of (possibly, not unique) global weak solutions \cite{LaMaSa2008}.
Turning to strong solutions, both the well-posedness via maximal $L^{p}$-regularity and the exponential stability of small solutions hold true
for general superquadratic $C^{3}$-nonlinearities $K(\cdot)$ \cite{LaWi2013}.

Introducing a second sound effect into the system (\ref{EQUATION_QUASILINEAR_PDE_IN_W_THETA_AND_Q_1})--(\ref{EQUATION_QUASILINEAR_PDE_IN_W_THETA_AND_Q_3})
by replacing the Fourier's law with the Cattaneo's one ($\tau > 0$) while neglecting rotational inertia ($\gamma = 0$), 
the well-posedness (on the amended phase space) is preserved
but the maximal $L^{p}$-regularity/analyticity of the system is violated and the exponential decay of solutions is destroyed \cite{FeRi2012, QuRa2008, QuRa2011}.
A change of qualitative behavior also occurs in the full space $\Omega = \mathbb{R}^{d}$,
where a regularity-loss phenomena occur \cite{RaUe2016}.
In the nonlinear situation, the well-posedness still remains open -- even in the full-space (cf. \cite[p. 8140]{RaUe2017}).

Taking into account rotational inertia ($\gamma > 0$) and adopting Fourier's law of heat conduction ($\tau = 0$),
the thermoelastic plate system is rendered hyperbolic-parabolic.
In the nonlinear situation, both the semilinear \cite{AvLa1998} and the quasilinear problems \cite{LaPoWa2017} have been studied in bounded domains.
While the model exhibits hyperbolic characteristics,
the viscous diffusive effect of the heat equation still has beneficial effects on regularity of the overall system.
This allowed to perform a Kato-type fixed-point iteration to establish the local well-posedness of the quasilinear system
by decoupling the elastic and the thermal parts of the system \cite{LaPoWa2017}.
Under a smallness condition on the initial data, the energy dissipation through the thermal component of the system
was sufficient to prove a global stabilization estimate leading to global existence of classical solutions.
In the full space $\Omega = \mathbb{R}^{d}$, the well-posedness and decay rates have also been established --
both in the linear \cite{RaUe2016} and the nonlinear cases \cite{RaUe2017}.
A recent systematic study \cite{FeLiuRa2018} on abstract fractional-power thermoelastic plate systems should also be mentioned.

A number of control-theoretic results for linear Kirchhoff-Love thermoelastic plates with and without rotational inertia ($\gamma \geq 0$)
subject to Fourier's law of heat conduction ($\tau = 0$) are also known in the literature.
See, e.g., \cite{Av2000, AvLa2004, AvLa2005, BeNa2002, ElLaTri2000, La1990, LaSei2006, LaTri1998} and references therein.

Turning to the hyperbolic-hyperbolic case ($\tau > 0$, $\gamma > 0$),
linear well-posedness and exponential stability in bounded domains have been established 
and singular limits $\tau \to 0$, $\gamma \to 0$ have been studied \cite{RiRaSeVi2018}.
Similar investigations of the linear system in the full space were performed as well \cite{RaUe2016}
and subsequently generalized to the nonlinear case \cite{RaUe2017}.

The thrust of this article is to investigate nonlinear Equations (\ref{EQUATION_QUASILINEAR_PDE_IN_W_THETA_AND_Q_1})--(\ref{EQUATION_QUASILINEAR_PDE_IN_W_THETA_AND_Q_3})
in a bounded smooth domain $\Omega$ subject to the boundary conditions (\ref{EQUATION_BC_IN_W_THETA_AND_Q}).
The distinct features of our problem are:
\begin{itemize}
    \setlength{\itemsep}{1pt}
    
    \item
    The rotational inertia are accounted for by the presence of the $\gamma \triangle w_{tt}$-term ($\gamma > 0$)
    in the `elastic' Equation (\ref{EQUATION_QUASILINEAR_PDE_IN_W_THETA_AND_Q_1}). This makes the problem hyperbolic-like. 
    
    \item
    The heat conduction obeys Maxwell--Cattaneo--Vernotte's law rather than the classic Fourier's law
    which translates into (1) lack of dissipative effect in Equation (\ref{EQUATION_QUASILINEAR_PDE_IN_W_THETA_AND_Q_2}) and (2) lack of the regularity otherwise typically associated with the heat equation. These two properties -- dissipation and regularity - the ``key players'' in any quasilinear theory - are severely  compromised  by the model under consideration.
\end{itemize}
Unlike the hyperbolic-parabolic case ($\gamma > 0$, $\tau = 0$),
as we will later see in Section \ref{SECTION_EQUIVALENT_TRANFORMATION}, our system is hyperbolic-hyperbolic.
Therefore, no regularizing effects are inherited either from analyticity (when $\gamma = \tau = 0$)
or dissipativity-viscosity of the heat transfer (when $\gamma > 0$, $\tau = 0$).
Though  the nonlinear plate system (\ref{EQUATION_QUASILINEAR_PDE_IN_W_THETA_AND_Q_1})--(\ref{EQUATION_QUASILINEAR_PDE_IN_W_THETA_AND_Q_3})
for $\tau, \gamma > 0$ has previously been investigated by Racke \& Ueda \cite{RaUe2017}. However,
our present study is inasmuch completely different from the cited work -- both phenomenologically and methodologically --
as (1) we consider an initial-boundary value problem in a bounded domain $\Omega$
instead of the full space $\mathbb{R}^{d}$, where the latter is amenable to classical differential calculus and (2)
only impose a genuine smallness condition on the lower energy of the initial data [corresponding to the topology of weak solution],
in contrast to the smallness of the highest-order energy assumed in \cite{RaUe2017}.

Our goal is to show that the resulting nonlinear system generates well-posed dynamics for `arbitrary' regular data satisfying compatibility conditions
and that, for small data, the dynamics is global provided the size of initial data is well calibrated. 
The challenge is, of course, to overcome difficulties related to compromised regularity of linear solutions and compromised dissipation in the presence of highly nonlinear internal force represented by a severely unbounded operator which nonlinearly depends on the principal part of the elliptic operator. 
This compels one to perform the analysis at a high topological level with appropriate mechanisms for the propagation of restricted  dissipation. 
This challenge  manifests itself on both levels: local and global. 
In particular, when carrying out the Kato-iteration in Appendix Section \ref{APPENDIX_SECTION_WELL_POSEDNESS},
in contrast to the parabolic-hyperbolic case (viz. \cite{LaPoWa2017}),
elastic Equation (\ref{EQUATION_QUASILINEAR_PDE_IN_W_THETA_AND_Q_1}) cannot be decoupled from thermal
Equations (\ref{EQUATION_QUASILINEAR_PDE_IN_W_THETA_AND_Q_2})--(\ref{EQUATION_QUASILINEAR_PDE_IN_W_THETA_AND_Q_3}).
As for the global stabilizability estimate, the main difficulty arises from the fact
that the dissipation is only available in Equation (\ref{EQUATION_QUASILINEAR_PDE_IN_W_THETA_AND_Q_3}) for $\mathbf{q}$ and as such does not propagate any regularity.
Therefore, suitable observability-type estimates become essential to reconstruct the integrals of potential and kinetic energy for the energies of $w$ and $\theta$ by propagating 
dissipation from the heat flux to higher-energy level.

Another important feature of the present paper is the smallness argument employed in our proofs.
Provided the nonlinear coefficient in the ``elliptic part'' stays positive,
apart from the aforementioned indispensable compatibility and regularity conditions,
the initial data are assumed small merely in the lowest topology associated with ``finite energy'' or mild solutions.
This is an important improvement as most quasilinear results assume the smallness of the data in the highest topology.
From the technical point of view, this makes the stability proof more challenging  and requires an extra degree of diligence
as the basic-level energy needs to be carefully traced and properly ``factored out'' in our nonlinear estimates using suitable interpolation procedures, etc.
The final argument for the ``globality'' depends on two coupled and cooperating ``barrier inequalities,'' rather than a single one -- as is the case in the usual quasilinear theory.

Last but not least, a further contribution of this paper is a physical derivation of the thermoelastic plate model (\ref{EQUATION_QUASILINEAR_PDE_IN_W_THETA_AND_Q_1})--(\ref{EQUATION_IC_IN_W_THETA_AND_Q}). While the macroscopic description of Kirchhoff \& Love plates with geometric nonlinearity \cite[Chapter 1]{LaLi1988} and nonlinear material response \cite{LaPoWa2017} is known in the literature for the case of Fourier's heat conduction, to the best of authors' knowledge, no rigorous Kirchhoff \& Love thermoelastic plate models with Cattaneo's heat conduction have been available in the literature up to date. (Parenthetically, one should mention the Reissner--Mindlin--Timoshenko plate with Cattaneo's heat conduction and the geometric nonlinearity derived in the thesis \cite{Po2011}.)
Instead, previous works on thermoelastic Kirchhoff \& Love plates with Cattaneo's law (viz. \cite{RaUe2016, RaUe2017, RiRaSeVi2018}, etc.) have implicitly `conjectured' the physical model.
However, the thermal moment $\theta$ and the (planar) heat flux moment $\mathbf{q}$ were invariably misspecified as the temperature and the heat flux, respectively, and the natural extra damping $\sigma \theta$ in Equation (\ref{EQUATION_QUASILINEAR_PDE_IN_W_THETA_AND_Q_2}) was overlooked. In this paper, we close this gap by combining various results on related plate systems into a consistent physical model behind Equations (\ref{EQUATION_QUASILINEAR_PDE_IN_W_THETA_AND_Q_1})--(\ref{EQUATION_IC_IN_W_THETA_AND_Q}).

To close the introduction, we mention several open problems which naturally arise. 
In addition to simply supported boundary conditions, one would like to have a theory for clamped and free boundary conditions. Particularly, the latter are challenging due to the fact that harmonic functions are not controlled by ``free'' boundary conditions imposed on the biharmonic operator. This difficulty can be overcome, while significantly increasing the level of technicality, by localizing the problem \cite{DeShi2017, LaTri1998.3}.

Another open problem is how a boundary feedback can be used as the only source of dissipation \cite{LaOng,La}. This, again, leads to a challenging problem of propagation of dissipation from the boundary -- a  technique developed in control theory and dependent on the rays of geodesic optics.

Finally, we mention that Kirchhoff-type equations have been recently considered  with fractional and, possibly, degenerate Laplacians \cite{pucci1, pucci}. 
Global solutions for small data and blow-up of solutions for data outside of the potential well have been recently established in \cite{pucci}. The arguments involved rely on nonlocal elliptic theory. It would be interesting to consider such models within the framework of thermoelasticity.  

The rest of the paper is structured as follows.
Following the present Introduction Section \ref{SECTION_INTRODUCTION},
Section \ref{SECTION_MAIN_RESULTS} summarizes all of the main results of the article on local and global well-posednesss as well as exponential stability of Equations 
(\ref{EQUATION_QUASILINEAR_PDE_IN_W_THETA_AND_Q_1})--(\ref{EQUATION_IC_IN_W_THETA_AND_Q}).
In Section \ref{SECTION_EQUIVALENT_TRANFORMATION}, the system (\ref{EQUATION_QUASILINEAR_PDE_IN_W_THETA_AND_Q_1})--(\ref{EQUATION_IC_IN_W_THETA_AND_Q})
is reduced to an equivalent non-vectorial second-order system.
Subsequently, in Section \ref{SECTION_LOCAL_WELL_POSEDNESS}, a local well-posedness result is established
by applying a fixed-point argument to a linearization of the equivalent reduced system from Section \ref{SECTION_EQUIVALENT_TRANFORMATION}.
In Section \ref{SECTION_LONG_TIME_BEHAVIOR}, the unique classical local solution is extended globally -- provided the initial data are sufficiently small at the basic energy level --
and an exponential decay rate is further proved.
Finally, in the appendix, a brief physical derivation of Equations (\ref{EQUATION_QUASILINEAR_PDE_IN_W_THETA_AND_Q_1})--(\ref{EQUATION_IC_IN_W_THETA_AND_Q}) is presented in Section \ref{SECTION_MODEL_DESCRIPTION}
while Section \ref{APPENDIX_SECTION_WELL_POSEDNESS} establishes a solution theory for the linearized version of the latter equations with time- and space-dependent coefficients. This furnishes a powerful auxiliary machinery for the development of the nonlinear local theory.

\section{Main Results}
\label{SECTION_MAIN_RESULTS}

In this Section, we summarize all of the central results of this paper on the quasilinear plate equations 
(\ref{EQUATION_QUASILINEAR_PDE_IN_W_THETA_AND_Q_1})--(\ref{EQUATION_IC_IN_W_THETA_AND_Q}).
In the following, we assume $\Omega \subset \mathbb{R}^{d}$, $d \in \{1, 2, 3\}$, is a bounded, smooth domain.
As previously announced in the Introduction Section \ref{SECTION_INTRODUCTION},
the constant $\sigma$ will be assumed zero throughout the rest of the paper.
All of the results stated below trivially remain true for $\sigma$
as the $\sigma \theta$-term is a Lipschitzian perturbation and has the correct sign adding even more damping to the system.

\begin{definition}
	\label{DEFINITION_CLASSICAL_SOLUTION_W_THETA_Q}
	Let $s \geq 2$. A classical solution to Equations (\ref{EQUATION_QUASILINEAR_PDE_IN_W_THETA_AND_Q_1})--(\ref{EQUATION_IC_IN_W_THETA_AND_Q}) on $[0, T]$ at the energy level $s$
	is a triple $(w, \theta, \mathbf{q}) \colon [0, T] \times \bar{\Omega} \to \mathbb{R} \times \mathbb{R} \times \mathbb{R}^{d}$ with
	\begin{align*}
		w, \triangle w &\in \Big(\bigcap_{m = 0}^{s - 1} C^{m}\big([0, T], H^{s - m}(\Omega) \cap H^{1}_{0}(\Omega)\big)\Big) \cap C^{s}\big([0, T], L^{2}(\Omega)\big), \\
		\theta &\in \Big(\bigcap_{m = 0}^{s - 1} C^{m}\big([0, T], H^{s - m}(\Omega) \cap H^{1}_{0}(\Omega)\big)\Big), \quad
		\mathbf{q} \in \Big(\bigcap_{m = 0}^{s - 1} C^{m}\big([0, T], \big(H^{s - m}(\Omega)\big)^{d}\big)\Big)
	\end{align*}
	which, being plugged into Equations (\ref{EQUATION_QUASILINEAR_PDE_IN_W_THETA_AND_Q_1})--(\ref{EQUATION_IC_IN_W_THETA_AND_Q}), renders them tautological.
	Classical solutions on $[0, T)$ and $[0, \infty)$ are defined correspondingly.
\end{definition}

The choice $s = 2$ in Definition \ref{DEFINITION_CLASSICAL_SOLUTION_W_THETA_Q}
is standard in the linear situation, i.e., when $K(\cdot)$ is linear.
In this case, by virtue of the standard semigroup theory,
for any initial data $(w^{0}, w^{1}, \theta^{0}, \mathbf{q}) \in \big(H^{4}(\Omega) \cap H^{1}_{0}(\Omega)\big) \times 
\big(H^{2}(\Omega) \cap H^{1}_{0}(\Omega)\big) \times H^{1}_{0}(\Omega) \times \big(H^{1}(\Omega)\big)^{d}$ with $\triangle w^{0} \in H^{1}_{0}(\Omega)$
there exists a unique classical solution at the energy level $s = 2$.
In contrast, if $K(\cdot)$ is genuinely nonlinear, in general, one cannot expect a classical solution 
for the initial data at the energy level $s = 2$ (cf. \cite[Remark 14.4]{Ka1985}).
Therefore, moving to higher energy levels is unavoidable to obtain classical solutions in the general nonlinear case.

In this paper, we prove the global well-posedness and exponential stability of classic solutions for $s \geq 3$. In particular, when $s=3$, the solution space in Definition \ref{DEFINITION_CLASSICAL_SOLUTION_W_THETA_Q} rewrites as
\begin{align*}
		w, \triangle w &\in \Big(\bigcap_{m = 0}^{2} C^{m}\big([0, T], H^{3 - m}(\Omega) \cap H^{1}_{0}(\Omega)\big)\Big) \cap C^{3}\big([0, T], L^{2}(\Omega)\big), \\
		\theta &\in \Big(\bigcap_{m = 0}^{2} C^{m}\big([0, T], H^{3 - m}(\Omega) \cap H^{1}_{0}(\Omega)\big)\Big), \quad
		\mathbf{q} \in \Big(\bigcap_{m = 0}^{2} C^{m}\big([0, T], \big(H^{3 - m}(\Omega)\big)^{d}\big)\Big).
\end{align*}

As usual in quasilinear theory, the presence of nonlinearity not only amounts to putting additional Sobolev regularity assumptions on the initial data and smoothness conditions on $K(\cdot)$,
but requires suitable `compatibility conditions' described below.

Given a classical solution to Equations (\ref{EQUATION_QUASILINEAR_PDE_IN_W_THETA_AND_Q_1})--(\ref{EQUATION_IC_IN_W_THETA_AND_Q}) at an energy level $s \geq 2$,
by applying the $\partial_{t}^{m}$-operator, $m = 0, \dots, s - 2$, we obtain the compatibility conditions
\begin{equation}
    \label{EQUATION_NECESSARY_CONDITION_REGULARITY_AT_ZERO}
    \begin{split}
        \partial_{t}^{m} w(0, \cdot), \triangle \partial_{t}^{m} w(0, \cdot) &\in H^{s - m}(\Omega) \cap H^{1}_{0}(\Omega), \quad
        \partial_{t}^{s} w(0, \cdot) \in H^{2}(\Omega) \cap H^{1}_{0}(\Omega), \\
        \partial_{t}^{m} \theta(0, \cdot) &\in H^{s - m}(\Omega) \cap H^{1}_{0}(\Omega) \quad \text{ and } \quad \partial_{t}^{m} \mathbf{q}(0, \cdot) \in \big(H^{s - m}(\Omega)\big)^{d}
    \end{split}
\end{equation}
for $m = 0, \dots, s - 1$. Although the solution is {\it a priori} unknown,
$\partial_{t}^{m} w(0, \cdot)$, $m = 2, \dots, s$, $\partial_{t}^{l} \theta(0, \cdot)$ and $\partial_{t}^{l} \mathbf{q}(0, \cdot)$, $l = 1, \dots, s - 1$,
can iteratively be computed from $w^{0}, w^{1}, \theta^{0}, \mathbf{q}^{0}$ using a procedure outlined below.

To this end, let
\begin{equation}
    \label{EQUATION_OPERATOR_A_DEFINITION}
	A \colon D(A) \subset L^{2}(\Omega) \to L^{2}(\Omega), \quad u \mapsto -\triangle u
	\text{ with } D(A) := \big\{u \in H^{1}_{0}(\Omega) \,|\, \triangle u \in L^{2}(\Omega)\big\}
\end{equation}
denote the $L^{2}(\Omega)$-realization of the negative Dirichlet-Laplacian.
If $\partial \Omega \in C^{2}$, the standard elliptic theory suggests
$D(A) = H^{2}(\Omega) \cap H^{1}_{0}(\Omega)$ with $A$ being an isomorphism between $D(A)$ and $L^{2}(\Omega)$.
Similarly, if $\partial \Omega \in C^{s}$ for some $s \geq 2$, the operator $A$ can be viewed as an isomorphism between $H^{s}(\Omega) \cap H^{1}_{0}(\Omega)$ and $H^{s - 2}(\Omega)$.
Here and in the sequel, we use the notation $H^{0}_{0}(\Omega) \equiv H^{0}(\Omega) := L^{2}(\Omega)$.

With this notation, Equations (\ref{EQUATION_QUASILINEAR_PDE_IN_W_THETA_AND_Q_1})--(\ref{EQUATION_QUASILINEAR_PDE_IN_W_THETA_AND_Q_3}) can be cast into the equivalent form:
\begin{align}
	\label{EQUATION_QUASILINEAR_PDE_IN_W_THETA_AND_Q_EQUIVALENT_1}
	A (\gamma + A^{-1}) w_{tt} + A K(\triangle w) - \alpha A \theta &= 0 &\text{ in } (0, \infty) \times \Omega, \\
	\label{EQUATION_QUASILINEAR_PDE_IN_W_THETA_AND_Q_EQUIVALENT_2}
	\beta \theta_{t} + \DIV \mathbf{q} + \alpha A w_{t} &= 0 &\text{ in } (0, \infty) \times \Omega, \\
	\label{EQUATION_QUASILINEAR_PDE_IN_W_THETA_AND_Q_EQUIVALENT_3}
	\tau \mathbf{q}_{t} + \mathbf{q} + \eta \nabla \theta &= \mathbf{0} &\text{ in } (0, \infty) \times \Omega.
\end{align}
Provided $K(\cdot)$ is sufficiently smooth,
by sequentially applying the $\partial_{t}$-operator to Equations (\ref{EQUATION_QUASILINEAR_PDE_IN_W_THETA_AND_Q_EQUIVALENT_1})--(\ref{EQUATION_QUASILINEAR_PDE_IN_W_THETA_AND_Q_EQUIVALENT_3}),
using the product rule, Fa\`{a} di Bruno's formula and exploiting the invertibility of $(\gamma + A^{-1})$,
for any $m \geq 1$, $w^{m + 1}, \theta^{m}, \mathbf{q}^{m}$ can be expressed via 
$w^{0}, \dots, w^{m}$, $\theta^{0}, \dots, \theta^{m - 1}$, $\mathbf{q}^{0}, \dots, \mathbf{q}^{m - 1}$.
Indeed, evaluating Equations (\ref{EQUATION_QUASILINEAR_PDE_IN_W_THETA_AND_Q_EQUIVALENT_1})--(\ref{EQUATION_QUASILINEAR_PDE_IN_W_THETA_AND_Q_EQUIVALENT_3}) at $t = 0$
and applying $A^{-1}$ to Equation (\ref{EQUATION_QUASILINEAR_PDE_IN_W_THETA_AND_Q_EQUIVALENT_1}), we get
\begin{align*}
    w^{2} &= -\big(\gamma + A^{-1}\big)^{-1} \big(K(\triangle w^{0}) - \alpha \theta^{0}\big), \\
	\theta^{1} &= -\tfrac{1}{\beta} \big(\operatorname{div} \mathbf{q}^{0} + \alpha A w^{1}\big) \quad \text{ and } \quad
	\mathbf{q}^{1} = -\tfrac{1}{\tau} \big(\mathbf{q} + \eta \nabla \theta\big)
\end{align*}
expressing $w^{2}, \theta^{1}, \mathbf{q}^{1}$ in terms of $w^{0}, w^{1}, \theta^{0}, \mathbf{q}^{0}$.
Similarly, for $m = 2, \dots, s$, applying the $\partial_{t}^{m - 2}$-operator, we get
\begin{align*}
    w^{m} &= -\big(\gamma + A^{-1}\big)^{-1} \Big(\partial_{t}^{m - 2} \big(K(\triangle w)\big)\Big)\Big|_{t = 0} - \alpha \theta^{m - 2}, \\
	\theta^{m - 1} &= -\tfrac{1}{\beta} \big(\operatorname{div} \mathbf{q}^{m - 2} + \alpha A w^{m - 1}\big) \quad \text{ and } \quad
	\mathbf{q}^{m - 1} = -\tfrac{1}{\tau}\big(\mathbf{q}^{m - 2} + \eta \nabla \theta^{m - 2}\big).
\end{align*}
Thus, by virtue of the product rule and Fa\`{a} di Bruno's formula,
the right-hand sides can be expressed in terms of $w^{0}, \dots, w^{m - 1}$, $\theta^{0}, \dots, \theta^{m - 2}$ and $\mathbf{q}^{0}, \dots, \mathbf{q}^{m - 2}$.
This construction can easily be made rigorous using an induction procedure starting at $m = 2$.

\begin{definition}
	Let $w^{m}$, $\theta^{m}$, $\mathbf{q}^{m}$, $m \geq 0$, denote the `initial values'
	for $\partial_{t}^{m} w$, $\partial_{t}^{m} \theta$ and $\partial_{t}^{m} \mathbf{q}$ as described above (See also \cite[p. 96]{JiaRa2000}).
\end{definition}

Suppose $w$ is smooth. Then, we can write:
\begin{equation}
    \label{EQUATION_TRIANGLE_OF_K_TRIANGLE_W}
	\triangle K(\triangle w) = K'(\triangle w) \triangle^{2} w + K''(\triangle w) |\nabla \triangle w|^{2}.
\end{equation}
Hence, the sign of $K'(\cdot)$ decides the positive ellipticity of $-\triangle K(\triangle \cdot)$. Further details and explanations will be presented in the sections to follow.

\begin{assumption}
	\label{ASSUMPTION_LOCAL_EXISTENCE}
	Let $s \geq \lfloor \tfrac{d}{2}\rfloor + 2$ be an integer and let $\Omega \subset \mathbb{R}^{d}$ be a bounded domain with $\partial \Omega \in C^{s}$.
	Here, the floor function $\lfloor x\rfloor$ denotes the integer part of $x$,
	i.e., the largest integer not exceeding $x$.
	\begin{enumerate}
        \item Let $K \in C^{s + 1}(\mathbb{R}, \mathbb{R})$.

		\item Let the initial data satisfy the regularity and compatibility conditions
		\begin{equation}
			\begin{split}
				w^{m}, \triangle w^{m} &\in H^{s - m}(\Omega) \cap H^{1}_{0}(\Omega) \text{ for } m = 0, \dots, s - 1, \quad w^{s} \in H^{2}(\Omega) \cap H^{1}_{0}(\Omega) \quad \text{ and } \\
				\theta^{k} &\in H^{s - k}(\Omega) \cap H^{1}_{0}(\Omega), \quad
				\mathbf{q}^{k} \in \big(H^{s - k}(\Omega)\big)^{d} \text{ for } k = 0, \dots, s - 1,
			\end{split}
			\notag
		\end{equation}
		where $H^{0}(\Omega) := L^{2}(\Omega)$.
		
		\item For the ``initial'' (positive) ellipticity of $K'(\triangle w^{0}) \triangle$, suppose
		\begin{equation}
			\min_{x \in \bar{\Omega}} K'\big(\triangle w^{0}(x)\big) > 0, \quad
			\text{where } \triangle w^{0} \in C^{0}(\bar{\Omega}) \text{ by virtue of Sobolev's embedding theorem.} \notag
		\end{equation}
	\end{enumerate}
\end{assumption}

\begin{theorem}[Local existence \& uniqueness]
	\label{THEOREM_LOCAL_EXISTENCE}
	
	Suppose Assumption \ref{ASSUMPTION_LOCAL_EXISTENCE} is satisfied for some $s \geq \lfloor \tfrac{d}{2}\rfloor + 2$.
    Then, Equations (\ref{EQUATION_QUASILINEAR_PDE_IN_W_THETA_AND_Q_1})--(\ref{EQUATION_IC_IN_W_THETA_AND_Q})
    possess a unique classical solution $(w, \theta, \mathbf{q})$ at the energy level $s$ on a maximal interval $[0, T_{\mathrm{max}})$ (possibly, small, but not empty) such that:
	\begin{enumerate}
        \item ``Local non-degeneracy:'' $\min\limits_{x \in \bar{\Omega}} K'\big(\triangle w(t, x)\big) > 0$ for any $t \in [0, T_{\mathrm{max}})$.
        
        \item ``Blow-up or eventual degeneracy if solution non-global:''
        Unless $T_{\mathrm{max}} = \infty$, either the ellipticity condition is eventually violated
        \begin{equation}
            \label{EQUATION_ELLIPTICITY_VIOLATION_AT_T_MAX}
            \min\limits_{x \in \bar{\Omega}} K'\big(\triangle w(t, x)\big) \to 0 \quad \text{ as } \quad t \nearrow T_{\mathrm{max}}
        \end{equation}
        or/and the blow-up occurs
        \begin{equation}
            \label{EQUATION_BLOW_UP_AT_T_MAX}
            \big\|\triangle w(t, \cdot)\big\|_{H^{s}(\Omega)}^{2} \to \infty \quad \text{ as } \quad t \nearrow T_{\mathrm{max}}.
        \end{equation}

        \item ``Solution map continuity:''
        For any $T > 0$, $\varepsilon > 0$ and $N > 0$, the solution mapping $(w^{0}, w^{1}, \theta^{0}, \mathbf{q}^{0}) \mapsto (w, w_{t}, \theta, \mathbf{q})$
        is a continuous function from
        \begin{align*}
            \mathscr{M}_{T, \varepsilon, N} :=
            \Big\{(w^{0}, w^{1}, \theta^{0}, \mathbf{q}^{0}) \,\big|\, &(w^{0}, w^{1}, \theta^{0}, \mathbf{q}^{0}) \text{ satisfy Assumption \ref{ASSUMPTION_LOCAL_EXISTENCE} and admit} \\
            &\text{a classical solution }  (w, \theta, \mathbf{q}) \text{ with } \min_{t \in [0, T]} \min_{x \in \bar{\Omega}} K'\big(\triangle w(t, x)\big) \geq \varepsilon, \\
            &\max_{0 \leq t \leq T} \sum_{m = 0}^{s} \big\|\partial_{t}^{m} w(t, \cdot)\big\|_{H^{s + 2 - m}(\Omega)}^{2} \leq N^{2}\Big\}
        \end{align*}
        endowed with the topology of $H^{3}(\Omega) \times H^{2}(\Omega) \times H^{1}(\Omega) \times \big(H^{1}(\Omega))^{d}\big)$
        to $L^{\infty}\big(0, T; H^{3}(\Omega) \times H^{2}(\Omega) \times H^{1}(\Omega) \times \big(H^{1}(\Omega))^{d}\big)\big)$.
        Note that the set $\mathscr{M}_{T, \varepsilon, N}$ is non-empty for small $\varepsilon$,
        large $N$ and small $T$.
	\end{enumerate}
\end{theorem}

This choice of $s = \lfloor \tfrac{d}{2}\rfloor + 2$ is known to be optimal for quasilinear wave-equation-like problems. 
For a more detailed discussion, we refer to \cite[Remark 4.2]{LaPoWa2017}.

For the sake of simplicity, we now assume $\Omega \subset \mathbb{R}^{d}$, $d = 2, 3$, 
and establish global existence and uniqueness of classical solutions at the energy level $s = \lfloor \tfrac{d}{2}\rfloor + 2 \equiv 3$.
In addition to Assumption \ref{ASSUMPTION_LOCAL_EXISTENCE}, we require:
\begin{assumption}
	\label{ASSUMPTION_GLOBAL_EXISTENCE}
	
	Let $K(\cdot)$ satisfy $K(0) = 0$, $K'(0) > 0$, $K''(0) = 0$.
	
%
\end{assumption}
For instance, for any real number $\alpha$, the function $K(z) = z + \alpha z^3$ from \cite{LaMaSa2008, LaPoWa2017} satisfies Assumption \ref{ASSUMPTION_GLOBAL_EXISTENCE}. 
Note that the condition $K(0) = 0$ is mathematically redundant, but is fulfilled by real-world material responses for physical reasons discussed in Appendix Section \ref{SECTION_MODEL_DESCRIPTION}.

By continuity, Assumption \ref{ASSUMPTION_GLOBAL_EXISTENCE} furnishes the existence of a number $\rho > 0$ such that
\begin{equation}
    \label{EQUATION_LOCAL_POSITIVITY_OF_K_PRIME}
    K'(z) > 0 \quad \text{ for } \quad |z| < \rho.
\end{equation}


\begin{theorem}[Global well-posedness] 
	\label{THEOREM_GLOBAL_EXISTENCE}
	
	Let Assumptions \ref{ASSUMPTION_LOCAL_EXISTENCE} and \ref{ASSUMPTION_GLOBAL_EXISTENCE} be satisfied for $d \in \{2, 3\}$ and $s = 3$.
	Then, for any number $M>0$, there exists a (small) number $\delta_{M,\rho} > 0$, depending on $M$ and $\rho$, such that for any initial data $(w^{0}, w^{1}, \theta^{0}, \mathbf{q}^{0})$ satisfying 
	\begin{align}
	  \|\triangle w^0\|_{L^{\infty}(\Omega)} &< \rho,      \label{Global_small_laplacian_w_0}   \\
         X_0 := \sum_{m = 0}^{3} \|w^{m}\|_{H^{5 - m}(\Omega)}^{2} + \sum_{m = 0}^{2} \|\theta^{m}\|_{H^{3 - m}(\Omega)}^{2} 
        + \sum_{m = 0}^{2} \|\mathbf{q}^{m}\|_{(H^{3 - m}(\Omega))^{d}}^{2} &< M^2,
        \label{EQUATION_GLOBAL_THEOREM_HIGEST_TOPOLOGY_NORM}
\\       
        \|w^{0}\|_{H^{3}(\Omega)}^{2} + \|w^{1}\|_{H^{2}(\Omega)}^{2} + \|\theta^{0}\|_{H^{1}(\Omega)}^{2}+
        \|\operatorname{div} \mathbf{q}^{0}\|_{L^{2}(\Omega)}^{2} &< \delta_{M,\rho}^{2}
         \label{EQUATION_GLOBAL_THEOREM_LOWER_TOPOLOGY_NORM}         
    \end{align}
    the unique local solution $(w, \theta, \mathbf{q})$ to Equations
	(\ref{EQUATION_QUASILINEAR_PDE_IN_W_THETA_AND_Q_1})--(\ref{EQUATION_IC_IN_W_THETA_AND_Q}) given in Theorem \ref{THEOREM_LOCAL_EXISTENCE} exists globally, 
	i.e., $T_{\mathrm{max}} = \infty$. 
\end{theorem}
    
\begin{remark}
    The boundedness assumption for $\|\triangle w^{0}\|_{L^{\infty}(\Omega)}$ formulated in Equation \eqref{Global_small_laplacian_w_0} above is natural in light of the initial positive ellipticity condition in Assumption \ref{ASSUMPTION_LOCAL_EXISTENCE} as well as Equation \eqref{EQUATION_LOCAL_POSITIVITY_OF_K_PRIME}. However, it can be eliminated if the function $K'(\cdot)$ is positive everywhere in $\mathbb{R}$, and consequently the choice of $\delta$ only depends on $M$.
    The latter assumption $K'(\cdot)>0$ is physically sound and is commonly employed in the Theory of Finite Elasticity when material fracture phenomena are ignored.
    In this case, we do not need the smallness of $\|\triangle w^{0}\|_{L^{\infty}(\Omega)}$ and we obtain a `large-data' result provided the function $K(\cdot)$ satisfies appropriate ``growth conditions'' at infinity (cf. Equation \eqref{F'_bound}--\eqref{F''''_bound}.
\end{remark}

As a `by-product,' we will obtain our main stabilization result:
\begin{theorem}[Exponential stability] 
	\label{THEOREM_STABILITY}
	Under the conditions of Theorem \ref{THEOREM_GLOBAL_EXISTENCE},
	there exist positive constants $C$ and $\kappa$ such that
	\begin{align*}
        \sum_{m = 0}^{s} &\big\|\partial_{t}^{m} w(t, \cdot)\big\|_{H^{s + 2 - m}(\Omega)}^{2} +
        \sum_{m = 0}^{s - 1} \big\|\partial_{t}^{m} \theta(t, \cdot)\big\|_{H^{s - m}(\Omega)}^{2} + 
        \sum_{m = 0}^{s - 1} \big\|\partial_{t}^{m} \mathbf{q}(t, \cdot)\big\|_{(H^{s - m}(\Omega))^{d}}^{2}
        \leq C e^{-\kappa t} X_0^{3/2}(t)
	\end{align*}
	for $t \geq 0$, where the initial energy $X_0$ is defined in Equation \eqref{EQUATION_GLOBAL_THEOREM_HIGEST_TOPOLOGY_NORM}.
\end{theorem}

\section{Equivalent Transformation}
\label{SECTION_EQUIVALENT_TRANFORMATION}

To facilitate the analytic treatment of Equations (\ref{EQUATION_QUASILINEAR_PDE_IN_W_THETA_AND_Q_1})--(\ref{EQUATION_IC_IN_W_THETA_AND_Q}),
we first reduce Equations (\ref{EQUATION_QUASILINEAR_PDE_IN_W_THETA_AND_Q_1})--(\ref{EQUATION_IC_IN_W_THETA_AND_Q}) to an equivalent lower-order non-vectorial system.
Exploiting the operator $A$ defined in Equation (\ref{EQUATION_OPERATOR_A_DEFINITION}), introduce the new variables
\begin{align}
	z := A w \quad \text{ and } \quad p := \DIV \mathbf{q},
	\label{EQUATION_DEFINITION_OF_Z_AND_P}
\end{align}
Applying the $\operatorname{div}$-operator to Equation (\ref{EQUATION_QUASILINEAR_PDE_IN_W_THETA_AND_Q_3}),
the system (\ref{EQUATION_QUASILINEAR_PDE_IN_W_THETA_AND_Q_1})--(\ref{EQUATION_IC_IN_W_THETA_AND_Q}) is reduced to
\begin{align}
	\label{EQUATION_QUASILINEAR_PDE_IN_Z_THETA_AND_P_1}
	\big(A^{-1} + \gamma) z_{tt} - A K(-z) - \alpha A \theta &= 0 &\text{ in } (0, \infty) \times \Omega, \\
	\label{EQUATION_QUASILINEAR_PDE_IN_Z_THETA_AND_P_2}
	\beta \theta_{t} + p + \alpha z_{t} &= 0 &\text{ in } (0, \infty) \times \Omega, \\
	\label{EQUATION_QUASILINEAR_PDE_IN_Z_THETA_AND_P_3}
	\tau p_{t} + p - \eta A \theta &= 0 &\text{ in } (0, \infty) \times \Omega\phantom{,}
\end{align}
subject to homogeneous Dirichlet-Dirichlet boundary conditions
\begin{align}
	\label{EQUATION_BC_IN_Z_THETA_AND_P}
	z &= \theta = 0 &\text{ on } (0, \infty) \times \partial \Omega
\end{align}
and initial conditions
\begin{align}
	\label{EQUATION_IC_IN_Z_THETA_AND_P}
	z(0, \cdot) &= z^{0}, \quad z_{t}(0, \cdot) = z^{1}, \quad
	\theta(0, \cdot) = \theta^{0}, \quad p(0, \cdot) = p^{0} &\text{ in } \Omega
\end{align}
with $z^{0} := A w^{0}$, $z^{1} := A w^{1}$ and $p^{0} := \DIV \mathbf{q}^{0}$.
We want to show the original system (\ref{EQUATION_QUASILINEAR_PDE_IN_W_THETA_AND_Q_1})--(\ref{EQUATION_IC_IN_W_THETA_AND_Q})
and the reduced system (\ref{EQUATION_QUASILINEAR_PDE_IN_Z_THETA_AND_P_1})--(\ref{EQUATION_IC_IN_Z_THETA_AND_P}) are equivalent in appropriate solution classes.
The uniqueness of solutions (without being indispensable) will simplify our arguments.
Similar to Definition \ref{DEFINITION_CLASSICAL_SOLUTION_W_THETA_Q} for the original system (\ref{EQUATION_QUASILINEAR_PDE_IN_W_THETA_AND_Q_1})--(\ref{EQUATION_IC_IN_W_THETA_AND_Q}),
for the reduced system, we have:
\begin{definition}
	\label{DEFINITION_CLASSICAL_SOLUTION_Z_THETA_P}
	Let $s \geq 2$. Under a classical solution
	to Equations (\ref{EQUATION_QUASILINEAR_PDE_IN_Z_THETA_AND_P_1})--(\ref{EQUATION_IC_IN_Z_THETA_AND_P}) on $[0, T]$ at the energy level $s$,
	we understand a function triple $(z, \theta, p) \colon [0, T] \times \bar{\Omega} \to \mathbb{R} \times \mathbb{R} \times \mathbb{R}$ satisfying
	\begin{align*}
		z      &\in \Big(\bigcap_{m = 0}^{s - 1} C^{m}\big([0, T], H^{s - m}(\Omega) \cap H^{1}_{0}(\Omega)\big)\Big) \cap C^{s}\big([0, T], L^{2}(\Omega)\big), \\
		\theta &\in \Big(\bigcap_{m = 0}^{s - 1} C^{m}\big([0, T], H^{s - m}(\Omega) \cap H^{1}_{0}(\Omega)\big)\Big), \quad
		p \in \Big(\bigcap_{m = 0}^{s - 1} C^{m}\big([0, T], H^{s - 1 - m}(\Omega)\big)\Big)
	\end{align*}
	and, being plugged into Equations (\ref{EQUATION_QUASILINEAR_PDE_IN_Z_THETA_AND_P_1})--(\ref{EQUATION_IC_IN_Z_THETA_AND_P}), turns them into tautology.
	Classical solutions on $[0, T)$ and $[0, \infty)$ are defined correspondingly.
\end{definition}

\begin{theorem}
	\label{THEOREM_SYSTEM_EQUIVALENCE}

	Let $s \geq \lfloor \tfrac{d}{2}\rfloor + 2$. A triple $(w, \theta, \mathbf{q})$ is a classical solution (\ref{EQUATION_QUASILINEAR_PDE_IN_W_THETA_AND_Q_1})--(\ref{EQUATION_IC_IN_W_THETA_AND_Q})
	if and only if $(z, \theta, p)$ defined in Equation (\ref{EQUATION_DEFINITION_OF_Z_AND_P}) 
	is a classical solution to Equations (\ref{EQUATION_QUASILINEAR_PDE_IN_Z_THETA_AND_P_1})--(\ref{EQUATION_IC_IN_Z_THETA_AND_P}).
	Conversely, $(z, \theta, p)$ is a classical solution to Equations (\ref{EQUATION_QUASILINEAR_PDE_IN_Z_THETA_AND_P_1})--(\ref{EQUATION_IC_IN_Z_THETA_AND_P}) if and only if
	\begin{equation}
		\notag
		w(t, \cdot) = A^{-1} z(t, \cdot) \quad \text{ and } \quad
		\mathbf{q}(t, \cdot) = \mathbf{q}^{0} + \nabla A^{-1} \operatorname{div} \mathbf{q}^{0} - \nabla A^{-1} p(t, \cdot)
	\end{equation}
	is a classical solution to Equations (\ref{EQUATION_QUASILINEAR_PDE_IN_W_THETA_AND_Q_1})--(\ref{EQUATION_IC_IN_W_THETA_AND_Q}) at the same energy level.
\end{theorem}

\begin{proof}
    On the strength of Theorem \ref{THEOREM_LOCAL_EXISTENCE} from Section \ref{SECTION_MAIN_RESULTS} below, classical solutions to the reduced system
    (\ref{EQUATION_QUASILINEAR_PDE_IN_Z_THETA_AND_P_1})--(\ref{EQUATION_IC_IN_Z_THETA_AND_P}) are unique. 
    To prove the same property holds true for the original system (\ref{EQUATION_QUASILINEAR_PDE_IN_W_THETA_AND_Q_1})--(\ref{EQUATION_IC_IN_W_THETA_AND_Q}),
    suppose $(w, \theta, \mathbf{q})$, $(\tilde{w}, \tilde{\theta}, \tilde{\mathbf{q}})$ are two classical solutions to 
    (\ref{EQUATION_QUASILINEAR_PDE_IN_W_THETA_AND_Q_1})--(\ref{EQUATION_IC_IN_W_THETA_AND_Q}).
    Letting $z := A w$, $\tilde{z} := A \tilde{w}$, $p := \operatorname{div} \mathbf{q}$ and $\tilde{p} := \operatorname{div} \tilde{\mathbf{q}}$,
    we easily conclude $(z, \theta, p)$ and $(\tilde{z}, \tilde{\theta}, \tilde{p})$ solve Equations (\ref{EQUATION_QUASILINEAR_PDE_IN_Z_THETA_AND_P_1})--(\ref{EQUATION_IC_IN_Z_THETA_AND_P})
    and, thus, must coincide. We prove $(w, \theta, \mathbf{q}) \equiv (\tilde{w}, \tilde{\theta}, \tilde{\mathbf{q}})$.
    Since $A$ is invertible, $z \equiv \tilde{z}$ implies $w \equiv \tilde{w}$.
    Since $\theta \equiv \tilde{\theta}$, Equation (\ref{EQUATION_QUASILINEAR_PDE_IN_W_THETA_AND_Q_3}) for 
    the solution difference $\bar{\mathbf{q}} := \mathbf{q} - \tilde{\mathbf{q}}$ yields
    \begin{equation*}
        \tau \partial_{t} \bar{\mathbf{q}} + \bar{\mathbf{q}} = \mathbf{0}, \quad \bar{\mathbf{q}}(0, \cdot) \equiv \mathbf{0}.
    \end{equation*}
    The latter ODE, being uniquely solvable by $\bar{\mathbf{q}} \equiv 0$,
    suggests $\mathbf{q} \equiv \tilde{\mathbf{q}}$. Hence, $(w, \theta, \mathbf{q}) \equiv (\tilde{w}, \tilde{\theta}, \tilde{\mathbf{q}})$
    furnishing the uniqueness for Equations (\ref{EQUATION_QUASILINEAR_PDE_IN_W_THETA_AND_Q_1})--(\ref{EQUATION_IC_IN_W_THETA_AND_Q}).

    Now, we proceed with the equivalence.
    Given a classical solution $(w, \theta, \mathbf{q})$ to Equations (\ref{EQUATION_QUASILINEAR_PDE_IN_W_THETA_AND_Q_1})--(\ref{EQUATION_IC_IN_W_THETA_AND_Q}), 
    letting $z := A w$ and $p := \operatorname{div} \mathbf{q}$ as in Equation (\ref{EQUATION_DEFINITION_OF_Z_AND_P}),
    we trivially observe $(z, \theta, p)$ has the regularity as mandated by Definition \ref{DEFINITION_CLASSICAL_SOLUTION_Z_THETA_P}
    and solves Equations (\ref{EQUATION_QUASILINEAR_PDE_IN_Z_THETA_AND_P_1})--(\ref{EQUATION_IC_IN_Z_THETA_AND_P}).
    Due to unique solvability of the aforementioned system, no further solutions exist.

	Since $A^{-1}$ is an isomorphism between $H^{k}(\Omega)$ and $H^{k + 2}(\Omega) \cap H^{1}_{0}(\Omega)$, 
	noting $\triangle w = -A w$, we trivially obtain a unique function $w = A^{-1} z$ satisfying
	\begin{equation}
        w, \triangle w \in \Big(\bigcap_{m = 0}^{s - 1} C^{m}\big([0, T], H^{s - m}(\Omega) \cap H^{1}_{0}(\Omega)\big)\Big) \cap C^{s}\big([0, T], L^{2}(\Omega)\big). \notag
	\end{equation}
	Applying $A^{-1}$ to Equation (\ref{EQUATION_QUASILINEAR_PDE_IN_Z_THETA_AND_P_1}) and exploiting the regularity of $\theta$,
	we easily see $w$ satisfies (\ref{EQUATION_QUASILINEAR_PDE_IN_Z_THETA_AND_P_1}).

	The $\theta$-component remains unaltered in both frameworks -- including the regularity.
	Regarding the equivalence of Equations (\ref{EQUATION_QUASILINEAR_PDE_IN_W_THETA_AND_Q_2}) and (\ref{EQUATION_QUASILINEAR_PDE_IN_Z_THETA_AND_P_2}),
	assuming we can find $\mathbf{q}$ with the desired regularity such that $\operatorname{div} \mathbf{q} = p$,
	$\theta$ satisfies Equation (\ref{EQUATION_QUASILINEAR_PDE_IN_Z_THETA_AND_P_2}).
	
	Thus, there only remains to consider the $\mathbf{q}$-component.
	We start at the basic regularity level $p \in C^{0}\big([0, T], L^{2}(\Omega)\big)$.
	Solving Equation (\ref{EQUATION_QUASILINEAR_PDE_IN_W_THETA_AND_Q_3}) for $\mathbf{q}$, we get
	\begin{equation}
		\notag
		\mathbf{q}(t, \cdot) = \mathbf{q}^{0} + \int_{0}^{t} e^{-(t - s)/\tau} \nabla \theta(s, \cdot) \mathrm{d}s
		\quad \text{ with } \quad \mathbf{q}^{0} \in \big(H^{1}(\Omega)\big)^{d}.
	\end{equation}
	Since $\theta \in C^{0}\big([0, T], H^{1}_{0}(\Omega)\big)$, we can write
	\begin{equation}
		\label{EQUATION_Q_SOLUTION_ANSATZ}
		\mathbf{q}(t, \cdot) = \mathbf{q}^{0} + \nabla  \varphi(t, \cdot)
	\end{equation}
	for some (yet unknown) function $\varphi \in C^{0}\big([0, T], \nabla H^{1}_{0}(\Omega)\big)$.
	Computing
	\begin{equation*}
		\operatorname{div} \mathbf{q}(t, \cdot) \equiv
		\operatorname{div} \mathbf{q}^{0} + \triangle \varphi(t, \cdot) \equiv p^{0} - A \varphi(t, \cdot) \quad \text{ for } t \geq 0.
	\end{equation*}
	and using Equation (\ref{EQUATION_DEFINITION_OF_Z_AND_P}) as ansatz, i.e.,
	\begin{equation*}
        \operatorname{div} \mathbf{q} = p,
	\end{equation*}
	we obtain a family of elliptic equations
	\begin{equation*}
        p^{0} - A \varphi(t, \cdot) = p(t, \cdot)
	\end{equation*}
	for $\varphi \in H^{1}_{0}(\Omega)$. Since $A$ is an isomorphism, the latter is uniquely solved by
	\begin{equation*}
        \varphi(t, \cdot) = A^{-1} p^{0} - A^{-1} p(t, \cdot). 
	\end{equation*}
	Plugging the function $\varphi$ back into Equation (\ref{EQUATION_Q_SOLUTION_ANSATZ}), we get
	\begin{equation*}
        \mathbf{q}(t, \cdot) = \mathbf{q}^{0} + \nabla  A^{-1} p^{0} - \nabla A^{-1} p(t, \cdot)
        \quad \text{ with } \quad \mathbf{q}^{0} \in \big(H^{1}(\Omega)\big)^{d}.
	\end{equation*}
	
	On one hand, we can easily verify $\mathbf{q}$ satisfies
	\begin{equation*}
        \operatorname{div} \mathbf{q}(t, \cdot) = \operatorname{div} \mathbf{q}^{0} + \triangle  A^{-1} p^{0} - \triangle \nabla A^{-1} p(t, \cdot)
        = \operatorname{div} \mathbf{q}^{0} - p^{0} + p(t, \cdot) \equiv p(t, \cdot).
	\end{equation*}
	On the other hand, $\mathbf{q}$ solves Equation (\ref{EQUATION_QUASILINEAR_PDE_IN_Z_THETA_AND_P_3}) and fulfills the initial condition (\ref{EQUATION_IC_IN_W_THETA_AND_Q}):
	\begin{equation*}
         \mathbf{q}(0, \cdot) = \mathbf{q}^{0} + \nabla  A^{-1} p^{0} - \nabla A^{-1} p(0, \cdot)
         = \mathbf{q}^{0} + \nabla  A^{-1} p^{0} - \nabla A^{-1} p^{0} = \mathbf{q}^{0}.
	\end{equation*}
	Since the operator $\nabla A^{-1}$ is a continuous mapping between $H^{k}(\Omega)$ and $H^{k + 1}(\Omega)$ for $k \geq 0$,
	the desired regularity of $\mathbf{q}$ follows from that of $p$.
    Thus, we have proved $(w, \mathbf{q}, \theta)$ is a classical solution to (\ref{EQUATION_QUASILINEAR_PDE_IN_W_THETA_AND_Q_1})--(\ref{EQUATION_IC_IN_W_THETA_AND_Q}).
    Again, invoking the uniqueness, no further solutions exist.
\end{proof}

Hence, in the following, we investigate the more tractable -- but nonetheless equivalent -- non-vectorial reduced system
(\ref{EQUATION_QUASILINEAR_PDE_IN_Z_THETA_AND_P_1})--(\ref{EQUATION_IC_IN_Z_THETA_AND_P}).

\section{Local Well-Posedness: Proof of Theorem \ref{THEOREM_LOCAL_EXISTENCE}}
\label{SECTION_LOCAL_WELL_POSEDNESS}

To prove Theorem \ref{THEOREM_LOCAL_EXISTENCE}, we consider the reduced system (\ref{EQUATION_QUASILINEAR_PDE_IN_Z_THETA_AND_P_1})--(\ref{EQUATION_IC_IN_Z_THETA_AND_P}).
Recalling Equation (\ref{EQUATION_TRIANGLE_OF_K_TRIANGLE_W}) and letting
\begin{equation*}
    a(\xi) = K'(-\xi) \quad \text{ and } \quad f(\xi, \boldsymbol{\eta}) = K''(\xi) |\boldsymbol{\eta}|^{2} 
    \quad \text{ for } \quad \xi \in \mathbb{R}, \quad \boldsymbol{\eta} \in \mathbb{R}^{d},
\end{equation*}
we can write
\begin{equation}
    \label{EQUATION_RELATION_K_TO_A_AND_F}
    A K(z) = a(z) Az - f(z, \nabla z).
\end{equation}
Hence, Equations (\ref{EQUATION_QUASILINEAR_PDE_IN_Z_THETA_AND_P_1})--(\ref{EQUATION_IC_IN_Z_THETA_AND_P}) can equivalently be expressed as
\begin{subequations}
\begin{align}
	\big(A^{-1} + \gamma\big) z_{tt} + a(z) A z - \alpha A \theta &= f(z, \nabla z)\phantom{0} \text{ in } (0, \infty) \times \Omega,
	\label{EQUATION_NONLINEAR_PLATE_EQUATION_PDE_1} \\
	\beta \theta_{t} + p + \alpha z_{t} &= 0\phantom{f(z, \nabla z)} \text{ in } (0, \infty) \times \Omega,
	\label{EQUATION_NONLINEAR_PLATE_EQUATION_PDE_2} \\
	\tau p_{t} + p - \eta A \theta &= 0\phantom{f(z, \nabla z)} \text{ in } (0, \infty) \times \Omega,
	\label{EQUATION_NONLINEAR_PLATE_EQUATION_PDE_3} \\
	z = \theta &= 0\phantom{f(z, \nabla z)} \text{ in } (0, \infty) \times \partial \Omega,
	\label{EQUATION_NONLINEAR_PLATE_EQUATION_BC} \\
	z(0, \cdot) = z^{0}, \quad z_{t}(0, \cdot) = z^{1}, \quad 
	\theta(0, \cdot) = \theta^{0}, \quad p(0, \cdot) &= p^{0}\phantom{f(z, \nabla z} \text{ in } \Omega.
	\label{EQUATION_NONLINEAR_PLATE_EQUATION_IC}
\end{align}
\end{subequations}

\begin{remark}
    The results of this section remain true for general functions $a(\cdot)$ and $f(\cdot, \cdot)$, which are not necessarily related to the function $K(\cdot)$
    via Equation (\ref{EQUATION_RELATION_K_TO_A_AND_F}).
\end{remark}

With a straightforward modification of the construction performed in Section \ref{SECTION_MAIN_RESULTS}, we get:
\begin{definition}
	Let $z^{m}$, $\theta^{m}$, $p^{m}$, $m \geq 0$, denote the `initial values' for $\partial_{t}^{m} z$, $\partial_{t}^{m} \theta$ and $\partial_{t}^{m} p$.
\end{definition}

In the spirit of Assumption \ref{ASSUMPTION_LOCAL_EXISTENCE}, we impose the following conditions:
\begin{assumption}
	\label{ASSUMPTION_LOCAL_EXISTENCE_REDUCED_SYSTEM}
	Let $s \geq \lfloor \tfrac{d}{2}\rfloor + 2$ be an integer and let $\Omega \subset \mathbb{R}^{d}$ be a bounded domain with $\partial \Omega \in C^{s}$.
	\begin{enumerate}
        \item Let $a \in C^{s - 1}(\mathbb{R}, \mathbb{R})$ and $f \in C^{s - 1}(\mathbb{R} \times \mathbb{R}^{d}, \mathbb{R})$.

		\item Let the initial data satisfy the regularity and compatibility conditions
		\begin{equation}
			\begin{split}
				z^{m} &\in H^{s - m}(\Omega) \cap H^{1}_{0}(\Omega) \quad \text{ for } m = 0, \dots, s - 1, \quad z^{s} \in L^{2}(\Omega) \quad \text{ and } \\
				\theta^{k} &\in H^{s - k}(\Omega) \cap H^{1}_{0}(\Omega), \quad
				p^{k} \in H^{s - 1 - k}(\Omega) \text{ for } k = 0, \dots, s - 1,
			\end{split}
			\notag
		\end{equation}
		where $H^{0}(\Omega) := L^{2}(\Omega)$.
		
		\item For the ``initial ellipticity'' of $a(z^{0}) A$, suppose
		\begin{equation}
			\min_{x \in \bar{\Omega}} a\big(z^{0}(x)\big) > 0, \quad
			\text{where } z^{0} \in C^{0}(\bar{\Omega}) \text{ by virtue of Sobolev's embedding theorem.} \notag
		\end{equation}
	\end{enumerate}
\end{assumption}

We also introduce the following notation for the time-space gradient operator used for the proof of Theorem \ref{THEOREM_LOCAL_EXISTENCE_REDUCED_SYSTEM} below:
\begin{equation}
    \label{EQUATION_OPERATOR_D_N}
    \bar{D}^{n} := \big((\partial_{t}, \nabla)^{\alpha} \,|\, \alpha \in \mathbb{N}_{0}^{d + 1}, \; 0 \leq |\alpha| \leq n\big) 
    \quad \text{ for } \quad n \geq 0,
\end{equation}

\begin{remark}
    By the equivalence Theorem \ref{THEOREM_SYSTEM_EQUIVALENCE},
    $(w^{0}, w^{1}, \theta^{0}, p^{0})$ satisfy Assumption \ref{ASSUMPTION_LOCAL_EXISTENCE} if and only if
    $(z^{0}, z^{1}, \theta^{0}, \mathbf{q}^{0})$ satisfy Assumption \ref{ASSUMPTION_LOCAL_EXISTENCE_REDUCED_SYSTEM}
    and $\mathbf{q}^{0} \in \big(H^{s - 1}(\Omega)\big)^{d}$.
\end{remark}

Next, we prove the following `auxiliary' result for the reduced system (\ref{EQUATION_NONLINEAR_PLATE_EQUATION_PDE_1})--(\ref{EQUATION_NONLINEAR_PLATE_EQUATION_IC}):

\begin{theorem}[Local Well-Posedness]
	\label{THEOREM_LOCAL_EXISTENCE_REDUCED_SYSTEM}
	Suppose Assumption \ref{ASSUMPTION_LOCAL_EXISTENCE_REDUCED_SYSTEM} is true for some $s \geq \lfloor \tfrac{d}{2}\rfloor + 2$.
	Then, Equations (\ref{EQUATION_NONLINEAR_PLATE_EQUATION_PDE_1})--(\ref{EQUATION_NONLINEAR_PLATE_EQUATION_IC})
	possess a unique classical solution $(z, \theta, p)$ at the energy level $s$ on a maximal interval $[0, T_{\mathrm{max}}) \neq \emptyset$ such that:
	\begin{enumerate}
        \item ``Local non-degeneracy:'' $\min\limits_{x \in \bar{\Omega}} a\big(z(t, x)\big) > 0$ for any $t \in [0, T_{\mathrm{max}})$
        
        \item ``Blow-up or eventual degeneracy if solution non-global:''
        Unless $T_{\mathrm{max}} = \infty$, either the ellipticity condition is violated
        \begin{equation}
            \min\limits_{x \in \bar{\Omega}} a\big(z(t, x)\big) \to 0 \text{ as } t \nearrow T_{\mathrm{max}}
            \label{EQUATION_REDUCED_SYSTEM_ELLIPTICITY_VIOLATION_AT_T_MAX}
        \end{equation}
        or/and the blow-up occurs
        \begin{equation}
            \big\|z(t, \cdot)\big\|_{H^{s}(\Omega)}^{2} \to \infty
            \text{ as } t\nearrow T_{\mathrm{max}}.
            \label{EQUATION_REDUCED_SYSTEM_BLOW_UP_AT_T_MAX}
        \end{equation}
        
        \item ``Solution map continuity:''
        For any $T > 0$, $\varepsilon > 0$ and $N > 0$, the solution mapping $(z^{0}, z^{1}, \theta^{0}, p^{0}) \mapsto (z, z_{t}, \theta, p)$
        is a continuous function from
        \begin{align*}
            \mathscr{M}_{T, \varepsilon, N} :=
            \Big\{(z^{0}, z^{1}, \theta^{0}, p^{0}) \,\big|\, &(z^{0}, z^{1}, \theta^{0}, p^{0}) \text{ satisfy Assumption \ref{ASSUMPTION_LOCAL_EXISTENCE_REDUCED_SYSTEM}} \text{ and admit } \\
            &\text{a classical solution } (z, \theta, p) \text{ with }
            \min_{t \in [0, T]} \min_{x \in \bar{\Omega}} a\big(z(t, x)\big) \geq \varepsilon, \\
            &\max_{0 \leq t \leq T} \sum_{m = 0}^{s} \big\|\partial_{t}^{m} z(t, \cdot)\big\|_{H^{s - m}(\Omega)}^{2} \leq N^{2}\Big\}
        \end{align*}
        endowed with the topology of $H^{1}(\Omega) \times L^{2}(\Omega) \times H^{1}(\Omega) \times \big(H^{1}(\Omega))^{d}\big)$
        to $L^{\infty}\big(0, T; H^{1}(\Omega) \times L^{2}(\Omega) \times H^{1}(\Omega) \times \big(H^{1}(\Omega))^{d}\big)\big)$.
	\end{enumerate}
\end{theorem}

\begin{proof}
    When treating the $z$-component in this proof, we will rather closely follow the streamlines of \cite[Section 4]{LaPoWa2017}.
    A major difference over \cite[Section 4]{LaPoWa2017} is that the equation for $z$ cannot be decoupled from those for $\theta, p$
    due to hyperbolicity of the problem under consideration due to the presence of a strong coupling between the equations.
    
    Using the second Hilbert's identity
    \begin{equation}
		\big(A^{-1} + \gamma\big)^{-1} = \tfrac{1}{\gamma} - \tfrac{1}{\gamma} A^{-1} \big(A^{-1} + \gamma\big)^{-1}, \notag
	\end{equation}
	Equations (\ref{EQUATION_NONLINEAR_PLATE_EQUATION_PDE_1})--(\ref{EQUATION_NONLINEAR_PLATE_EQUATION_IC}) are transformed to a second-order hyperbolic system
	\begin{align}
		z_{tt} + \tfrac{1}{\gamma} a(z) A z - \tfrac{\alpha}{\gamma} A \theta + B \theta &= F(z, \theta) & &\text{ in } (0, \infty) \times \Omega,
		\label{EQUATION_NONLINEAR_PLATE_EQUATION_TRANSFORMED_PDE_1} \\
		\beta \theta_{t} + p + \alpha z_{t} &= 0 & &\text{ in } (0, \infty) \times \Omega,
		\label{EQUATION_NONLINEAR_PLATE_EQUATION_TRANSFORMED_PDE_2} \\
		\tau p_{t} + p - \eta A \theta &= 0 & &\text{ in } (0, \infty) \times \Omega,
		\label{EQUATION_NONLINEAR_PLATE_EQUATION_TRANSFORMED_PDE_3} \\
		z = \theta &= 0 & &\text{ on } (0, \infty) \times \partial \Omega,
		\label{EQUATION_NONLINEAR_PLATE_EQUATION_TRANSFORMED_BC} \\
		z(0, \cdot) = z^{0}, \quad z_{t}(t, \cdot) = z^{1}, \quad \theta(0, \cdot) = \theta^{0}, \quad p(0, \cdot) &= p^{0} & &\text{ in } \Omega,
		\label{EQUATION_NONLINEAR_PLATE_EQUATION_TRANSFORMED_IC}
	\end{align}
	perturbed by the nonlinear nonlocal operator $F(\cdot)$ given by
	\begin{equation}
        \label{EQUATION_NONLINEAR_MAPPING_F_OF_Z}
		F(z) = \tfrac{1}{\gamma} (1 - I_{\gamma}) f(z, \nabla z) + \tfrac{1}{\gamma} I_{\gamma} \big(a(z) A z\big),
	\end{equation}
	where the compact linear operator
	\begin{equation}
		I_{\gamma} := A^{-1} \big(\gamma + A^{-1}\big)^{-1} \notag
	\end{equation}
	is a continuous mapping from $H^{s}(\Omega)$ to $H^{s + 2}(\Omega) \cap H^{1}_{0}(\Omega)$ for any $s \geq 0$ and
	\begin{equation*}
        B := \tfrac{\alpha}{\gamma} I_{\gamma} A = \tfrac{\alpha}{\gamma} \big(\gamma + A^{-1}\big)^{-1}
	\end{equation*}
	is a bounded linear operator on both $H^{s}(\Omega)$ and $H^{s}(\Omega) \cap H^{1}_{0}(\Omega)$ for any $s \geq 0$.
	
	\medskip

	\noindent {\it Step 1: Amending the nonlinearity $a(\cdot)$.}
	Since no global positivity is available for $a(\cdot)$,
	the ellipticity condition for $a\big(z(t, \cdot)\big) A$ can be violated at any time $t > 0$.
	To preliminarily rule out this possible degeneracy, the following construction proves to be helpful.
	
	On the strength of the continuity of $z^{0}$ and the connectedness of $\Omega$, we get
	\begin{equation}
        \label{EQUATION_IMAGE_OF_DOMAIN_OMEGA_WRT_Z_NOUGHT}
		z^{0}(\bar{\Omega}) = \big[\min_{x \in \bar{\Omega}} z^{0}(x), \max_{x \in \bar{\Omega}} z^{0}(x)\big]
		=: J_{0}.
	\end{equation}
	By Assumption \ref{ASSUMPTION_LOCAL_EXISTENCE_REDUCED_SYSTEM}.3, $a(\cdot)$ is positive on $J_{0}$.
	Since $a^{-1}\big((0, \infty)\big)$ is open and $J_{0} \subset a^{-1}\big((0, \infty)\big)$,
	for $\epsilon > 0$ sufficiently small, we consider
	\begin{equation}
        \label{EQUATION_SET_J_NONTRIVIAL}
        J_{0} \subset J_{\epsilon} := a^{-1}\big([\epsilon, \infty)\big) \neq \emptyset.
	\end{equation}
	Further, there exists a global $C^{s}$-extension $\hat{a}_{\epsilon}(\cdot)$
	(denoted for simplicity by $\hat{a}(\cdot)$) of $a(\cdot)$ such that
	\begin{equation}
        \label{EQUATION_PROPERTY_A_HAT}
        \hat{a}(z) = a(z) \quad \text{ for } \quad z \in J_{0} \quad \text{ and } \quad 
        \inf_{z \in J_{\epsilon}} \hat{a}(z) \geq \epsilon > 0.
	\end{equation}
	We replace Equation (\ref{EQUATION_NONLINEAR_PLATE_EQUATION_TRANSFORMED_PDE_1}) with
	\begin{align}
		z_{tt} + \tfrac{1}{\gamma} \hat{a}(z) A z - \tfrac{\alpha}{\gamma} A \theta = F(z, \theta) \quad \text{ in } (0, \infty) \times \Omega
		\label{EQUATION_NONLINEAR_PLATE_EQUATION_TRANSFORMED_NEW_NONLINEARITY_PDE_1}
	\end{align}
	and first consider the amended system
	(\ref{EQUATION_NONLINEAR_PLATE_EQUATION_TRANSFORMED_NEW_NONLINEARITY_PDE_1}),
	(\ref{EQUATION_NONLINEAR_PLATE_EQUATION_TRANSFORMED_PDE_2})--(\ref{EQUATION_NONLINEAR_PLATE_EQUATION_TRANSFORMED_IC}).
	The idea behind this modification is that, despite both systems being {\it a priori} not equivalent,
	the equivalence will turn out to be valid {\it a posteriori} -- provided the time $T$ is short.
	
	To solve the amended problem 
	(\ref{EQUATION_NONLINEAR_PLATE_EQUATION_TRANSFORMED_NEW_NONLINEARITY_PDE_1}),
	(\ref{EQUATION_NONLINEAR_PLATE_EQUATION_TRANSFORMED_PDE_2})--(\ref{EQUATION_NONLINEAR_PLATE_EQUATION_TRANSFORMED_IC}),
	we transform it to a fixed-point problem and consequently solved using the Banach's fixed-point theorem.
	As previously pointed out, our procedure is reminiscent of \cite[Theorem 5.2]{JiaRa2000} and \cite[Section 4]{LaPoWa2017}.

	\medskip
	
	\noindent {\it Step 2: Defining the fixed-point mapping.}
	Recalling $H^{0}_{0}(\Omega) \equiv H^{0}(\Omega) := L^{2}(\Omega)$,
	for $N > 0$ and $T > 0$, let $X(N, T)$ denote the set of all regular distributions $(z, \theta, q)$ such that
	\begin{align}
        \label{EQUATION_SET_X_OF_N_AND_T_REGULARITY}
        \begin{split}
        	\partial_{t}^{m} z &\in C^{0}\big([0, T], H^{s - m}(\Omega)\big) \text{ for } m = 0, 1, \dots, s, \\
            \partial_{t}^{k} \theta &\in C^{0}\big([0, T], H^{s - m}(\Omega)\big) \text{ and }
            \partial_{t}^{k} p \in C^{0}\big([0, T], H^{s - 1 - m}(\Omega)\big) \quad \text{ for } k = 0, 1, \dots, s - 1
        \end{split}
	\end{align}
	satisfying the boundary conditions
	\begin{equation*}
		\partial_{t}^{m} z = \partial_{t}^{k} \theta = 0 \text{ on } [0, T] \times \partial \Omega \text{ for } 
		m = 0, 1, \dots, s - 1 \text{ and } k = 0, 1, \dots, s - 2
	\end{equation*}
	and the initial conditions
	\begin{equation}
        \label{EQUATION_DEFINITION_OF_X_N_T_INITIAL_CONDITIONS}
        \begin{split}
            \partial_{t}^{m} z(0, \cdot) &= z^{m} \text{ in } \Omega \quad \text{ for } m = 0, 1, \dots, s, \\
            \partial_{t}^{k} \theta(0, \cdot) = \theta^{k}, \quad \partial_{t}^{k} p(0, \cdot) &= p^{k} \text{ in } \Omega \quad \text{ for } k = 0, 1, \dots, s - 1 
        \end{split}
	\end{equation}
	along with the energy estimate
	\begin{equation}
        \label{EQUATION_ENERGY_CONSTRAINT_FOR_X_N_T}
        \max_{0 \leq t \leq T} \Big(\|\bar{D}^{s} z(t, \cdot)\|_{L^{2}(\Omega)}^{2} + \|\bar{D}^{s - 1} \theta(t, \cdot)\|_{H^{1}(\Omega)}^{2} +
        \|\bar{D}^{s - 1} p(t, \cdot)\|_{L^{2}(\Omega)}^{2}\Big) \leq N^{2}
	\end{equation}
	with the time-space gradient $\bar{D}^{(\cdot)}$ defined in Equation (\ref{EQUATION_OPERATOR_D_N}).
	By a standard argument (cf. \cite[p. 195]{LaPoWa2017}),
	for any $T_{0} > 0$ and sufficiently large $N > 0$, the set $X(N, T)$ is not empty for any $T \in (0, T_{0}]$.
	
	For $(\bar{z}, \bar{\theta}, \bar{p}) \in X(N, T)$, consider the linear operator $\mathscr{F}$
	mapping $(\bar{z}, \bar{\theta}, \bar{p})$ to a function triple $(z, \theta, p)$ solving the linear nonhomogeneous system
	\begin{align}
		z_{tt} - \bar{a}_{ij}(t, x) \partial_{x_{i}} \partial_{x_{j}} z - \tfrac{\alpha}{\gamma} A \theta + B \theta &= \bar{f}(t, x) & &\text{ in } (0, \infty) \times \Omega,
		\label{EQUATION_FIXED_POINT_PDE_1} \\
		\beta \theta_{t} + p + \alpha z_{t} &= 0 & &\text{ in } (0, \infty) \times \Omega,
		\label{EQUATION_FIXED_POINT_PDE_2} \\
		\tau p_{t} + p - \eta A \theta &= 0 & &\text{ in } (0, \infty) \times \Omega,
		\label{EQUATION_FIXED_POINT_PDE_3} \\
		z = \theta &= 0 & &\text{ on } (0, \infty) \times \partial \Omega,
		\label{EQUATION_FIXED_POINT_BC} \\
		z(0, \cdot) = z^{0}, \quad z_{t}(t, \cdot) = z^{1}, \quad \theta(0, \cdot) = \theta^{0}, \quad p(0, \cdot) &= p^{0} & &\text{ in } \Omega,
		\label{EQUATION_FIXED_POINT_IC}
	\end{align}
	with
	\begin{equation}
        \label{EQUATION_DEFINITION_BAR_A_AND_BAR_F_LOCAL_EXISTENCE}
        \begin{split}
            \bar{a}_{ij}(t, x) &:= \tfrac{1}{\gamma} \hat{a}\big(\bar{z}(t, x)\big) \delta_{ij} \partial_{x_{i}} \partial_{x_{j}}, \\
            \bar{f}(t, x) &:= \tfrac{1}{\gamma} \big((1 - I_{\gamma}) f(\bar{z}, \nabla \bar{z})\big)(t, x) + 
            \tfrac{1}{\gamma} \big(I_{\gamma} \big(\hat{a}(\bar{z}) A \bar{z}\big)\big)(t, x).
        \end{split}
	\end{equation}
    for $(t, x) \in [0, T] \times \Omega$. Note that $\mathscr{F}$ only depends on $\bar{z}$, not on $\bar{\theta}, \bar{p}$.
	
	We show $\mathscr{F}$ is well-defined.
	Taking into account the regularity of $\bar{z}$, invoking Assumption \ref{ASSUMPTION_LOCAL_EXISTENCE_REDUCED_SYSTEM} and Equation (\ref{EQUATION_PROPERTY_A_HAT})
	as well as Sobolev's embedding theorem, we can verify that Assumption \ref{ASSUMPTION_APPENDIX} is satisfied with
	\begin{equation}
		\gamma_{i} = \max_{0 \leq t \leq T} \bar{\gamma}_{i}\big(\big\|\bar{z}(t, \cdot)\big\|_{H^{s - 1}(\Omega)}\big)
		\text{ for } i = 0, 1,
		\label{EQUATION_LOCAL_EXISTENCE_DEFINITION_OF_GAMMA_I}
	\end{equation}
	for appropriate continuous functions $\gamma_{0}, \gamma_{1} \colon [0, \infty) \to (0, \infty)$,
	where we used the Sobolev's embedding $\nabla \bar{z}(t, \cdot) \in H^{2}(\Omega) \hookrightarrow L^{\infty}(\Omega)$ along with the elliptic estimate
	\begin{equation}
		\big\|I_{\gamma}\big(\hat{a}(\bar{z}) A \bar{z}\big)\big\|_{H^{m + 2}(\Omega)} \leq
		C \big\|\hat{a}(\bar{z}) A \bar{z}\big\|_{H^{m}(\Omega)} \text{ for } m = 0, 1, \dots, s - 2. \notag
	\end{equation}
	Here and in the sequel, $C > 0$ denotes a generic constant.
	Hence, by virtue of Theorem \ref{THEOREM_APPENDIX}, Equations (\ref{EQUATION_FIXED_POINT_PDE_1})--(\ref{EQUATION_FIXED_POINT_IC}) 
	possesses a unique classical solution $(z, \theta, p)$ with
	\begin{align*}
		z &\in \Big(\bigcap_{m = 0}^{s - 1} C^{m}\big([0, T], H^{s - m}(\Omega)\big)\Big) \cap C^{s}\big([0, T], L^{2}(\Omega)\big), \\
		\theta &\in \bigcap_{m = 0}^{s - 1} C^{m}\big([0, T], H^{s - m}(\Omega)\big), \quad
		p \in \bigcap_{m = 0}^{s - 1} C^{m}\big([0, T], H^{s - 1 - m}(\Omega)\big).
	\end{align*}
    Therefore, the mapping $\mathscr{F}$ is well-defined. 
    
	\medskip
	
	\noindent {\it Step 3: Showing the self-mapping property.}
	
	We prove that $\mathscr{F}$ maps $X(N, T)$ into itself provided $N$ is sufficiently large and $T$ is sufficiently small. To this end, we define
	\begin{align*}
		E_{0} &:= \sum_{m = 0}^{s} \|z^{m}\|_{H^{s - m}(\Omega)}^{2} + \sum_{k = 0}^{s - 1} \|\theta^{k}\|_{H^{s - k}(\Omega)}^{2} + \sum_{k = 0}^{s - 1} \|p^{k}\|_{H^{s - 1 - k}(\Omega)}^{2} \\
		&+ \sum_{m = 0}^{s - 2} \max_{0 \leq t \leq T} \big\|\partial_{t}^{m} f(t, \cdot)\big\|_{H^{s - 2 - m}(\Omega)}^{2} 
		+ \int_{0}^{T} \big\|\partial_{t}^{s - 1} f(t, \cdot)\big\|_{L^{2}(\Omega)}^{2} \mathrm{d}t.
	\end{align*}
	Similar to \cite[Equations (4.19), (4.20)]{LaPoWa2017} Taking into account Equations (\ref{EQUATION_DEFINITION_OF_X_N_T_INITIAL_CONDITIONS}), 
	(\ref{EQUATION_ENERGY_CONSTRAINT_FOR_X_N_T}) and (\ref{EQUATION_DEFINITION_BAR_A_AND_BAR_F_LOCAL_EXISTENCE})
	and applying Sobolev's embedding theorem, the fundamental theorem of calculus along with \cite[Theorem B.6]{JiaRa2000} yields
	\begin{align}
		\label{EQUATION_FIXED_POINT_MAPPING_RIGHT_HAND_SIDE_ESTIMATE_HIGHEST_ENERGY_LEVEL}
		\int_{0}^{T} \|\partial_{t}^{s - 1} \bar{f}(t, \cdot)\|_{L^{2}(\Omega)}^{2} \mathrm{d}t &\leq C(N) (1 + T), \\
		\label{EQUATION_FIXED_POINT_MAPPING_RIGHT_HAND_SIDE_ESTIMATE}
		\max_{0 \leq t \leq T} \Big(\sum_{m = 0}^{s - 2}  \|\partial_{t}^{m} \bar{f}(t, \cdot)\|_{H^{s - 2 - m}(\Omega)}^{2}\Big)
		&\leq C(E_{0}) + C(N)(1 + T).
	\end{align}
	
	Plugging Equations (\ref{EQUATION_FIXED_POINT_MAPPING_RIGHT_HAND_SIDE_ESTIMATE_HIGHEST_ENERGY_LEVEL}) and (\ref{EQUATION_FIXED_POINT_MAPPING_RIGHT_HAND_SIDE_ESTIMATE})
	into the energy estimate in Theorem \ref{THEOREM_APPENDIX}, we arrive at
	\begin{equation}
        \label{EQUATION_LOCAL_EXISTENCE_ESTIMATE}
		\max_{0 \leq t \leq T} \Big(\|\bar{D}^{s} z(t, \cdot)\|_{L^{2}(\Omega)}^{2} + \|\bar{D}^{s - 1} \theta(t, \cdot)\|_{H^{1}(\Omega)}^{2}
		+ \|\bar{D}^{s - 1} p(t, \cdot)\|_{L^{2}(\Omega)}^{2}\Big) \leq \bar{K}(E_{0}, \gamma_{0}, \gamma_{1}) \zeta(N, T)
	\end{equation}
	with positive constants $\gamma_{0}, \gamma_{1}$ defined in Equation (\ref{EQUATION_LOCAL_EXISTENCE_DEFINITION_OF_GAMMA_I}),
	a positive `constant' $\bar{K}$, being a continuous function of its variables, and
	\begin{equation}
		\zeta(N, T) = \Big(1 + C(N) T^{1/2} \sum_{i = 0}^{5} T^{i/2}\Big) \exp\big(T^{1/2} C(N) (1 + T^{1/2} + T + T^{3/2})\big). \notag
	\end{equation}
	
	Select $N$ such that
	\begin{equation}
		 \bar{K}(E_{0}, \gamma_{0}, \gamma_{1}) \leq \tfrac{1}{2} N^{2}. \notag
	\end{equation}
	Due to the continuity of $\zeta(N, \cdot)$ in $T = 0$ and the fact $\zeta(N_{0}, 0) = 1$, there exists $T > 0$ such that $\zeta\big(N, (0, T]\big) \subset [1, 2]$.
	Hence, the estimate in Equation (\ref{EQUATION_LOCAL_EXISTENCE_ESTIMATE}) is satisfied with $N^{2}$ on the right-hand side.
	Thus, $(z, \theta, p) \in X(N, T)$ implying $\mathscr{F}$ maps $X(N, T)$ into itself. 
	
	\medskip
	
	\noindent {\it Step 4: Proving the contraction property.}
	Consider the metric space
	\begin{equation}
		Y := \Big\{(z, \theta, p) \,\big|\,
		z, z_{t}, |\nabla z| \in L^{\infty}\big(0, T; L^{2}(\Omega)\big), \theta \in L^{\infty}\big(0, T; H^{1}(\Omega)\big)  \text{ and }
		p \in L^{\infty}\big(0, T; L^{2}(\Omega)\big)\Big\} \notag
	\end{equation}
	endowed with the distance
	\begin{align*}
		\rho\big((z, \theta, p), (\bar{z}, \bar{\theta}, \bar{p})\big) =
		\mathop{\operatorname{ess\,sup}}_{0 \leq t \leq T}
		\Big(\big\|\bar{D}^{1} \big(z - \bar{z}\big)(t, \cdot)\big\|_{L^{2}(\Omega)}^{2} &+ \big\|(\theta - \bar{\theta})(t, \cdot)\big\|_{H^{1}(\Omega)}^{2} \\
		&+\big\|(p - \bar{p})(t, \cdot)\big\|_{L^{2}(\Omega)}^{2}\Big)^{1/2}
		\notag
	\end{align*}
	for $(z, \theta, p), (\bar{z}, \bar{\theta}, \bar{p}) \in Y$.
	Being endowed with its natural topology, $Y$ is complete. Arguing as \cite[p. 197]{LaPoWa2017}, we see $X(N, T) \subset Y$ is closed in $Y$.
	
	We now prove that $\mathscr{F} \colon X(N, T) \to X(N, T)$ is a contraction mapping with respect to $\rho$.
	For $(\bar{z}, \bar{\theta}, \bar{p}), (\bar{z}^{\ast}, \bar{z}^{\ast}, \bar{p}^{\ast}) \in X(N, T)$, let
	$(z, \theta, p) := \mathscr{F}\big((\bar{z}, \bar{\theta}, \bar{p})\big)$, 
	$(z^{\ast}, \theta^{\ast}, p^{\ast}) := \mathscr{F}\big((\bar{z}^{\ast}, \bar{\theta}^{\ast}, \bar{p}^{\ast})\big)$.
	With $(\bar{z}, \bar{\theta}, \bar{p})$, $(\bar{z}^{\ast}, \bar{\theta}^{\ast}, \bar{p}^{\ast})$, $(z, \theta, p)$, $(z^{\ast}, \theta^{\ast}, p^{\ast})$ all lying in $X(N, T)$,
	Equation (\ref{EQUATION_ENERGY_CONSTRAINT_FOR_X_N_T}) together with Sobolev's embedding theorem imply
	\begin{equation}
		\mathop{\operatorname{ess\,sup}}_{0 \leq t \leq T}
		\big\|\big(\bar{D}^{1}(\bar{z}, \bar{z}^{\ast}, z, z^{\ast})\big)(t, \cdot)\big\|_{L^{\infty}(\Omega)} \leq CN.
		\label{EQUATION_NONLINEAR_PLATE_EQUATION_NABLA_Z_ESTIMATE}
	\end{equation}
	Recalling Equations (\ref{EQUATION_NONLINEAR_PLATE_EQUATION_TRANSFORMED_PDE_1})--(\ref{EQUATION_NONLINEAR_PLATE_EQUATION_TRANSFORMED_PDE_2}),
	we can easily see $(\tilde{z}, \tilde{\theta}, \tilde{p}) := (z - z^{\ast}, \theta - \theta^{\ast}, p - p^{\ast})$ satisfies
	\begin{align}
		\tilde{z}_{tt} + \tfrac{1}{\gamma} \hat{a}(z) A \tilde{z} - \tfrac{\alpha}{\gamma} A \tilde{\theta} + B \tilde{\theta} &=
		\big(F(\bar{z}) - F(\bar{z}^{\ast})\big) - \big(\hat{a}(\bar{z}) - \hat{a}(\bar{z}^{\ast})\big) A z^{\ast}, 
		\label{EQUATION_NONLINEAR_PLATE_EQUATION_TRANSFORMED_SOLUTION_DIFFERENCE_PDE_1} \\
		\beta \tilde{\theta}_{t} + \tilde{p} + \alpha \tilde{z}_{t} &= 0,
		\label{EQUATION_NONLINEAR_PLATE_EQUATION_TRANSFORMED_SOLUTION_DIFFERENCE_PDE_2} \\
		\tau \tilde{p}_{t} + \tilde{p} - \eta A \tilde{\theta} &= 0,
		\label{EQUATION_NONLINEAR_PLATE_EQUATION_TRANSFORMED_SOLUTION_DIFFERENCE_PDE_3} \\
		\tilde{z}|_{\partial \Omega} = \tilde{\theta}|_{\partial \Omega} &= 0,
		\label{EQUATION_NONLINEAR_PLATE_EQUATION_TRANSFORMED_SOLUTION_DIFFERENCE_BC} \\
		\tilde{z}(0, \cdot) \equiv 0, \quad \tilde{z}_{t}(t, \cdot) \equiv 0, \quad \tilde{\theta}(0, \cdot) \equiv 0, \quad \tilde{p}(0, \cdot) &\equiv 0.
		\label{EQUATION_NONLINEAR_PLATE_EQUATION_TRANSFORMED_SOLUTION_DIFFERENCE_IC}
	\end{align}
	Multiplying Equations (\ref{EQUATION_NONLINEAR_PLATE_EQUATION_TRANSFORMED_SOLUTION_DIFFERENCE_PDE_1})--(\ref{EQUATION_NONLINEAR_PLATE_EQUATION_TRANSFORMED_SOLUTION_DIFFERENCE_PDE_3})
	in $L^{2}\big((0, t) \times \Omega\big)$ with $\frac{1}{\gamma} \tilde{z}_{t}$, $\frac{1}{\gamma \eta} A \tilde{\theta}$ and $\tilde{p}$, respectively,
	using the fact
	\begin{equation*}
        \|\nabla \hat{a}(z)\|_{L^{\infty}((0, T) \times \Omega)} \leq C(N)
	\end{equation*}
	and proceeding similar to the energy estimate part of the proof of Theorem \ref{THEOREM_APPENDIX}, we obtain
	\begin{align}
        \label{EQUATION_SELF_MAPPING_F_CONTRACTION_ESTIMATE}
        \begin{split}
            \big\|\bar{D}^{1} &\tilde{z}(t, \cdot)\|_{L^{2}(\Omega)}^{2} + \big\|\tilde{\theta}(t, \cdot)\|_{H^{1}(\Omega)}^{2} + \big\|\tilde{p}(t, \cdot)\|_{L^{2}(\Omega)}^{2} \\
            &\leq C(N) \int_{0}^{t} \Big\|\big(F(\bar{z}) - F(\bar{z}^{\ast})\big) - \big(\hat{a}(\bar{z}) - \hat{a}(\bar{z}^{\ast})\big) A z^{\ast}\Big\|_{L^{2}(\Omega)}^{2} \mathrm{d}t
        \end{split}
    \end{align}
    for $t \in [0, T]$.
    In view of local Lipschitzianity of $\hat{a}(\cdot)$ and $f(\cdot, \cdot)$ 
    (and, thus, that of $F(\cdot)$ from Equation (\ref{EQUATION_NONLINEAR_MAPPING_F_OF_Z}) on $L^{2}(\Omega)$)
    together with the energy bound in Equation (\ref{EQUATION_ENERGY_CONSTRAINT_FOR_X_N_T}), Equation (\ref{EQUATION_SELF_MAPPING_F_CONTRACTION_ESTIMATE}) implies
    \begin{align*}
        \max_{0 \leq t \leq T} &\Big(\big\|\bar{D}^{1} \tilde{z}(t, \cdot)\|_{L^{2}(\Omega)}^{2} + \big\|\tilde{\theta}(t, \cdot)\|_{H^{1}(\Omega)}^{2} + \big\|\tilde{p}(t, \cdot)\|_{L^{2}(\Omega)}^{2}\Big) \\
        &\leq C(N)T \Big(\max_{0 \leq t \leq T} \big\|\bar{D}^{1} \tilde{z}(t, \cdot)\|_{L^{2}(\Omega)}^{2} + \big\|\tilde{\theta}(t, \cdot)\|_{H^{1}(\Omega)}^{2} + \big\|\tilde{p}(t, \cdot)\|_{L^{2}(\Omega)}^{2}\Big).
    \end{align*}
    Thus, selecting $T$ sufficiently small such that $\lambda := C(N)T < 1$, we arrive at
    \begin{equation}
		\rho\big((z, \theta, p), (z^{\ast}, \theta^{\ast}, p^{\ast})\big) \leq
		\lambda \rho\big((\bar{z}, \bar{\theta}, \bar{p}), (\bar{z}^{\ast}, \bar{\theta}^{\ast}, \bar{p}^{\ast})\big) \notag
	\end{equation}
	meaning $\mathscr{F}$ is a contraction on the closed subset $X(N, T)$ of the metric space $Y$.
    Thus, on the strength of Banach's fixed-point theorem, $\mathscr{F}$ possesses a unique fixed point $(z, \theta, p) \in X(N, T)$.
    Having the smoothness specified in Equation (\ref{EQUATION_SET_X_OF_N_AND_T_REGULARITY}), by definition of the fixed-point mapping $\mathscr{F}(\cdot)$,
    $(z, \theta, p)$ is the a unique classical solution to Equations 
    (\ref{EQUATION_NONLINEAR_PLATE_EQUATION_TRANSFORMED_PDE_1})--(\ref{EQUATION_NONLINEAR_PLATE_EQUATION_TRANSFORMED_IC}) at the energy level $s$.
    
    \medskip
    
	\noindent {\it Step 5: Continuation to the maximal interval.}
	Due to the smoothness of $(z, \theta, p)$ at $t = T$, $\big(z(T, \cdot), z_{t}(T, \cdot), \theta(T, \cdot), p(T, \cdot)\big)$ 
	satisfies regularity and compatibility Assumption \ref{ASSUMPTION_LOCAL_EXISTENCE_REDUCED_SYSTEM}.2
	(the initial ellipticity condition is satisfied automatically according to the definition of $\hat{a}(\cdot)$),
	a standard continuation argument yields a maximal interval $[0, T^{\ast}_{\epsilon})$ for which the classical solution uniquely exists.
	Due to the interval's maximality, unless $T_{\epsilon}^{\ast} = \infty$, we have
	\begin{equation}
        \big\|\bar{D}^{s} z(t, \cdot)\big\|_{L^{2}(\Omega)}^{2} + \big\|\bar{D}^{s - 1} z(t, \cdot)\big\|_{H^{1}(\Omega)}^{2} +
        \big\|\bar{D}^{s - 1} p(t, \cdot)\big\|_{L^{2}(\Omega)}^{2} \to \infty \text{ as } t \nearrow T^{\ast}_{\epsilon}.
		\label{EQUATION_BLOW_UP_AT_T_AST_NAIVE}
	\end{equation}
	
	\noindent {\it Step 6: Returning to the original system.}
	We argue similar to \cite[pp. 198--199]{LaPoWa2017}.
	By virtue of Sobolev's embedding theorem,
	the composition $a \circ z$ is continuous on $[0, T^{\ast}_{\epsilon}) \times \bar{\Omega}$.
	Hence, the number
	\begin{equation}
		T_{\mathrm{max}, \epsilon} :=
		\left\{\begin{array}{cl}
			T^{\ast}_{\epsilon}, & \text{if } \hat{a} \circ z \equiv a \circ z \text{ in } [0, T^{\ast}_{\epsilon}) \times \bar{\Omega}, \\
			\min\big\{t \in [0, T_{\epsilon}^{\ast}) \,\big|\,
			a\big(z(t, x)\big) \not \in \operatorname{int}(J_{\epsilon}) \text{ for } x \in \bar{\Omega}\big\}, & \text{otherwise}
		\end{array}
		\right.
		\notag
	\end{equation}
	is well-defined and positive by Equation (\ref{EQUATION_SET_J_NONTRIVIAL}).
	For any sufficiently small $\epsilon > 0$, 
	denoting by by $(z_{\epsilon}, \theta_{\epsilon}, p_{\epsilon})$ the unique classical solution to Equation (\ref{EQUATION_NONLINEAR_PLATE_EQUATION_TRANSFORMED_NEW_NONLINEARITY_PDE_1}),
	(\ref{EQUATION_NONLINEAR_PLATE_EQUATION_TRANSFORMED_PDE_2})--(\ref{EQUATION_NONLINEAR_PLATE_EQUATION_TRANSFORMED_IC}) restricted onto $[0, T_{\mathrm{max}, \epsilon})$,
	we obtain an increasing sequence of closed sets $J_{\varepsilon}$ satisfying Equation (\ref{EQUATION_SET_J_NONTRIVIAL}) such that
	\begin{equation}
		T_{\mathrm{max}, J_{\epsilon}} \nearrow T_{\mathrm{max}} := \sup\big\{T_{\mathrm{max}, \epsilon} \,\big|\, 
		J_{\epsilon} \text{ satisfies Equation } (\ref{EQUATION_SET_J_NONTRIVIAL})\big\} \text{ as } \epsilon \searrow 0.
		\label{EQUATION_DEFINITION_OF_T_MAX}
	\end{equation}
	By construction, $(z_{J_{\epsilon}}, \theta_{J_{\epsilon}}, p_{J_{\epsilon}})$ solves the original problem 
	(\ref{EQUATION_NONLINEAR_PLATE_EQUATION_TRANSFORMED_PDE_1})--(\ref{EQUATION_NONLINEAR_PLATE_EQUATION_TRANSFORMED_IC}) for $t \in [0, T_{\mathrm{max}, J_{\epsilon}})$ and
	\begin{equation}
		(z_{J_{\epsilon'}}, \theta_{J_{\epsilon'}}, p_{J_{\epsilon'}}) \equiv (z_{J_{\epsilon}}, \theta_{J_{\epsilon}}, p_{J_{\epsilon}}) 
		\text{ on } [0, T_{\mathrm{max}, J_{\epsilon'}}) \text{ for } \epsilon' \geq \epsilon > 0. \notag
	\end{equation}
	Hence, letting for $t \in [0, T_{\mathrm{max}})$
	\begin{equation}
		(z, \theta, p)(t) := (z_{J_{\epsilon}}, \theta_{J_{\epsilon}}, p_{J_{\epsilon}})(t) 
		\text{ for any sufficiently small } \epsilon > 0 \text{ such that } T_{\mathrm{max}, J_{\epsilon}} > t, \notag
	\end{equation}
	we observe $(z, \theta, p)$ uniquely defines a classical solution to (\ref{EQUATION_NONLINEAR_PLATE_EQUATION_TRANSFORMED_PDE_1})--(\ref{EQUATION_NONLINEAR_PLATE_EQUATION_TRANSFORMED_IC})
	on $[0, T_{\mathrm{max}})$.
	Moreover, unless $T_{\mathrm{max}} = \infty$, we either have the blow-up
    \begin{equation}
        \sum_{m = 0}^{s} \big\|\partial_{t}^{m} z(t, \cdot)\big\|_{H^{s - m}(\Omega)}^{2} +
		\sum_{k = 0}^{s - 1} \big\|\partial_{t}^{k} \theta(t, \cdot)\big\|_{H^{s - k}(\Omega)}^{2} +
		\sum_{k = 0}^{s - 1} \big\|\partial_{t}^{k} p(t, \cdot)\big\|_{H^{s - 1 - k}(\Omega)}^{2} \to \infty
		\label{EQUATION_REDUCED_SYSTEM_BLOW_UP_AT_T_MAX_NAIVE}
	\end{equation}
	as $t\nearrow T_{\mathrm{max}}$ or/and the violation of ellipticity condition (\ref{EQUATION_ELLIPTICITY_VIOLATION_AT_T_MAX}) as $t \nearrow T_{\mathrm{max}}$.
	Indeed, if neither was the case, we could amend the set $J_{0}$ from Equation (\ref{EQUATION_IMAGE_OF_DOMAIN_OMEGA_WRT_Z_NOUGHT}) via
	\begin{equation}
		J_{0} := \big[\min_{x \in \bar{\Omega}} a(z(T_{\mathrm{max}}, x)), \max_{x \in \bar{\Omega}} a(z(T_{\mathrm{max}}, x))\big] \notag
	\end{equation}
	and repeat Step 5 to obtain a classical solution $(z_{\epsilon}, \theta_{\epsilon}, p_{\epsilon})$ existing beyond $T_{\mathrm{max}}$, 
	which would contradict Equation (\ref{EQUATION_DEFINITION_OF_T_MAX}).
	The overall uniqueness on $[0, T_{\max})$ follows from an energy estimate shown in Step 4.
	
	\noindent {\it Step 7: Improving the blow-up condition.}
	There remains to prove Equation (\ref{EQUATION_REDUCED_SYSTEM_BLOW_UP_AT_T_MAX_NAIVE}) is equivalent with (\ref{EQUATION_REDUCED_SYSTEM_BLOW_UP_AT_T_MAX}).
	This amounts to showing all norms in Definition \ref{DEFINITION_CLASSICAL_SOLUTION_Z_THETA_P} remain bounded -- 
	provided the norm in Equation (\ref{EQUATION_REDUCED_SYSTEM_BLOW_UP_AT_T_MAX}) stays finite.
	
	We argue by contradiction. Suppose $z \in C^{0}\big([0, T_{\max}], H^{s}(\Omega)\big)$.
	Then, multiplying Equations (\ref{EQUATION_NONLINEAR_PLATE_EQUATION_TRANSFORMED_PDE_1})--(\ref{EQUATION_NONLINEAR_PLATE_EQUATION_TRANSFORMED_IC})
	in $L^{2}\big((0, t) \times \Omega\big)$ with $\frac{1}{\gamma} z_{t}$, $\frac{1}{\gamma \eta} A \theta$ and $p$, using the fact
	$\|\nabla a(z)\|_{L^{\infty}((0, T_{\max}) \times \Omega)} < \infty$
	and proceeding similar to the energy estimate part of the proof of Theorem \ref{THEOREM_APPENDIX},
	we obtain $z, \theta \in C^{0}\big([0, T_{\max}], H^{1}_{0}(\Omega)\big)$, $z_{t}, p \in C^{0}\big([0, T_{\max}], L^{2}(\Omega)\big)$.
	Again, similar to the proof of Theorem \ref{THEOREM_APPENDIX}, formally differentiating Equations, we obtain
    \begin{align}
		\partial_{tt} z_{t} + \tfrac{1}{\gamma} a(z) A z_{t} - \tfrac{\alpha}{\gamma} A \theta_{t} + B \theta_{t} &= F'(z) z_{t} - \tfrac{1}{\gamma} a(z) z_{t} A z & &\text{ in } (0, \infty) \times \Omega,
		\label{EQUATION_NONLINEAR_PLATE_EQUATION_TRANSFORMED_BLOW_UP_PROOF_PDE_1} \\
		\beta \partial_{t} \theta_{t} + p_{t} + \alpha \partial_{t} z_{t} &= 0 & &\text{ in } (0, \infty) \times \Omega,
		\label{EQUATION_NONLINEAR_PLATE_EQUATION_TRANSFORMED_BLOW_UP_PROOF_PDE_2} \\
		\tau \partial_{t} p_{t} + p_{t} - \eta A \theta_{t} &= 0 & &\text{ in } (0, \infty) \times \Omega,
		\label{EQUATION_NONLINEAR_PLATE_EQUATION_TRANSFORMED_BLOW_UP_PROOF_PDE_3}
	\end{align}
	with the initial conditions given through the compatibility conditions.
	Since the right-hand side of Equation (\ref{EQUATION_NONLINEAR_PLATE_EQUATION_TRANSFORMED_BLOW_UP_PROOF_PDE_1}) is in $L^{2}\big(0, T_{\max}; L^{2}(\Omega)\big)$,
	we similarly obtain $z_{t}, \theta_{t} \in C^{0}\big([0, T_{\max}], H^{1}_{0}(\Omega)\big)$, $z_{tt}, p_{t} \in C^{0}\big([0, T_{\max}], L^{2}(\Omega)\big)$
	and, using Equations (\ref{EQUATION_NONLINEAR_PLATE_EQUATION_TRANSFORMED_PDE_1})--(\ref{EQUATION_NONLINEAR_PLATE_EQUATION_TRANSFORMED_IC}),
	along with the elliptic regularity of $a(z) A$, it follows $z, \theta \in C^{0}\big([0, T_{\max}], H^{2}(\Omega) \cap H^{1}_{0}(\Omega)\big)$,
	$p \in C^{0}\big([0, T_{\max}], H^{1}(\Omega)\big)$.
	Repeating this procedure iteratively up to the level $s - 1$ and using a regularization technique as the one in Theorem \ref{THEOREM_APPENDIX} to obtain an estimate at the level $s$,
	while observing that the right-hand side of differentiated equation always stays in $L^{2}\big(0, T_{\max}; L^{2}(\Omega)\big)$,
	we can show all of the remaining norms in Definition \ref{DEFINITION_CLASSICAL_SOLUTION_Z_THETA_P} are finite.
	Hence, the blow-up (\ref{EQUATION_REDUCED_SYSTEM_BLOW_UP_AT_T_MAX_NAIVE}) does not occur, which contradicts our assumption.
	Therefore, conditions (\ref{EQUATION_REDUCED_SYSTEM_BLOW_UP_AT_T_MAX_NAIVE}) and (\ref{EQUATION_REDUCED_SYSTEM_BLOW_UP_AT_T_MAX}) are equivalent.
	
	\noindent {\it Step 8: Showing continuity of the solution map.}
	The (Lipschitz) continuity of the solution map on the set $\mathscr{M}_{T, \varepsilon, N}$
	follows {\it mutatis mutandis} with an estimate analogous to that from Step 4.
\end{proof}

Recalling the equivalence Theorem \ref{THEOREM_SYSTEM_EQUIVALENCE} stating
\begin{equation}
    \notag
    w(t, \cdot) = A^{-1} z(t, \cdot) \quad \text{ and } \quad
    \mathbf{q}(t, \cdot) = \mathbf{q}^{0} + \nabla A^{-1} \operatorname{div} \mathbf{q}^{0} - \nabla A^{-1} p(t, \cdot),
\end{equation}
we get the local well-posedness in the class specified in Definition \ref{DEFINITION_CLASSICAL_SOLUTION_W_THETA_Q}
for the original system (\ref{EQUATION_QUASILINEAR_PDE_IN_W_THETA_AND_Q_1})--(\ref{EQUATION_IC_IN_W_THETA_AND_Q}) as claimed in Theorem \ref{THEOREM_LOCAL_EXISTENCE}.

\section{Global Existence and Long-Time Behavior: Proof of Theorems \ref{THEOREM_GLOBAL_EXISTENCE} and \ref{THEOREM_STABILITY}}
\label{SECTION_LONG_TIME_BEHAVIOR}

In this section, we restrict ourselves to the case $s = 3$ and prove that the local solutions to
Equations \eqref{EQUATION_QUASILINEAR_PDE_IN_W_THETA_AND_Q_1}--\eqref{EQUATION_IC_IN_W_THETA_AND_Q} 
(or, equivalently, \eqref{EQUATION_NONLINEAR_PLATE_EQUATION_PDE_1}--\eqref{EQUATION_NONLINEAR_PLATE_EQUATION_IC})
established in Theorem \ref{THEOREM_LOCAL_EXISTENCE} exist globally (i.e., $T_{\max} = \infty$) and their energy decays exponentially -- given the initial data are small enough in the lowest topology. It is worth pointing out that, throughout the proofs in this section, we operate with general $s$ and only put $s = 3$ at the very end to achieve the desired results. This demonstrates the crucialness of Assumption \ref{ASSUMPTION_GLOBAL_EXISTENCE} paired with the smallness of the initial data in the lower topology instead of the higher one. See also Remark \ref{remark_on_s} for details.

For technical convenience, in lieu of the functions $a(z)$ and $f(z, \nabla z)$ from Equation \eqref{EQUATION_NONLINEAR_PLATE_EQUATION_PDE_1}, 
throughout this Section, we will use the function $F(\cdot)$ defined via
\begin{equation} 
	\label{definition_F}
	F(z) = K'(0) z - K(z) 
\end{equation}
representing the remainder of the second-order Taylor expansion of $-K(\cdot)$ around $0$. Then, it follows from Assumption \ref{ASSUMPTION_GLOBAL_EXISTENCE} that
\begin{equation}        \label{F_properties}
    F'(0)=K'(0)-K'(0) = 0 \qquad \mbox{and} \qquad
    F''(0)=K''(0) = 0.
\end{equation}
As before, the operator $A$ denotes the negative Dirichlet-Laplacian and $z = -\triangle w = Aw$.

With this notation, the system \eqref{EQUATION_NONLINEAR_PLATE_EQUATION_PDE_1}--\eqref{EQUATION_NONLINEAR_PLATE_EQUATION_IC} becomes
\begin{subequations}
\begin{align}
(A^{-1} + \gamma I) z_{tt} +  A z  - \alpha A \theta &= AF(z)  &	\mbox{in } (0,\infty) \times \Omega 	\label{z_sys_1}\\
\beta \theta_t + p  + \alpha z_t &= 0 \label{z_sys_2} 	&	\mbox{in } (0,\infty) \times \Omega 	\\
\tau p_t + p - \eta A \theta &= 0 	& 		\mbox{in } (0,\infty) \times \Omega 	\label{z_sys_3} 	\\
z=\theta &= 0 	& 		\mbox{on } (0,\infty) \times \partial\,\Omega 	\label{z_sys_BC} 	\\
z(0,\cdot)=z^0, \ \ z_t(0,\cdot)=z^1, \ \ \theta(0,\cdot) = \theta^0,\ \  p(0,\cdot) &= p^0 	& 		\mbox{in } \Omega 	\label{z_sys_IC}
\end{align}
\end{subequations}
\emph{Notation:} For a local classic solution triple $(z,\theta,p)$ established in Theorem \ref{THEOREM_LOCAL_EXISTENCE}, recall the following topological solution spaces:
\begin{align*}
		z, \theta     &\in \Big(\bigcap_{m = 0}^{2} C^{m}\big([0, T], H^{3 - m}(\Omega) \cap H^{1}_{0}(\Omega)\big)\Big) \cap
		C^{3}\big([0, T], L^{2}(\Omega)\big), \\
		& \hspace*{2in} =: C^{0}\big([0, T], \calZ_3\big) \equiv C^{0}\big([0, T], \calT_3\big), \\
		p=\divq &\in \Big(\bigcap_{k = 0}^{2} C^{k}\big([0, T], H^{2 - k}(\Omega)\big)\Big) =: C^{0}\big([0, T], \calP_3\big), 
\end{align*}
where we denote
\[
\begin{aligned}
\calZ_3 
&=
\left\{(z, z_t, z_{tt}, z_{ttt}) \,|\, z \in H^3(\Omega) \cap H_0^1(\Omega), z_t \in H^2(\Omega) \cap H_0^1(\Omega), z_{tt} \in H^1(\Omega) \cap H_0^1(\Omega), z_{ttt} \in L_2(\Omega)\right\},	\\
\calT_3
&=
\left\{(\theta, \theta_t, \theta_{tt}) \,|\, \theta \in H^3(\Omega) \cap H_0^1(\Omega), \theta_t \in H^2(\Omega) \cap H_0^1(\Omega), \theta_{tt} \in H^1(\Omega) \cap H_0^1(\Omega), \theta_{ttt} \in L^2(\Omega) \right\}, 	\\
\calP_3
&=
\left\{(p, p_{t}, p_{tt}) \,|\, p \in H^2(\Omega), p_t \in H^1(\Omega),  p_{tt} \in L_2(\Omega)\right\}
\end{aligned}
\]
equipped with the natural product norms.
For instance,
\[
\|z(t)\|_{\calZ_3}^2=\|z(t)\|_{H^3(\Omega)}^2 + \|z_t(t)\|_{H^2(\Omega)}^2 + \|z_{tt}(t)\|_{H^1(\Omega)}^2 + \|z_{ttt}(t)\|_{L^2(\Omega)}^2,
\]
and 
\begin{equation} 		\label{topological_energy}
X(t) = \|z(t)\|_{\calZ_3}^2 + \|\theta(t)\|_{\calT_3}^2 + \|p(t)\|_{\calP_3}^2.
\end{equation}
For the sake of simplicity, we write $z(t)$ instead of $\big(z, z_t, z_{tt}, z_{ttt}\big)(t)$, etc.

In addition, to facilitate the application of multiplier techniques in this section, we introduce the weighted energies $E_k(t)$, $k=1,2,3$, as follows,
\begin{align}
E_1(t) &:= \dfrac{1}{2} \Big(\ltwo{A^{-1/2}z_{t}}^2 + \gamma\ltwo{z_{t}}^2 + \ltwo{A^{1/2} z}^2 
 + \beta \ltwo{A^{1/2} \theta}^2 + \dfrac{\tau}{\eta} \ltwo{p}^2 \Big)    \label{E1} \\
E_2(t) &:= \dfrac{1}{2} \Big(\ltwo{A^{-1/2}z_{tt}}^2 + \gamma\ltwo{z_{tt}}^2 + \ltwo{A^{1/2} z_{t}}^2  
 + \beta \ltwo{A^{1/2} \theta_{t}}^2 + \dfrac{\tau}{\eta} \ltwo{p_{t}}^2 \Big)  \label{E2} \\
E_3(t) &:= \dfrac{1}{2} \Big(\ltwo{A^{-1/2}z_{ttt}}^2 + \gamma\ltwo{z_{ttt}}^2 + \ltwo{A^{1/2} z_{tt}}^2 
 + \beta \ltwo{A^{1/2} \theta_{tt}}^2 + \dfrac{\tau}{\eta} \ltwo{p_{tt}}^2 \Big)  \label{E3} 
\end{align}
and
\begin{equation}    \label{X}
	E(t)= \big\|(z,z_t,z_{tt},\theta,\theta_t,p,p_t)(t,\cdot)\big\|_E^2 = E_1(t) + E_2(t) + E_3(t).
\end{equation}
Finally, the higher-order norms as parts of $X(t)$, but not bounded by $E(t)$, are defined as
\begin{equation}    \label{Y}
	Y(t)=\|z\|_{H^3(\Omega)}^2 + \|z_t\|_{H^2(\Omega)}^2 + \|\theta\|_{H^3(\Omega)}^2+\|\theta_t\|_{H^2(\Omega)}^2 + \|\theta_{ttt}\|_{L^2(\Omega)}^2 + \|p\|_{H^2(\Omega)}^2 + \|p_t\|_{H^1(\Omega)}^2.
\end{equation}
It is clear that $X(t)$ and $E(t) + Y(t)$ are equivalent, thus, we write 
\[
Y(t) = X(t) - E(t)
\]
for the convenience of proofs.


Last but not least, throughout this section, we use $\inner{\cdot,\cdot}$ to denote the $L^2(\Omega)$-inner product. By $\|u\|_p$ we denote $L^p(\Omega)$-norm of $u$. 
In what follows, we work with ``smooth'' solutions whose existence has been already guaranteed by Theorem  \ref{THEOREM_LOCAL_EXISTENCE}. Thus, formal PDE calculations performed below are well justified.

\begin{lemma}[{\it A priori} energy observability]		\label{global_lemma_1}
    Let Assumptions \ref{ASSUMPTION_LOCAL_EXISTENCE} and \ref{ASSUMPTION_GLOBAL_EXISTENCE} be satisfied. Then, for $T \in (0,T_{\max}]$,
\begin{align}
E(T) + C_1 \int_0^T E(t)\ \mathrm{d}t &\leq 
C_2 \left[ E(0) + [E_1(0)]^{(s-1)/4} [X(0)]^{(s+1)/4} + [E_1(T)]^{(s-1)/4} [X(T)]^{(s+1)/4}\right] 	\nonumber	\\ 
&+ C_3 \sum_{i = 1}^{N}\int_0^T [E_1(t)]^{\alpha_i}  [X(t)]^{\beta_i} \  \mathrm{d}t,  \label{Lemma_1_result}
\end{align}
for some $N \in \BN$ and $\alpha_i > 0, \beta_i > 1$, $i = 1, \dots, N$.
\end{lemma}

\begin{proof}


\emph{Step 1: Level 1 energy estimates.} We start with estimates of energies at Level 1 defined in \eqref{E1}. Thereafter, in Step 2, time differentiation of the system will lead to desired estimates at Levels 2 and 3.
\medskip

\emph{Step 1.1: Energy identity.}
We multiply Equations \eqref{z_sys_1}--\eqref{z_sys_3} with $z_t, A\theta$, and $p$, respectively, and then add up the (appropriately weighted) three identities to get
\begin{equation} 	\label{level_1_energy_identity}
E_1(T) + \int_0^T \dfrac{1}{\eta} \ltwo{p}^2 \  \mathrm{d}t = E_1(0) + \int_0^T \inner{AF(z), z_t}\  \mathrm{d}t,
\end{equation}
whence we also obtain the energies bound
\begin{align}
    \int_0^T \ltwo{p}^2  \mathrm{d}t \leq 
    C \Big( E_1(0) + \int_0^T \inner{AF(z), z_t}\  \mathrm{d}t \Big).  	\label{E-4-18} 	 
\end{align}

\emph{Step 1.2:}
We multiply \eqref{z_sys_2} with $z_t$, integrate the equation from 0 to $T$ and apply Young's inequality to deduce
\begin{equation} \label{E1_step1}
\dfrac{1}{2} \int_0^T \ltwo{z_t}^2   \mathrm{d}t
\leq
C\int_0^T \ltwo{p}^2  \mathrm{d}t - C \int_0^T \inner{\theta_t, z_t}  \mathrm{d}t. 
\end{equation}
To estimate the last term in \eqref{E1_step1}, we first rewrite \eqref{z_sys_1} as
\begin{equation} 		\label{E4-20}
z_{tt} = BAF(z) - \alpha BAz + BA\theta
\end{equation}
with the linear operator 
$$
B:=(A^{-1}+\alpha I)^{-1} \colon D(A^{\alpha}) \to D(A^{\alpha}) \text{ for any } \alpha \geq 0
$$ 
being bounded due to the norm invariance. In particular, by 
\cite[p. 194 or p. 203]{LaPoWa2017}, $B$ is a bounded, self-adjoint operator on $L^2(\Omega)$. We thus have 
$\ltwo{Bx} \leq C_{\gamma} \ltwo{x}$ and, more generally,
\begin{equation}\label{Bop}
\ltwo{A^{\alpha} Bz} \leq  C \ltwo{A^{\alpha} z} \text{ for any } \alpha \geq 0
\end{equation}
paralleled by the same estimate for the adjoint of $B$. 

Hence, after multiplying \eqref{E4-20} with $\theta$ and integrating by parts in time, we get
\begin{equation*} 
-\int_0^T \inner{z_t,\theta_t}  \mathrm{d}t + (z_t,\theta_t)\big|_0^T
=
\int_0^T \inner{BAF(z),\theta}  \mathrm{d}t -\alpha \int_0^T (BAz,\theta)  \mathrm{d}t + \int_0^T (BA\theta,\theta)  \mathrm{d}t,
\end{equation*}
or, by noticing that the operators $A$ and $B$ commute,
\begin{align}    \label{E4-21}
-\int_0^T \inner{z_t,\theta_t}  \mathrm{d}t
&\leq
C_1 \big(E_1(0) + E_1(T)\big) + \int_0^T \inner{BAF(z),\theta}  \mathrm{d}t \\
&+\epsilon_1 \int_0^T \ltwo{A^{1/2}z}^2  \mathrm{d}t + C_{\epsilon_1} \int_0^T \ltwo{A^{1/2}\theta}^2  \mathrm{d}t. \nonumber
\end{align}
Plugging \eqref{E4-21} into \eqref{E1_step1}, we see that
\begin{align}    \label{E4-22}
\dfrac{1}{2}\int_0^T \ltwo{z_t}^2  \mathrm{d}t
&\leq
C\int_0^T \ltwo{p}^2  \mathrm{d}t + C_1\big(E_1(0) + E_1(T)\big) \\
&+\epsilon_1 \int_0^T \ltwo{A^{1/2}z}^2  \mathrm{d}t + C_{\epsilon_1} \int_0^T \ltwo{A^{1/2}\theta}^2  \mathrm{d}t+ \int_0^T \inner{BAF(z),\theta}  \mathrm{d}t. \nonumber
\end{align}

\emph{Step 1.3:}
Next, we multiply \eqref{z_sys_1} with $z$ and do integration by parts in time to get
\begin{align*}
\alpha\int_0^T \ltwo{A^{1/2}z}^2  \mathrm{d}t
=&
-\int_0^T \inner{A^{-1}z_{tt},z}  \mathrm{d}t - \gamma \int_0^T \inner{z_{tt},z}  \mathrm{d}t   \nonumber 		\\
&+\int_0^T \inner{A\theta,z}  \mathrm{d}t -\int_0^T \inner{AF(z),z}  \mathrm{d}t 	\nonumber 	\\
=&
-\inner{A^{-1}z_{t},z}\big|_0^T + \int_0^T \inner{A^{-1}z_{t},z_t}  \mathrm{d}t  - \gamma\inner{z_{t},z}\big|_0^T \int_0^T + \gamma \ltwo{z_{t}}^2  \mathrm{d}t   \nonumber 		\\
&+\int_0^T \inner{A^{1/2}\theta,A^{1/2}z}  \mathrm{d}t -\int_0^T \inner{AF(z),z}  \mathrm{d}t 	\nonumber 	\\
\leq &\ 
C_2 \big(E_1(0) + E_1(T)\big)+ C_{\gamma} \int_0^T \ltwo{z_t}^2  \mathrm{d}t + C_{\alpha} \int_0^T \ltwo{A^{1/2} \theta}^2  \mathrm{d}t 	\nonumber \\
&+ \dfrac{\alpha}{2} \int_0^T \ltwo{A^{1/2}z}^2  \mathrm{d}t  - \int_0^T \inner{AF(z), z}  \mathrm{d}t,	
\end{align*}
or,
\begin{align}	 \label{E1_step3}
\dfrac{\alpha}{2} \int_0^T \ltwo{A^{1/2}z}^2  \mathrm{d}t &\leq  C_2 \big(E_1(0) + E_1(T)\big)		\\
&+ C_{\gamma} \int_0^T \ltwo{z_t}^2  \mathrm{d}t + C_{\alpha} \int_0^T \ltwo{A^{1/2} \theta}^2  \mathrm{d}t - \int_0^T \inner{AF(z), z}  \mathrm{d}t.	\nonumber
\end{align}

\emph{Step 1.4:}
We multiply \eqref{E4-22} with $4C_{\gamma}$ and add it to \eqref{E1_step3} to get (after cancellations and estimating $E_1(T)$ from the first energy estimate):
\begin{align} 	
C_{\gamma} \int_0^T \ltwo{z_t}^2  \mathrm{d}t &+ \dfrac{\alpha}{2} \int_0^T \ltwo{A^{1/2}z}^2  \mathrm{d}t
\leq C \int_0^T \ltwo{p}^2  \mathrm{d}t + C\big(E_1(0) + E_1(T)\big) \nonumber \\
&+ 4C_{\gamma} \epsilon_1 \int_0^T \ltwo{A^{1/2}z}^2  \mathrm{d}t 
+(4C_{\gamma} C_{\epsilon_1} + C_{\alpha}) \int_0^T \ltwo{A^{1/2} \theta}^2  \mathrm{d}t \label{E4-24} \\
&+ 4C_{\gamma}  \int_0^T \inner{BAF(z),\theta}  \mathrm{d}t -  \int_0^T \inner{AF(z),z}  \mathrm{d}t \nonumber,
\end{align}
where $C_3 = 4C_{\gamma}C_1 + C_2$.

Finally, we estimate $\displaystyle \int_0^T \ltwo{A^{1/2}\theta}^2\  \mathrm{d}t$. Multiplying \eqref{z_sys_3} with $\theta$, we get
\[
\int_0^T \eta \ltwo{A^{1/2}\theta}^2\  \mathrm{d}t = \tau \int_0^T \inner{p_t, \theta} \  \mathrm{d}t + \int_0^T \inner{p, \theta} \  \mathrm{d}t \leq -\tau \int_0^T \inner{p, \theta_t} \  \mathrm{d}t + \inner{p,\theta}\big|_0^T + \int_0^T \inner{p, \theta} \  \mathrm{d}t.
\]
After exploiting the Young's inequality and Poincar\'{e} \& Friedrichs' inequality, we arrive at
\begin{equation}			\label{global_E_18}
\int_0^T \ltwo{A^{1/2}\theta}^2\  \mathrm{d}t 
\leq
C\big(E_1(0) + E_1(T)\big)
+ C \int_0^T \ltwo{p}^2 \  \mathrm{d}t 
+ \epsilon_2 \int_0^T \ltwo{\theta_t}^2 \  \mathrm{d}t,
\end{equation}
where the last term can be estimated via \eqref{z_sys_2} using the multiplier $\theta_t$:
\begin{equation}		\label{global_E_19}
\dfrac{\beta}{2} \int_0^T \ltwo{\theta_t}^2\  \mathrm{d}t \leq C \int_0^T \ltwo{p}^2 \  \mathrm{d}t + C_{\beta}  \int_0^T \ltwo{z_t}^2 \  \mathrm{d}t.
\end{equation}
Plugging \eqref{global_E_18} and \eqref{global_E_19} into \eqref{E4-24}, we get
\begin{align} 	
C_{\gamma} \int_0^T \ltwo{z_t}^2  \mathrm{d}t &+ \dfrac{\alpha}{2} \int_0^T \ltwo{A^{1/2}z}^2  \mathrm{d}t + \int_0^T \ltwo{A^{1/2}\theta}^2  \mathrm{d}t 	\leq C \int_0^T \ltwo{p}^2  \mathrm{d}t + C\big(E_1(0) + E_1(T)\big) \nonumber \\
&+ 4C_{\gamma} \epsilon_1 \int_0^T \ltwo{A^{1/2}z}^2  \mathrm{d}t  +\Big((4C_{\gamma} C_{\epsilon_1} + C_{\alpha}+1) \dfrac{2C_{\beta}}{\beta}\epsilon_2 \Big) \int_0^T \ltwo{z_t}^2  \mathrm{d}t  \label{E_1_last_one} \\
&+ 4C_{\gamma}  \int_0^T \inner{BAF(z),\theta}  \mathrm{d}t -  \int_0^T \inner{AF(z),z}  \mathrm{d}t \nonumber,
\end{align}

\emph{Step 1.5:}
Now, in \eqref{E_1_last_one}, we first choose $\epsilon_1$  small enough such that $4C_{\gamma} \epsilon_1 < \frac{\alpha}{4}$, then second $\epsilon_2$  small enough such that $(4C_{\gamma} C_{\epsilon_1} + C_{\alpha}+1) \frac{2C_{\beta}}{\beta}\epsilon_2 < \frac{C_{\gamma}}{2}$.
Cancellations of terms on both sides of \eqref{E_1_last_one} then follow, which leads to
\begin{align*}
E_1(T) &+ \int_0^T \dfrac{C_{\gamma}}{2} \ltwo{z_t}^2 + \dfrac{\alpha}{4} \int_0^T \ltwo{A^{1/2}z}^2 + \ltwo{A^{1/2} \theta}^2 + \ltwo{p}^2  \mathrm{d}t \\
&\leq C_{\alpha,\beta,\gamma} E_1(0) + 4C_{\gamma}  \int_0^T \inner{BAF(z),\theta}  \mathrm{d}t -  \int_0^T \inner{AF(z),z}  \mathrm{d}t + \widetilde{C}_{\alpha,\beta,\gamma} \inner{AF(z),z_t}  \mathrm{d}t,
\end{align*}
or
\begin{align}    \label{E1_estimate}
E_1(T) &+ C_1 \int_0^T E_1(t)  \mathrm{d}t \\
&\leq C_2 E_1(0) + C_3  \int_0^T \inner{AF(z),z_t} + \inner{BAF(z),\theta}  \mathrm{d}t - \inner{AF(z),z}  \mathrm{d}t, \nonumber
\end{align}
where $C_1, C_2, C_3$ depend on $\alpha, \beta, $ and $\gamma$.
\medskip

\emph{Step 2: Level 2 and 3 energy estimates.} 
Recall from \eqref{E2} and \eqref{E3} that
\begin{align*}
E_2(t) &= \dfrac{1}{2} \Big(\ltwo{A^{-1/2}z_{tt}}^2 + \gamma\ltwo{z_{tt}}^2 + \alpha \ltwo{A^{1/2} z_{t}}^2 + \beta \ltwo{A^{1/2} \theta_{t}}^2 + \dfrac{\tau}{\eta} \ltwo{p_{t}}^2  \Big),  \\
E_3(t) &= \dfrac{1}{2} \Big(\ltwo{A^{-1/2}z_{ttt}}^2 + \gamma\ltwo{z_{ttt}}^2 + \alpha \ltwo{A^{1/2} z_{tt}}^2 +  \beta \ltwo{A^{1/2} \theta_{tt}}^2 + \dfrac{\tau}{\eta} \ltwo{p_{tt}}^2  \Big).
\end{align*}
To mimic the energy estimate \eqref{E1_estimate} for the system \eqref{z_sys_1}--\eqref{z_sys_3}, we perform a time-differentiation first. Denoting
\begin{equation}
G(z) = \partial_t F(z)=F'(z) z_t    
\qquad \mbox{and} \qquad
H(z) = \partial_t G(z)=F''(z) z_t^2 + F'(z) z_{tt},    
\end{equation}
we obtain higher-energy inequalities
\begin{align}    \label{E2_estimate}
E_2(T) &+ C_1 \int_0^T E_2(t)  \mathrm{d}t \\
&\leq C_2 E_2(0) + C_3  \int_0^T \inner{AG(z),z_{tt}} + \inner{BAG(z),\theta_{t}}  \mathrm{d}t - \inner{AG(z),z_{t}}  \mathrm{d}t \nonumber
\end{align}
and
\begin{align}    \label{E3_estimate}
E_3(T) &+ C_1 \int_0^T E_3(t)  \mathrm{d}t \\
&\leq C_2 E_3(0) + C_3  \int_0^T \inner{AH(z),z_{ttt}} + \inner{BAH(z),\theta_{tt}}  \mathrm{d}t - \inner{AH(z),z_{tt}}  \mathrm{d}t. \nonumber
\end{align}
Adding \eqref{E1_estimate}--\eqref{E3_estimate} together and using $X(t)= E_1(t) + E_2(t) + E_3(t)$ from \eqref{X} leads to
\begin{align} 
E(T) +  C_1 \int_0^T E(t)  \mathrm{d}t &\leq C_2 E(0) 
+ C_3 \Big( \int_0^T \big( \inner{AF(z),z_{t}} + \inner{BAF(z),\theta}  + \inner{AF(z),z} \big)  \mathrm{d}t \nonumber \\
 &+ \int_0^T \big(\inner{AG(z),z_{tt}}+ \inner{BAG(z),\theta_{t}} +  \inner{AG(z),z_{t}} \big)  \mathrm{d}t \label{X_estimate} \\
 &- \int_0^T \big(\inner{AH(z),z_{ttt}}+\inner{BAH(z),\theta_{tt}} + \inner{AH(z),z_{tt}} \big)  \mathrm{d}t\Big). \nonumber
\end{align}
The estimate above is the fundamental observability/stabilizability estimate which reconstructs the full integral of the energy  $\displaystyle \int_0^T E(t)\ \mathrm{d}t $ from the initial data $E(0)$ modulo nonlinear terms represented by $F(z)$. Clearly, such observability inequality captures the effect of the propagation of the damping in the system. Indeed, the  original system has only one source of dissipation -- the thermal flux $\mathbf{q}$.  Observability estimate demonstrates that this dissipation is propagated onto the remaining quantities: the vertical displacement and the thermal moment. 
In the linear case, such estimates lead at once to exponential decays of the energy valid for the entire system. A similar result, yet using Lyapunov's indirect method, has recently been obtained for the linear case in \cite{RiRaSeVi2018}.

\emph{Step 3: Superlinear estimates.}
Let us recall from \eqref{X} that $ E(t) = E_1(t) + E_2(t) + E_3(t)$. Our goal now is to estimate the nine different integrals containing nonlinear terms in \eqref{X_estimate}.
\medskip

Apart from the integral of $\inner{AH(z),z_{ttt}}$, eight of nine terms in \eqref{X_estimate} are estimated similarly. We will illustrate the estimates based on $\displaystyle \int_0^T \inner{AF(z),z_{t}}  \mathrm{d}t$ and $\displaystyle \int_0^T \inner{AG(z),z_{tt}}  \mathrm{d}t$ in Step 3.1 and 3.2 below and skip the rest. Regarding the highest-order term (in time and space combined), $\displaystyle \int_0^T \inner{AH(z),z_{ttt}}  \mathrm{d}t$, its estimation involves total differentiation to be demonstrated in Step 3.3 and 3.4 below.
\medskip

\emph{Step 3.0:}
We first prepare some technical estimates that will be frequently used for the remainder of this section. First, via Sobolev's embedding and interpolation inequalities, we have 
\begin{equation}		\label{est}
\begin{split}
\|z(t)\|^s_{\infty} 
&\leq 
c  \|z(t)\|^{s/2}_{H^2} \|z(t)\|^{s/2}_{H^1} \leq  c [E_1(t)]^{s/4} \|z(t)\|^{s/2}_{H^2}, 	 \\
\|\nabla z(t)\|_{4} 
&\leq 
c  \|z(t)\|^{3/4}_{H^2} \|z(t)\|_{H^1}^{1/4} \leq c  [E_1(t)]^{1/8} \|z(t)\|_{H^2}^{3/4},  	 \\
\|z(t)\|_{4}
&\leq 
c  \|z(t)\|_{H^1}^{3/4} \ltwo{z(t)}^{1/4}.
\end{split}
\end{equation}
In the calculations below, we shall use  $E_1$ as a shorthand for the norm of $z$ bounded from above by the $z$-component of  $E_1(t)$.
Similar convention applies to $E_2$ and $E_3$ as well.  

Further, recall $F(z) = K'(0)z-K(z) \in C^{s+1}(\BR, \BR)$ by Assumption \ref{ASSUMPTION_LOCAL_EXISTENCE} and $F'(0)=0$ and $F''(0)=K''(0)=0$ on the strength of Equation \eqref{F_properties}. Therefore, provided $|z| \leq M$ for some positive number $M$, we have the following bounds on the derivatives of $F(\cdot)$:
\begin{align}
    |F'(z)| &\leq c_M|z|^{s-1},  \label{F'_bound} \\ 
    |F''(z)| &\leq c_M|z|^{s-2},  \label{F''_bound}   \\
    |F'''(z)| + |F^{(4)}(z)| &\leq c_M,   \label{F''''_bound} 
\end{align}
with the constant $c_M$ depending on $M$. Due to the boundedness of the initial data (cf. Equation \eqref{EQUATION_GLOBAL_THEOREM_HIGEST_TOPOLOGY_NORM} or \eqref{X_bounded_proof}) as well as the temporal continuity of local solutions, we can invoke \eqref{F'_bound}--\eqref{F''''_bound} for $t \in [0,T_{\max})$. Later, we will show {\it a posteriori} that the solution is globally bounded (cf. Equation \eqref{global_result_in_proof}), hence Equations \eqref{F'_bound}--\eqref{F''''_bound} hold for any $t > 0$.

\medskip

\emph{Step 3.1:}
A direct computation of $AF(z),$ while exploiting the fact that $f(z)$ vanishes on the boundary and $A F(z) = -\triangle F(z)$, furnishes the identity
\begin{equation}        \label{AF_identity}
    AF(z)=F'(z)Az - F''(z)|\nabla z|^2.
\end{equation}
Hence, 
\begin{align}
    \int_0^T \left|\inner{AF(z),z_{t}}\right|  \mathrm{d}t     
& \leq
    \int_0^T \left|\inner{F'(z)Az,z_{t}}\right|  \mathrm{d}t + \int_0^T \left|\inner{F''(z)|\nabla z|^2,z_{t}}\right|  \mathrm{d}t   \nonumber \\
& \leq
    \int_0^T C_{\epsilon} \|z\|^{2(s-1)}_{\infty} \ltwo{Az}^2 +\epsilon  \ltwo{z_t}^2  \mathrm{d}t + C_{\epsilon}  \|z\|^{2(s-2)}_{\infty} \|\nabla z\|_{4}^4  \mathrm{d}t   \nonumber \\
& \leq   
   \int_0^T  C_{\epsilon} \|z\|^{2(s-1)}_{\infty} \ltwo{Az}^2 +\epsilon  \ltwo{z_t}^2  \mathrm{d}t + C_{\epsilon} \|z\|_{\infty}^{2s-4} \|z\|_{H^2}^3 \|z\|_{H^1}    \mathrm{d}t   \nonumber \\
   &\leq \epsilon \int_0^T E_1(t)  + C_{\epsilon} \Big[ [E_1(t)]^{(s-1)/2} \|z\|_{H^2}^{s-1 +2} + \|z\|_{H^1}^{s-2+1} \|z\|_{H^2}^{s-2 +3}\Big]  \mathrm{d}t  \nonumber \\
   &\leq  \epsilon \int_0^T E_1(t)  \mathrm{d}t  + C_{\epsilon} \int_0^T  E_1(t)^{(s-1)/2} \|z\|_{H^2}^{s+1}   \mathrm{d}t,  \nonumber
\end{align}
where we have used the estimates from \eqref{est}.
Recall that we can choose $\epsilon \ll 1$. Therefore, without loss of generality, $\epsilon C_3 \ll C_1$. This assumption enables us to dominate the integral $\displaystyle \epsilon \int_0^T E_1(t)  \mathrm{d}t$ by $\displaystyle C_1 \int_0^T X  \mathrm{d}t$ on the left-hand side of \eqref{X_estimate}. The same procedure can be repeated a finite number of times whenever a similar term appears while estimating other terms on the right-hand side of \eqref{X_estimate}.
\medskip

\emph{Step 3.2:}
Computing $AG(z)$ explicitly:
\begin{equation}        \label{AG_identity}
    AG(z)=A(F'(z) z_t) = -F'''(z) z_t |\nabla z|^2 - 2F''(z) (\nabla z \cdot \nabla z_t) + F''(z) z_t Az + F'(z)Az_t,
\end{equation}
we get
\begin{align}
    \int_0^T \left|\inner{AG(z),z_{tt}}\right|  \mathrm{d}t     
&\leq 
\int_0^T \left|\inner{F'''(z) z_t |\nabla z|^2,z_{tt}}\right|  \mathrm{d}t + 2\int_0^T \left|\inner{F''(z) (\nabla z \cdot \nabla z_t),z_{tt}}\right|  \mathrm{d}t  \nonumber \\
&+ \int_0^T \left|\inner{F''(z) z_t Az,z_{tt}}\right|  \mathrm{d}t  + \int_0^T \left|\inner{F'(z)Az_t,z_{tt}}\right|  \mathrm{d}t,    \label{E4-29}
\end{align}
which can be estimated term-by-term as  follows:
\begin{itemize}
\item
The last term of \eqref{E4-29} can be treated similarly as in Step 3.1:
$$ 
\begin{aligned}
\int_0^T \left|\inner{F'(z)Az_t,z_{tt}}\right|  \mathrm{d}t    
&\leq 
\int_0^T \|z\|_{\infty}^{s-1} \ltwo{Az_t} ~\ltwo{z_{tt}} 	\\
&\leq 
C \int_0^T E_1^{(s-1)/4}  E_2^{1/2} \zHtwo^{(s-1)/2}  \ztHtwo \, \mathrm{d}t.
\end{aligned}
$$ 

\item
The second and the third terms yield:
\begin{align*}
2\int_0^T &\left|\inner{F''(z) (\nabla z \cdot \nabla z_t),z_{tt}}\right|  \mathrm{d}t  
+ \int_0^T \left|\inner{F''(z) z_t Az,z_{tt}}\right|  \mathrm{d}t        \\
&\leq 
2 \int_0^T \|z\|_{\infty}^{s-2}  \ltwo{\nabla z \cdot \nabla z_t} \ltwo{z_{tt}}  \mathrm{d}t +  \int_0^T \|z\|_{\infty}^{s-2} \ltwo{z_t z_{tt}} \ltwo{Az}  \mathrm{d}t       \\
&\leq
2 \int_0^T\|z\|_{\infty}^{s-2}  \|\nabla z\|_{L^4}^2 \|\nabla z_t\|_{L^4}^2 E_2^{1/2}  \mathrm{d}t +  \int_0^T \|z\|_{\infty}^{s-2}  \|z_t\|_{L^4}^2 \|z_{tt}\|_{L^4}^2 \zHtwo  \mathrm{d}t \\
&\leq
C \int_0^T\|z\|_{\infty}^{s-2} \|z\|_{H^2}^{3/2} \|z\|^{1/2}_{H^1} \|z_t\|_{H^2}^{3/2}  \|z_t\|^{1/2}_{H^1}  E_2^{1/2} \\
&+
\|z\|_{\infty}^{s-2}  \|z_t\|^{3/2}_{H^1}  \|z_t\|_{2}^{1/2} \|z_{tt}\|_{H^1}^{3/2}  \|z_{tt}\|_{2}^{1/2}  \zHtwo  \mathrm{d}t \\
&\leq C \int_0^T \big( E_1^{(s-1)/4} E_2^{3/4} \zHtwo^{(s+1)/2} \|z_t\|_{H^2}^{3/2} + E_1^{(s-1)/4} E_2  E_3^{3/4}  \zHtwo^{s/2}\big) \mathrm{d}t.
\end{align*}

\item
The first term is treated similarly:
$$
\begin{aligned}
\int_0^T \left|\inner{F'''(z) z_t |\nabla z|^2,z_{tt}}\right|  \mathrm{d}t 
&\leq 
C \int_0^T \|\nabla z\|^{2}_{4} \ltwo{z_{tt}}   \mathrm{d}t 
\leq
C \int_0^T  E_1^{1/4}  \zHtwo^{3/2} E_2^{1/2}   \mathrm{d}t.
\end{aligned}
$$
\end{itemize}


\medskip

\emph{Step 3.3:}
Now, directly evaluating $AH(z)$, we estimate 
\begin{equation}        \label{estimate_highest}
\int_0^T \inner{AH(z), z_{ttt}}  \mathrm{d}t 
\leq 
\epsilon \int_0^T  \ltwo{z_{ttt}}^2  \mathrm{d}t 
+
C_{\epsilon}\int_0^T \ltwo{A(F''(z) z_t^2}^2  \mathrm{d}t
+
\int_0^T (A(F'(z)z_{tt}), z_{ttt})  \mathrm{d}t.
\end{equation}
For the second part, in view of
\begin{align}    \label{AH_part1}
    A(F''(z) z_t^2) &= F^{(4)}(z) |\nabla z|^2 z_t^2 + F'''(z) Az z_t^2
    + 4F'''(z) z_t (\nabla z \cdot \nabla z_t) \\
    &+2F''(z)|\nabla z_t|^2 + 2F''(z)z_t Az_t, \nonumber
\end{align}
we can estimate term-by-term as follows:
\begin{eqnarray*}		\label{AHpart1}
\ltwo{F^{(4)}(z) |\nabla z|^2 z_t^2}^2, 
&\leq&  
C  \|\nabla z\|_{L^4}^4 \|z_t\|^4_{\infty}
	\leq  C E_1^{1/2} \zHtwo^3 \ztHtwo^4,	\\
\ltwo{F'''(z) Az z_t^2}^2
&\leq&  
C  \zHtwo^2 \|z_t\|_{\infty}^4
	\leq C  \zHtwo^2 \ztHtwo^2 \|z_t\|_{H^1}^2, 		\\
&\leq&  C  \zHtwo^2 \ztHtwo^2 \ztHtwo \|z_t\|_{2} \leq E_1^{1/2}\zHtwo^2 \ztHtwo^3,		\\
\ltwo{F'''(z) z_t (\nabla z \cdot \nabla z_t)}^2
&\leq&
C\|z_t\|_{\infty}^2 \ltwo{\nabla z \cdot \nabla z_t)}^2 \leq CE_1^{1/4} \ztHtwo^{11/2},		\\
\ltwo{F''(z)z_t Az_t}^2
&\leq&
C\|z\|_{\infty}^{s-2} \|z_t\|_{\infty}^{2} \ztHtwo^2 \leq C E_1^{(s-2)/4} \zHtwo^{(s-2)/2} \ztHtwo^4.
%
%
%
%
\end{eqnarray*}

\medskip

\emph{Step 3.4:}
The third part contains the highest-order term among yhe nonlinear terms. We first observe that
\begin{align}
    \inner{A(F'(z)z_{tt}), z_{ttt}} 
    & = \inner{A^{1/2} (F'(z)z_{tt}), A^{1/2} z_{ttt}} \nonumber \\
    & =\inner{F''(z) z_{tt} A^{1/2}z, A^{1/2} z_{ttt}} + \inner{F'(z) A^{1/2} z_{tt}, A^{1/2} z_{ttt}}  \nonumber \\
    & = \inner{A^{1/2} \big(F''(z) z_{tt} A^{1/2}z\big), z_{ttt}} + \inner{F'(z) A^{1/2} z_{tt}, A^{1/2} z_{ttt}}.     \label{highest_1}
\end{align}
For the former term, we have 
\[
\inner{A^{1/2} \big(F''(z) z_{tt} A^{1/2}z\big), z_{ttt}} \leq c \ltwo{z_{ttt}} \Big(\ltwo{z_{tt} (A^{1/2} z)^2} + \ltwo{ F''(z) A^{1/2} z_{tt} A^{1/2} z} + \ltwo{ F''(z) z_{tt }Az } \Big).
\]
We estimate these terms similarly as before:
\begin{align*}
\ltwo{z_{tt} (A^{1/2}z)^2} & \leq \|z_{tt} \|_{6} \|\nabla z\|_{6}^2 \leq c\|z_{tt} \|_{H^1} \|\nabla z \|_{H^1}^2 \leq c E_3^{1/2}  \ltwo{\nabla z} \|\nabla z\|_{H^2} \leq c E_3^{1/2} E_1^{1/2} \|z\|_{H^3},  	\\
\ltwo{A^{1/2} z_{tt} A^{1/2}z} & \leq \|A^{1/2}z\|_{\infty} \ltwo{A^{1/2}z_{tt}} \leq c \|A^{1/2}z\|^{1/2}_{H^1} \|A^{1/2}z\|^{1/2}_{H^2} \ltwo{A^{1/2} z_{tt}},	\\
&\leq 
c E_3^{1/2} \|A^{1/2}z\|^{1/2}_{H^1} \|A^{1/2}z\|^{1/2}_{H^2}    	\\
\ltwo{z_{tt} Az} & \leq \|z_{tt}\|_{4} \|Az\|_{4} \leq c \|z_{tt}\|_{6}\|Az\|_{3}\leq  E_3^{1/2} \ltwo{Az}^{1/2}\|Az\|^{1/2}_{H^1}.
\end{align*}
Having accounted for $\|F''(z)\|_{\infty} \leq c E_1^{(s-2)/4} \|z\|_{H^2}^{(s-2)/2}$, we obtain
\begin{equation}
\inner{A^{1/2} \big(F''(z) z_{tt} A^{1/2}z\big), z_{ttt}} \leq c  E_3 \Big(E_1^{1/2} \|z\|_{H^3} + E_1^{(s-2)/4} \|z\|_{H^2}^{(s-1)/2} \|z\|_{H^3}^{1/2}  \Big).    
\end{equation}
Finally, the second term in \eqref{highest_1} can be cast into the form
\begin{equation}    \label{total_derivative}
    \inner{F'(z) A^{1/2} z_{tt}, A^{1/2} z_{ttt}} = \dfrac{1}{2} \partial_t \inner{F'(z) A^{1/2} z_{tt}, A^{1/2} z_{tt}} - \dfrac{1}{2} \inner{F''(z) z_t, (A^{1/2} z_{tt})^2}.
\end{equation}
Therefore, the time integral of \eqref{total_derivative} becomes
\begin{multline}    \label{highest_2}
\int_0^T \inner{F'(z) A^{1/2} z_{tt}, A^{1/2} z_{ttt}}\  \mathrm{d}t 
 \leq 
\|z\|_{\infty}^{s-1} \ltwo{ A^{1/2} z_{tt} }^2\Big|_0^T  + C\int_0^T \|z\|_{\infty}^{s-2} \|z_t\|_{\infty} \ltwo{A^{1/2} z_{tt}}^2  \  \mathrm{d}t		\\
\leq  
E_1^{(s-1)/4} \|z\|^{(s-1)/2}_{H^2} E_3\Big|_0^T + C \int_0^T E_1^{(s-2)/4}  \|z\|^{(s-2)/2}_{H^2} E_2^{1/4} \|z_t\|^{1/2}_{H^2} E_3\  \mathrm{d}t. 
\end{multline}
{\bf Step 3.5}. The estimates for the remaining superlinear terms produce similar results, and are somewhat simpler due to their higher regularity. 
Consider, for instance, the least regular term 
$\displaystyle \int_0^T \inner{ BAH(z), \theta_{tt}}  \mathrm{d}t $. Recalling $B$ is a bounded operator on $H^1_0(\Omega)$ and the fact it commutes with $A^{\alpha}$, we obtain
\begin{eqnarray}
\inner{ BAH(z), \theta_{tt}} = \inner{ A^{1/2} H(z), A^{1/2} B \theta_{tt} } \leq \epsilon E_3(t) + C_{\epsilon} \|A^{1/2} H(z)\|^2,
\end{eqnarray}
where
\begin{align*}
\|A^{1/2}H(z)\|
& \leq
\|F'(z)\|_{\infty} \ltwo{A^{1/2}z_{tt}} + \|F''(z)\|_{\infty} \left( \ltwo{z_{tt} \nabla z } + \ltwo{z_t \nabla z_t} \right) + \|F'''(z)\|_{\infty}  \ltwo{ z_t^2 \nabla z}	\\
& \leq
E_1^{(s-1)/4} \zHtwo^{(s-1)/2} E_3^{1/2} + E_1^{(s-2)/4} \zHtwo^{(s-2)/2}\left(  \|z\|_{H^3}  E_2^{1/2}  + \ltwo{z_t}^{1/4} \|z_t\|_{H^1} \|z_t\|_{H^2}^{3/4} \right) 	\\
& + C\|z_t\|_{H^2}^2 E_1^{1/2}.
\end{align*}
%
%



\emph{Step 4:} Combining Steps 3.1--3.5 together, we arrive at the final superlinear inequality:

\begin{align}
E(T) + C_1 \int_0^T E(t)\ \mathrm{d}t &\leq 
C_2 \Big(E(0) + E_1^{(s-1)/4}(0) X^{(s+1)/4}(0) + E_1^{(s-1)/4}(T) X^{(s+1)/4}(T)\Big) 	\nonumber	\\ 
&+ C_3 \sum_{i = 1}^{N}\int_0^T E_1^{\alpha_i}(t) X^{\beta_i}(t) \  \mathrm{d}t,  \label{E_estimate_super_final}
\end{align}
where $N \in \BN$, $\alpha_i > 0, \beta_i > 1$ for $i = 1, \dots, N$.
\end{proof}

\begin{lemma}[Super-linear estimate of the full energy $X(t)$]			
    \label{global_lemma_2}
    Under the same assumptions as in Lemma \ref{global_lemma_1}, the following inequality holds true:
    \begin{align}  
    X(T) + C_1 \int_0^T X(t)\  \mathrm{d}t &\leq 
    C_2\Big(E(0) + E_1^{(s-1)/4}(0) X^{(s+1)/4}(0)	\nonumber \\ 
    &+  
    \sum_{i = 1}^{N} \int_0^T E_1^{\alpha_i} (T) X^{\beta_i}(T)
    + 
    \sum_{i = 1}^{N} \int_0^T E_1^{\alpha_i}(t) X^{\beta_i}(t)\  \mathrm{d}t \Big), \label{final_X_estimates}    \end{align}
    where $N \in \BN$ and $\alpha_i > 0, \beta_i > 1$ for $i = 1, \dots, N$.
\end{lemma}

\begin{proof}
There remains to estimate the missing terms $Y(T)$ and $\displaystyle \int_0^T Y(t)\  \mathrm{d}t$.

\emph{Step 0:} 
On the strength of Equations \eqref{z_sys_1}--\eqref{z_sys_3}, we have the estimate 
\begin{align}	
\sum_{j=0}^{j=2} \Big(\|\partial_t^j z(T)\|_{H^1}^2 &+ \|\partial_t^j \theta (T)\|_{H^1}^2  +  \ltwo{\partial_t^j  p(T) }^2 \Big) + \ltwo{z_{ttt}(T)}^2 \nonumber \\
&+ \int_0^T \Big(\|\partial_t^j z(t)\|_{H^1}^2 + \|\partial_t^j \theta (t)\|_{H^1}^2  +  \ltwo{\partial_t^j  p(t) }^2 \Big) + \ltwo{z_{ttt}(t)}^2  \mathrm{d}t \label{A-4a} \\
&\leq C E(T) +  C\int_0^T E(t)\  \mathrm{d}t. \nonumber
\end{align}

\emph{Step 1: Space-regularity boost for $\theta$ and $p$.} 
From Equation \eqref{z_sys_2},
$$
\|p\|^2_{H^1}\leq C \big( \|\theta_t\||^2_{ H^1} + \|z_t\|^2_{H^1} \big) \leq C \big(E_1(t) + E_2(t)\big) \leq CE(t).
$$
Applying the time-derivative operator to the both sides of \eqref{z_sys_2}, we can see with \eqref{A-4a} that
\begin{equation}        \label{boost_p_t}
    \|p_{t}\|_{H^1}^2 
    \leq 
    C(\|\theta_{tt}\|_{H^1}^2 + \|z_{tt}\|_{H^1}^2)  
    \leq  c\big( E_3(t)  + E_3(t) \big) \leq C  E(t).
\end{equation}
Another temporal differentiation yields
$$
\ltwo{\theta_{ttt}(t)}^2 \leq C \big( \ltwo{p_{tt}}^2 + \ltwo{z_{ttt}}^2 \big) \leq C \big(E_2(t) +E_3(t)\big) 
\leq C E(t). 
$$
Similarly, a time-differentiation of \eqref{z_sys_3} gives
 \begin{equation}        \label{boost_theta_t}
     \|\theta_t\|_{H^2(\Omega)}^2 \leq C \ltwo{A\theta_{t}}^2 \leq C\big(\ltwo{p_{tt}}+\ltwo{p_t}\big)^2 \leq  C \big(E_2(t) + E_3(t)\big) \leq C E(t).
 \end{equation} 
Moreover, via \eqref{z_sys_3}, \eqref{boost_p_t} leads to the  estimate of the $H^3$-norm of $\theta$:
\begin{equation}        \label{boost_theta}
    \|\theta\|_{H^3(\Omega)}^2 \leq C\|A\theta\|_{H^1}^2
    \leq 
    C(\|p_{t}\|_{H^1}^2 + \|p\|_{H^1}^2)
    \leq C E(t).
\end{equation}
Finally, using  (\ref{z_sys_3}), we get
\begin{equation}\label{theta-t}
\|A \theta_t\|^2 \leq C \big(\| p_t\|^2 + \|p_{tt} \|^2\big) \leq C E(t). 
\end{equation}
Adding up the estimates above, we obtain the following extra regularity 
\begin{equation}\label{boost}
\ltwo{A \theta(t)}^2 + \ltwo{A \theta_t (t)}^2 + \ltwo{\theta_{ttt}(t)}^2+ \|\theta(t)\|^2_{H^3} + \|p(t)\|^2_{H_1} +\|p_t(t)\|^2_{H^1}  \leq C E(t).
\end{equation}

\emph{Step 2: Estimate of the $H^2$-norm of $z$.} This estimate involves nonlinear terms. 

We first recall the identity
\begin{equation}        \label{app_F_identity}
    AF(z)=F'(z)Az - F''(z)|\nabla z|^2.
\end{equation}
Plugging the above identity into \eqref{z_sys_1}, we observe
\begin{equation}    \label{app_Az_iden}
Az = (\gamma I + A^{-1})z_{tt} + \alpha A\theta + F'(z)Az - F''(z)|\nabla z|^2,
\end{equation}
which implies 
\begin{equation}    \label{app_Az_iden_estimate_1}
    \ltwo{Az}^2 \leq
     CE(t) + \|F'(z)\|^2_{\infty} \|Az\|^2 + \|F''(z)\|_{\infty}^2 \ltwo{|\nabla z|^2}^2
\end{equation}
or
\begin{equation}
\zHtwo^2 
\leq 
     C \left(E(t) + \|z\|^{2s-2}_{\infty} \|z\|^2_{H^2} +  \|z\|^{2s-4}_{\infty}  \|\nabla z\|_{4}^4 \right).
\end{equation}

To proceed, we invoke the embedding $H^{7/4}(\Omega) \hookrightarrow W^{1,4}(\Omega)$ for $d \leq 3$ and obtain 
$$
\|\nabla z\|_{4}^4 \leq C \|z\|^4_{H^{7/4}} \leq  \|z\|_{H^2}^3\|z\|_{H^1} \leq \epsilon \|z\|_{H^2}^6 + C_{\epsilon}\|z\|_{H^1}^2 \leq \epsilon \|z\|_{H^2}^6 + C_{\epsilon}E_1(t).
$$
Applying this inequality with the frequently used embedding \eqref{est} to \eqref{app_Az_iden_estimate_1}, we get
\begin{equation}\label{zh21}
\|z(t)\|_{H^2}^2 
\leq 
C \big( E(t) + E_1^{(s-1)/2} \|z\|_{H^2}^{s+1} + E_1^{(s-2)/2} \|z\|_{H^2}^{s+4} + E_1^{s/2} \|z\|_{H^2}^{s-2}\big)
\end{equation}

\emph{Step 3: $H^2$-norms of $z_t$ and $p$.} 
We start by applying the time derivative to both sides of \eqref{app_Az_iden} to obtain
\begin{equation}    \label{app_Az_t_iden}
    Az_t = (\gamma I + A^{-1})z_{ttt} + \alpha A\theta_t + F'(z)Az_t + F''(z) z_t Az - F'''(z) z_t|\nabla z|^2 - F''(z)(\nabla z \cdot \nabla z_t),
\end{equation}
This gives, after accounting for (\ref{theta-t}),
\begin{align*}    \label{app_Az_t_estimate}
    \ltwo{Az_t}^2 
    \leq & 
        \Big(\ltwo{(\gamma I + A^{-1})z_{ttt} + A\theta_t}^2 + \ltwo{F'(z) A z_t}^2 + \linfinity{F''(z)}^2 \ltwo{z_t Az}^2  \nonumber \\
        & + \linfinity{F'''(z)}^2 \ltwo{z_t|\nabla z|^2} + \linfinity{F''(z)}^2\ltwo{(\nabla z \cdot \nabla z_t)}^2 \Big)   \nonumber \\
    \leq & \ 
        C \left(E(t)+ \|z\|_{\infty}^{2s-2} \ltwo{Az_t}^2+ \|z\|_{\infty}^{2s-4}  \ltwo{z_t Az}^2 + C \ltwo{z_t|\nabla z|^2}^2 +  \|z\|_{\infty}^{2s-4} \ltwo{(\nabla z \cdot \nabla z_t)}^2 \right).
\end{align*}
A term-by-term estimation as before leads to
\begin{align*}
\ztHtwo^2 
&\leq  \ CE + E_1^{(s-1)/2} \zHtwo^{s-1} \ztHtwo^2 + E_1^{(s-2)/2} \zHtwo^{s} \|z\|_{H^1} \ztHtwo 		\\
&+ E_1^{1/2} \zHtwo^{3} \|z\|_{H^1} \ztHtwo + E_1^{(2s-3)/4} \zHtwo^{s-1/2} \|z\|_{H^1}^{1/2} \ztHtwo^{3/2}.
\end{align*}
Now, \eqref{z_sys_2} and \eqref{boost} yield the same bounds for the $H^2$-norm of $p$.

\emph{Step 4: $H^3$-norm of $z$.}
This time, we apply the $\nabla$-operator on both sides of \eqref{z_sys_1}, or \eqref{app_Az_iden}, to get
\begin{equation}    \label{app_A32z_iden}
    \nabla A z = \nabla(\gamma I + A^{-1})z_{tt} + \alpha \nabla A\theta + 3F''(z) Az \nabla z + F'(z)\nabla A z 
    - F'''(z) |\nabla z|^2 \nabla z
\end{equation}
and, thus,
\begin{align*}    \label{app_A32z_iden_final}
    \ltwo{\nabla A z}^2 
    \leq & \  CE +\|\theta\|_{H^3}^2 + \|z\|^{2s-4}_{\infty} \|Az\|^2_{4}\|\nabla z\|^2_{4} + \|z\|^{2s-2}_{\infty} \ltwo{\nabla Az}^2 + \|\nabla z\|_{\infty} \ltwo{|\nabla z|^2}^2.
\end{align*}
One last application of the embedding and interpolation inequalities leads to
\begin{equation}
\|z\|^2_{H^3} 
\leq  \ CE + E_1^{(s-2)/2} \zHtwo^{s} \|z\|_{H^3}^2 + E_1^{(s-1)/2} \zHtwo^{s-1} \|z\|_{H^3}^2  	
 + E_1^{1/2} \zHtwo^{3} \|z\|_{H^3}.
\end{equation}

\emph{Step 5:} 
Combining all of the previous steps with the estimate of $E(t)$ \eqref{E_estimate_super_final}, the desired result follows.
\end{proof}

\begin{theorem}
\label{final_proof}
Under the conditions of Lemma \ref{global_lemma_2}, in particular, Equation \eqref{final_X_estimates}, and the assumption
\begin{equation}        \label{X_bounded_proof}
    X(0) \leq M \quad \text{ for some } M > 0,
\end{equation}
there exists a $\delta>0$, depending on $M$ and $\rho$ (defined in Equation  \eqref{EQUATION_LOCAL_POSITIVITY_OF_K_PRIME}), 
such that, whenever $E_1(0) < \delta$, the local solution exists globally, i.e., $T_{\max} = \infty$. Moreover, the associated higher energy decays exponentially, namely, 
\begin{equation}        \label{global_result_in_proof}
X(t) \leq C e^{-\kappa t} X^{(s+3)/4}(0),    
\end{equation}
where $C, \kappa > 0$ and the constant $\delta$ is given in Equation \eqref{delta_choice} in the proof below. Hence, when $s=3$, Theorems \ref{THEOREM_GLOBAL_EXISTENCE} and \ref{THEOREM_STABILITY} are established.
\end{theorem}

\begin{proof}
After some minimal amendments, Equation \eqref{final_X_estimates} rewrites as
\begin{align}	
X(T) + C_1 \int_0^T X(t)\  \mathrm{d}t &\leq 
C_2\Big( \big(1 + E_1^{(s-1)/4}(0) \big) X^{(s+3)/4}(0) 	\nonumber \\ 
&+  
\sum_{i = 1}^{N} \int_0^T E_1^{\alpha_i}(T)  X^{\beta_i}(T)
+ 
\sum_{i = 1}^{N} \int_0^T E_1^{\alpha_i}(t)  X^{\beta_i}(t)\  \mathrm{d}t \Big).     \label{final_X_recall}
\end{align}

We first proceed with our energy estimate inequality. 
Let $C_1, C_2, C_3$ and $\alpha_1, \dots, \alpha_N$, $\beta_1, \dots, \beta_N$, $N$ be the constants from \eqref{final_X_recall}. Introduce the (smooth) functions
\begin{align}
f(x,y) &= y-C_2\sum_{i = 1}^{N} x^{\alpha_i} y^{\beta_i} = y\Big( 1 - C_2\sum_{i = 1}^{N} x^{\alpha_i} y^{\beta_i - 1}  \Big)		\label{function_f}		\\
g(x,y) &= C_1 y-C_2\sum_{i = 1}^{N} x^{\alpha_i} y^{\beta_i} = C_1y \Big( 1-\dfrac{C_2}{C_1} \sum_{i = 1}^{N} x^{\alpha_i} y^{\beta_i - 1}\Big).			\label{function_g}		
\end{align}
Recall $\beta_i - 1 > 0$. Now, we can express the function $f$ from Equation \eqref{final_X_recall} as
\begin{equation} 		\label{X_estimate_short_functions}
f\big(E_1(T),X(T)\big) + \int_0^T g\big(E_1(t),X(t)\big)\  \mathrm{d}t \leq C_2\big(1 + E_1^{(s-1)/4}(0) \big) X^{(s+3)/4}(0).
\end{equation}
We will apply a modified `barrier method' to show the desired result. 
Before doing so, we need an additional estimate on $E_1(t)$.

\emph{Step 1:} Refer to the $E_1$-inequality \eqref{E1_estimate}. We apply usual embedding and interpolation techniques to these `lower-order' terms. In particular,
\begin{equation}        \label{AF_z_t_E_1}
\begin{aligned}
\inner{AF, z_t} 
& \leq \|z\|_{\infty}^{s-1} \ltwo{Az} \ltwo{z_t} + \|z\|_{\infty}^{s-2} \ltwo{\nabla z}^2 		\\
& \leq CE_1^{(s+1)/4} \zHtwo^{(s+1)/2} + CE_1^{(s+2)/4} \zHtwo^{(s-2)/2} 				\\
& \leq CE_1^{(s+1)/4} \|z\|_{H_1}^{(s+1)/8} \|z\|_{H_3}^{3(s+1)/8} + 	CE_1^{(s+2)/4} \zHtwo^{(s-2)/2}		\\
& = CE_1^{3(s+1)/8}  \|z\|_{H_3}^{3(s+1)/8} + 	CE_1^{(s+2)/4} \zHtwo^{(s-2)/2}
\end{aligned}
\end{equation}
with $\frac{3(s+1)}{8}>1$ and $\frac{s+2}{4}>1$. The other two nonlinear terms can be estimated in exactly the same way.

Hence,
\begin{equation}		\label{final_E_1}
E_1(T) + C_3 \int_0^T E_1(t)\  \mathrm{d}t \leq 
C_4\Big(E_1(0) + \sum_{j = 1}^{N} \int_0^T E_1^{\gamma_j} (t)  X^{\tau_j} (t)\  \mathrm{d}t \Big),
\end{equation}
where $\gamma_j > 1$ and $\tau_j > 0$ for $j = 1, \dots, N$. Letting
\[
k(x,y) = C_3 x - C_4 \sum_{j=1}^{N} x^{\gamma_j} y^{\tau_j} = x \Big(C_3  - C_4 \sum_{j=1}^{N} x^{\gamma_j - 1} y^{\tau_j} \Big),
\]
Equation \eqref{final_E_1} can be expressed as
\begin{equation}		\label{final_E_1_short}
E_1 (T) + \int_0^T k\big(E_1(t), X(t)\big) \  \mathrm{d}t \leq C_4E_1(0).
\end{equation}

\emph{Step 2: Global existence.} Let $\big(z(0),\theta(0),p(0)\big)$ be the initial data triple and let $T_{\max}$ denote the maximal existence time of the local (smooth-in-time) solution. Our thrust is to show that, given a some small bound on $E_1(0)$, we have $T_{\max} = \infty$.

\emph{Step 2.1:}
Arguing by contradiction, suppose $T_{\max} < \infty$, and $X(T_{\max}) = \infty$. Letting $M_0 = (2C_2 + 1)M^{(s+3)/4}$, define
\begin{equation}			\label{T_star}
T^*=\sup\big\{t \in (0,T_{\max}] \,|\, X(s) \leq 2M_0 \mbox{ for any } s \in [0,t]\big\} \leq T_{\max}<\infty
\end{equation}
and
\begin{equation}			\label{T_dual_star}
T^{**}=\sup\Big\{t \in (0,T^*] \,\big|\,  k\big(E_1(s), X(s)\big) - \tfrac{C_3}{2} E_1(s) \geq 0 \mbox{ for all } s \in [0,t]\Big\} \leq T^*.
\end{equation}
Equation \eqref{T_star} suggests that 
\begin{equation} 		\label{X_T_star}
X(T^*)=2M_0,
\end{equation}
otherwise, we can extend $T^*$ due to the continuity of $X(t)$ in time.

We also observe that $T^{**} > 0$ if $E_1(0)$ is small enough. Indeed, since the function $k(E_1(s),X(s)) - \frac{C_3}{2}E_1(s)$ vanishes for $E_1(s) = 0$, it suffices to prove that it is increasing with respect to $E_1(s)$. A quick calculation shows that for $|y| \leq 2M_0$, 
\[
\frac{\partial}{\partial x} \left(k(x,y) - \frac{C_3}{2} x\right)  \geq \dfrac{C_2}{2} - C_4 \sum_{j=1}^{N} \gamma_j \ x^{\gamma_j - 1} (2M_0)^{\tau_j}.
\]
Thanks to the fact that $\gamma_j - 1 > 0$, the right hand side function has a unique positive root. Let $\delta_1$ be this root, namely, 
\begin{equation} 		\label{smallness_1}
\delta_1>0 \mbox{ is the number such that }\ \  \dfrac{C_2}{2} - C_4 \sum_{j=1}^{N} \gamma_j \ \delta_1^{\gamma_j - 1} (2M_0)^{\tau_j} = 0.
\end{equation}

Hence, when $0<E_1(0) < \delta_1$, $k(E_1(0),X(0)) - \frac{C_3}{2}E_1(0)$ is strictly positive, making $T^{**}$ strictly positive, as well.

\emph{Step 2.2:}
We claim that $T^{**} = T^*$. Again, arguing by contradiction, we would otherwise have $k\big(E_1(T^{**}), X(T^{**}\big) - \frac{C_3}{2} E_1(T^{**}) = 0$ due to the temporal continuity of the latter function. Thus,
\[
C_3  - C_4 \sum_{j=1}^{N} E_1^{\gamma_j - 1} (T^{**}) X^{\tau_j}  (T^{**}) = \tfrac{C_3}{2}
\qquad \mbox{and} \qquad
k\big(E_1(t), X(t)\big) \geq \tfrac{C_3}{2} E_1(t) \mbox{ for any } t \in [0,T^{**}].
\]
Solving the first equation for $E_1(T^{**})$, we get a unique solution (depending on $X(T^{**})$) subsequently denoted as $a \in \BR$. 

Using the second inequality and plugging it into \eqref{final_E_1_short}, we obtain
\begin{equation}
E_1 (T^{**}) + \dfrac{C_3}{2} \int_0^{T^{**}} E_1(t)\  \mathrm{d}t \leq C_4E_1(0).
\end{equation}
Imposing the condition
\begin{equation}			\label{smallness_2}
E_1(0) < \delta_2 = \dfrac{a}{2 C_4},
\end{equation} 
we would arrive at the contradiction: $E_1 (T^{**}) < \frac{a}{2} \neq a$.

\emph{Step 2.3:}
Thus, $T^{**} = T^*$ and $E_1(t) \leq C_4 E_1(0)$ for any $t \in [0,T^*]$. Now, assuming
\begin{equation}        \label{smallness_2.5}
    E_1(0)<\delta_3 = \frac{\rho}{C_4}
\end{equation}
with $\rho$ defined in \eqref{EQUATION_LOCAL_POSITIVITY_OF_K_PRIME},
we have $K'(z) > 0$ for any $t\in [0,T^*]$ meaning the positive ellipticity holds true.

Next, we proceed to Equation \eqref{X_estimate_short_functions}. Recalling the definitions in Equations \eqref{function_f} and \eqref{function_g} as well as the bound $X(t) \leq 2M_0$ in $[0,T^*]$, we let $\delta_4$ be a small number such that
\begin{equation} 		\label{smallness_3}
    1 - C_2\sum_{i = 1}^{N} \delta_4^{\alpha_i} (2M_0)^{\beta_i - 1} \geq \dfrac{1}{2}
    \qquad \mbox{and} \qquad
    1-\dfrac{C_2}{C_1} \sum_{i = 1}^{N} \delta_4^{\alpha_i} (2M_0)^{\beta_i - 1} \geq \dfrac{1}{2}.
\end{equation}
Hence, $f\big(E_1(T^*),X(T^*)\big) \geq \frac{1}{2}X(T^*)$ and $g\big(E_1(t),X(t)\big) \geq \frac{1}{2} X(t)$ for any $t \in [0,T^*]$. Together with Equation \eqref{X_estimate_short_functions}, we get
\begin{equation} 		\label{Global_E_63}
X(T^*) + \int_0^{T^*} X(t)\  \mathrm{d}t \leq 2C_2\big(1 + E_1^{(s-1)/4}(0) \big) X^{(s+3)/4}(0).
\end{equation}
Since $X(t) \leq 2M_0$, it follows that
\begin{equation} 		\label{smallness_4}
2C_2\big(1 + E_1^{(s-1)/4}(0) \big) X^{(s+3)/4}(0) \leq 2C_2 \cdot 2X(0) <2M_0 \qquad
\mbox{if } E_1(0) < 1.
\end{equation}
In summary, selecting
\begin{equation}			\label{delta_choice}
\delta = \min\{\delta_1, \delta_2, \delta_3, 1\}
\end{equation} 
with $\delta_i$, $i=1, \dots, 4$ defined in Equations \eqref{smallness_1}, \eqref{smallness_2}, \eqref{smallness_2.5}, \eqref{smallness_3}, respectively, we get $X(T^*) < 2M_0$ contradicting Equation \eqref{X_T_star}.
Hence, we have $T_{\max} = \infty$ and $2M_0$ as a global bound for $X(t)$.

\emph{Step 3: Exponential stability.}
Equation \eqref{Global_E_63} now becomes
\[
X(T) + \int_0^{T} X(t)\  \mathrm{d}t \leq 4C_2X^{(s+3)/4}(0)
\]
for any $T > 0$. A standard Datko \& Pazy-type propagation argument (cf. \cite[p. 211]{LaPoWa2017}) furnishes the exponential decay of $X(T)$.
\end{proof}

\begin{remark}
\label{remark_on_s}
We conclude this section by pointing out that, in order a smallness condition on the lower energy $E_1(0)$ (in lieu of the smallness of the full energy $X(0)$) to be sufficient, Assumption \ref{ASSUMPTION_GLOBAL_EXISTENCE} is critical. To see this, consider, for instance, the lower-order nonlinearity $K(z)=z-z^2 +\alpha z^3$. Thus, $F('z)=2z - 3\alpha z^2$ and $F''(0)=2$. With this example, the superlinear inequality for $X(t)$ (cf. Equation \eqref{final_X_recall}) still holds, but the other one for $E_1(t)$ (cf. Equation \eqref{final_E_1}) fails. Indeed, the inner product $\inner{AF,z_t}$ in \eqref{AF_z_t_E_1} contains $\inner{F''(z) |\nabla z|^2, z_t}$, which can (optimally) by estimated as
\[
\inner{F''(z) |\nabla z|^2, z_t} \leq C\ltwo{|\nabla z|^2}\ltwo{z_t} = C\|\nabla z\|_{4}^2\ltwo{z_t} \leq C \|z\|_{H^1}^{1/2} \|z\|_{H^2}^{3/2} \|z_t\|_{L^2} \leq C E_1^{3/4} X^{3/4}.
\]
with a constant bound from $|F''(z)|$. Therefore, one cannot obtain a superlinear bound for $E_1$. Nevertheless, if seeking  for a weaker result with a smallness condition on $X(0)$ instead (for instance, as in \cite{LaPoWa2017} or \cite{RaUe2017}), the nonlinearity $K(z)=z-z^2$ would be admissible. This should not be surprising as our Assumption \ref{ASSUMPTION_GLOBAL_EXISTENCE} is weakened, accordingly.
\end{remark}

\begin{appendix}

\section{Model Description (\texorpdfstring{$d = 2$}{d=2})} 
\label{SECTION_MODEL_DESCRIPTION}
    In this appendix section, we derive a macroscopic model for a prismatic thermoelastic plate of uniform thickness $h > 0$ and constant material density $\rho > 0$.
    As a reference configuration, we choose the domain $\mathcal{B}_{h} := \Omega \times (-\tfrac{h}{2}, \tfrac{h}{2})$ of $\mathbb{R}^{3}$,
    where the bounded domain $\Omega \subset \mathbb{R}^{2}$ is referred to as the mid-plane of the plate.
    The governing equation for the elastic part can be adopted {\it verbatim} from \cite[Section 2]{LaPoWa2017},
    while the thermal equations will closely resemble \cite[Chapter 1.5.2]{Po2011}.
    The material comprising the plate is assumed incompressible and elastically/thermally homogeneous and isotropic.
    While geometric and thermal nonlinearities are going to be discarded in what follows,
    our model will incorporate a nonlinearity in the hypoelastic material law
    allowing us to adequately describe such genuinely nonlinear elastic materials as rubber, liquid crystal elastomers, biological tissues, etc.
    The model will be obtained as a sort of Taylor's expansion of the 3D equations of thermoelasticity as $h \to 0$ (cf. \cite[Chapter 1]{LaLi1988}).

    \subsection{Thermoelastic Plate as a 3D Prismatic Body}
    We begin with formulating the system of nonlinear 3D thermoelasticity.
    To this end, in the Lagrangian coordinates, let $\mathbf{U} = (U_{1}, U_{2}, U_{3})^{T}$ denote the displacement vector, $T$ stand for the absolute temperature
    and $\mathbf{Q} = (Q_{1}, Q_{2}, Q_{3})^{T}$ be the associated heat flux.
    Denote by $T_{0} > 0$ a constant reference temperature rendering the body free of elastic and/or thermal stresses.
    Further, let $S$ be the entropy and 
    \begin{equation}
    	\boldsymbol{\sigma} = (\sigma_{ij})_{1 \leq i, j \leq 3}
    	\mbox{\quad and \quad}
    	\boldsymbol{\varepsilon}^{\mathrm{elast}} = \tfrac{1}{2}\big(\nabla \mathbf{U} + (\nabla \mathbf{U})^{T}\big) \notag
    \end{equation}
    stand for the first Piola \& Kirchhoff stress tensor and the infinitesimal Cauchy strain tensor, respectively.
    The total stress tensor is assumed to decompose into elastic and thermal stresses via
    \begin{equation}
    	\boldsymbol{\sigma} = \boldsymbol{\sigma}^{\mathrm{elast}} - \boldsymbol{\sigma}^{\mathrm{therm}}.
    	\label{EQUATION_LINEAR_STRESS_DECOMPOSITION}
    \end{equation}
    In absence of external body forces and heat sources, according to \cite[p. 142]{AmBeMi1984} and \cite[Chapter 1]{LaLi1988},
    the momentum and energy balance equations are expressed as
    \begin{subequations}
    \begin{align}
    	\rho \mathbf{U_{tt}} + \operatorname{div} \boldsymbol{\sigma} &= \mathbf{0} \quad \text{ in } (0, \infty) \times \mathcal{B}_{h}, 
    	\label{EQUATION_BALANCE_OF_MOMENTUM_3D_BASIC} \\
    	T S_{t} + \operatorname{div} \mathbf{Q} &= 0 \quad  \text{ in } (0, \infty) \times \mathcal{B}_{h}.
    	\label{EQUATION_BALANACE_OF_THERMAL_ENERGY_3D_BASIC}
    \end{align}
    \end{subequations}
    
    With conservation/continuity Equations (\ref{EQUATION_BALANCE_OF_MOMENTUM_3D_BASIC})--(\ref{EQUATION_BALANACE_OF_THERMAL_ENERGY_3D_BASIC}) at hand, we proceed to constitutive relations
    relating the stress tensor to the strain tensor/displacement gradient, the entropy to the temperature and the heat flux to the temperature gradient.
    As previously mentioned, following \cite[p. 142]{AmBeMi1984}, we assume the material is incompressible and isotropic 
    and the elastic stress and strain directional tensors (viz. \cite[Chapter 1, \S 6]{Il2004}) coincide.
    Next, we postulate a hypoelastic relation between the stress tensor $\boldsymbol{\sigma}^{\mathrm{elast}}$ and the strain tensor $\boldsymbol{\varepsilon}^{\mathrm{elast}}$.
    Due to material isotropy and incompressibility, following \cite[p. 42]{Il2004}, such hypoelastic relation can be described as
    \begin{equation}
    	\sigma_{\mathrm{int}}^{\mathrm{elast}} = \kappa(\varepsilon_{\mathrm{int}}^{\mathrm{elast}})
    	\label{EQUATION_HYPOELASTIC_LAW_ELASTIC_PART}
    \end{equation}
    where the elastic strain and stress intensities
    \begin{align*}
        \varepsilon_{\mathrm{int}}^{\mathrm{elast}} &= \tfrac{\sqrt{2}}{3} \Big((\operatorname{tr} \boldsymbol{\varepsilon}^{\mathrm{elast}})^{2}
        - \operatorname{tr}\big((\boldsymbol{\varepsilon}^{\mathrm{elast}})^{2}\big)\Big), &
        \sigma_{\mathrm{int}}^{\mathrm{elast}} &= \tfrac{\sqrt{2}}{3} \Big((\operatorname{tr} \boldsymbol{\sigma}^{\mathrm{elast}})^{2} 
    	- \operatorname{tr}\big((\boldsymbol{\sigma}^{\mathrm{elast}})^{2}\big)\Big)
    \end{align*}
    are the second invariants (i.e., properly scaled `second' eigenvalues) of $\boldsymbol{\varepsilon}^{\mathrm{elast}}$ and $\boldsymbol{\sigma}^{\mathrm{elast}}$, respectively.
    
    \begin{remark}
        The elastic stress-strain response $\kappa(\cdot)$ must naturally satisfy $\kappa(0) = 0$ and is typically measured experimentally, 
        for example, through a tensile test experiment.
        Being usually estimated through a statistical regression procedure, $\kappa(\cdot)$ can be reliably estimated only over a compact range of arguments.
        For this and many other reasons (such as possible material failure, etc.), the behavior of $\kappa(\cdot)$ at infinity should be viewed as a mathematical idealization.
        Nonetheless, assuming $\kappa(\cdot)$ is defined globally, in addition to satisfying $\lim_{|s| \to \infty} \big|\kappa(s)\big| = \infty$,
        the function needs to be globally positive to give rise to a signed elastic energy function $W = W(\nabla \mathbf{U})$.
        This is a reasonable assumption -- both physically and mathematically -- widely adopted in the Theory of Finite Elasticity (cf. \cite[Chapter 1]{Og2001}).
        When the global positivity of $\kappa(\cdot)$ is violated (see, e.g., a class of hypoelastic laws proposed in \cite[Equation (6)]{AmBeMi1984}),
        mathematical difficulties are artificially created putting unnecessarily constraints on the magnitude of the displacement gradient.
    \end{remark}
    
    For the thermal stresses and strains, following \cite[Chapter 1.6]{LaLi1988}, we adopt the linear isotropic homogeneous law
    \begin{equation}
    	\boldsymbol{\sigma}^{\mathrm{therm}} = 3 B \boldsymbol{\varepsilon}^{\mathrm{therm}},
    	\label{EQUATION_HYPOELASTIC_LAW_THERMAL_PART}
    \end{equation}
    where $B$ is the bulk modulus. Letting $E > 0$ and $\nu \in \big(-1, \tfrac{1}{2}\big)$ denote the Young's modulus and Poisson's ratio,
    since the material is assumed incompressible, we would formally obtain $\nu = \tfrac{1}{2}$ rendering the bulk modulus $B = \tfrac{E}{3(1 - 2\nu)}$ infinite.
    In fact, as recently demonstrated in \cite{MoDoRo2008}, this singularity does not occur experimentally as $B$ remains bounded (and even decreases!) as  $\nu \nearrow \tfrac{1}{2}$.
    Similarly, instead of hitting $0$ at $\nu = \tfrac{1}{2}-$, the shear modulus $G$ remains positive -- even though several magnitudes smaller than $B$.
    It has further been shown that $\frac{\partial B}{\partial \nu}(\tfrac{1}{2}-) = 0$.
    Hence, without violating the incompressibility condition, we can assume $\nu < \tfrac{1}{2}$.
    Linearizing $\kappa(\cdot)$ around $0$ (cf. \cite[Equation 13]{AmBeMi1984}) and using the classic definition of structural rigidity
    \begin{equation*}
        \frac{\kappa'(0) h^{3}}{9} = \frac{E h^{3}}{12 (1 - \nu^{2})},
    \end{equation*}
    we obtain $E = \frac{4 (1 - \nu^{2})}{3} \kappa(0)$, which leaves the Poisson's ratio $\nu$ a free parameter.
    
    Further, with $\tau = T - T_{0}$ representing the relative temperature, our thermal linearity and isotropy assumptions suggest
    \begin{equation}
    	\boldsymbol{\varepsilon}^{\mathrm{therm}} = \alpha \tau \mathbf{I}_{3 \times 3},
    	\label{EQUATION_THERMAL_STRAIN_MATERIAL_LAW}
    \end{equation}
    where $\alpha > 0$ denotes the thermal expansion coefficient (cf. \cite[p. 29]{LaLi1988}),
    while a linear approximation of the entropy relations \cite[Chapter 1]{No1986} around $T = T_{0}$ reads as
    \begin{equation}
    	S = \gamma \operatorname{tr}\big(\varepsilon^{\mathrm{elast}}\big) + \tfrac{\rho c}{T_{0}} \tau,
    	\label{EQUATION_ENTROPY_LINEARIZED}
    \end{equation}
    where $c > 0$ is the heat capacity and $\gamma = 3 B \alpha$.
    
    Finally, invoking the Cattaneo's law of relativistic heat conduction, we obtain
    \begin{equation}
        \label{EQUATION_CATTANEO_LAW}
        \tau_{0} \mathbf{Q}_{t} + \mathbf{Q} - \lambda_{0} \nabla \tau = \mathbf{0},
    \end{equation}
    where $\tau_{0} > 0$ is a relaxation time (not to be confused with the temperature $\tau$)
    and $\lambda_{0} > 0$ is the heat conductivity number.
    Compared to the classic Fourier's law (i.e., $\tau_{0} = 0$), the Cattaneo's law has a hyperbolic nature
    giving rise to the so-called `second sound' effect and having a finite signal propagation speed.
    While often being quantitatively indistinguishable from the Fourier's law,
    the Cattaneo's law becomes critical at small time-space scales and/or large heat pulse amplitudes
    as the latter is the case in laser cleaning (see, e.g., \cite{BaIv2014, QiXuGuo2013} and references therein) and numerous other applications \cite{Cha1998}, etc.
    Since the plate is thin in the $x_{3}$-direction, it is legitimate to approximate the equation for $Q_{3}$ in (\ref{EQUATION_CATTANEO_LAW}) via the Fourier's law and obtain
    \begin{equation}
        \label{EQUATION_CATTANEO_LAW_MODIFIED}
        \tau \mathbf{P} \mathbf{Q}_{t} + \mathbf{Q} - \lambda_{0} \nabla \tau = \mathbf{0} \quad \text{ with }
        \mathbf{P} := \begin{pmatrix} 1 & 0 & 0 \\ 0 & 1 & 0 \\ 0 & 0 & 0 \end{pmatrix} 
    \end{equation}
    or, equivalently,
    \begin{align*}
    	\tau Q_{i} + Q_{i} = -\lambda_{0} \partial_{x_{i}} \tau, \quad i = 1, 2, \qquad Q_{3} = -\lambda_{0} \partial_{x_{3}} \tau,
    \end{align*}
    while keeping the genuine Cattaneo's law for $Q_{1}$ and $Q_{2}$-components
    
    Combining Equations (\ref{EQUATION_LINEAR_STRESS_DECOMPOSITION}), 
    (\ref{EQUATION_HYPOELASTIC_LAW_THERMAL_PART})--(\ref{EQUATION_CATTANEO_LAW}) and plugging them 
    into Equations (\ref{EQUATION_BALANCE_OF_MOMENTUM_3D_BASIC})--(\ref{EQUATION_BALANACE_OF_THERMAL_ENERGY_3D_BASIC}), we arrive at
    \begin{subequations}
    \begin{align}
    \label{EQUATION_BALANCE_OF_MOMENTUM_3D_TRANSFORMED} 
    	\rho \mathbf{U}_{tt} + \operatorname{div} \boldsymbol{\sigma}^{\mathrm{elast}} + \gamma \nabla \tau
    	&= \mathbf{0} \quad \text{ in } (0, \infty) \times \mathcal{B}_{h}, \\
    	\label{EQUATION_BALANACE_OF_THERMAL_ENERGY_3D_TRANSFORMED}
    	\rho c \tau_{t} + \operatorname{div} \mathbf{Q}  + \gamma T_{0} \operatorname{div} \mathbf{U}_{t}
    	&= 0 \quad \text{ in } (0, \infty) \times \mathcal{B}_{h}, \\
    	\label{EQUATION_CATTANEO_LAW_3D_TRANSFORMED}
    	\tau_{0} \mathbf{P} \mathbf{Q}_{t} + \mathbf{Q} - \lambda_{0} \nabla \tau
    	&=  \mathbf{0} \quad \text{ in } (0, \infty) \times \mathcal{B}_{h}
    \end{align}
    \end{subequations}
    with $\boldsymbol{\sigma}^{\mathrm{elast}} = \boldsymbol{\sigma}^{\mathrm{elast}}(\nabla \mathbf{U})$ implicitly given 
    via Equation (\ref{EQUATION_HYPOELASTIC_LAW_ELASTIC_PART}) and the tensor alignment assumption.
    The equations of 3D dynamical thermoelasticity (\ref{EQUATION_BALANCE_OF_MOMENTUM_3D_TRANSFORMED})--(\ref{EQUATION_CATTANEO_LAW_3D_TRANSFORMED})
    are the starting point of our further plate modeling procedure.
    
    \subsection{The Averaging Procedure}
    
    Following \cite{LaLi1988} and neglecting the in-plane displacements, we adopt the Kirchhoff \& Love's structural assumption of undeformable normals:
    \begin{equation}
    	\begin{split}
    		U_{1}(x_{1}, x_{2}, x_{3}) &= -x_{3} w_{x_{1}}(x_{1}, x_{2}), \quad
    		U_{2}(x_{1}, x_{2}, x_{3})  = -x_{3} w_{x_{2}}(x_{1}, x_{2}), \\
    		U_{3}(x_{1}, x_{2}, x_{3}) &= \phantom{-x_{3}}w(x_{1}, x_{2}),
    	\end{split}
    	\label{EQUATION_KIRCHHOFF_LOVE_STRUCTURAL_ASSUMPTIONS}
    \end{equation}
    where $w$ is referred to as the bending component or the vertical displacement.
    Practically speaking, Equation (\ref{EQUATION_KIRCHHOFF_LOVE_STRUCTURAL_ASSUMPTIONS}) means 
    that the linear filaments, which were perpendicular to the mid-plane before deformation, are mandated to remain straight and perpendicular to the deformed mid-plane.
    Hence, the dynamics of the deflection vector $\mathbf{U}$ is reduced to that of the bending component $w$.
    
    Motivated by \cite[Chapter 1.6]{LaLi1988}, we introduce the thermal component
    \begin{equation}
    	\theta(x_{1}, x_{2}) = \frac{12 \alpha}{h^{3}} \int_{-h/2}^{h/2} x_{3} \tau \mathrm{d}x_{3} \notag
    \end{equation}
    as the $x_{3}$-moment of the thermal strain $\alpha \tau$,
    which, in turn, on the strength of Equation (\ref{EQUATION_THERMAL_STRAIN_MATERIAL_LAW}), is proportional to the $x_{3}$-moment of the temperature $\tau$.
    Here, the normalization factor is obtained as a reciprocal of $h^{3}/12 = \int_{-h/2}^{h/2} x_{3} \mathrm{d}x_{3}$.
    Similarly, following \cite[Chapter 1.5]{Po2011}, let
    \begin{equation}
    	\mathbf{q}(x_{1}, x_{2}) = \frac{12}{h^{3}} \int_{-h/2}^{h/2} x_{3} \begin{pmatrix} Q_{1} \\ Q_{2} \end{pmatrix} \mathrm{d}x_{3}.
    	\notag
    \end{equation}
    
    Proceeding as \cite[Section 2]{LaPoWa2017}, Equation (\ref{EQUATION_BALANCE_OF_MOMENTUM_3D_TRANSFORMED}) can be reduced to
    \begin{equation}
        \label{EQUATION_KIRCHHOFF_LOVE_MODELING_DERIVED_EQ_1}
    	\rho h w_{tt} - \tfrac{\rho h^{3}}{12} \triangle w_{tt} + \triangle K(\triangle w) 
    	+ D \tfrac{1 + \nu}{2} \triangle \theta = 0 \quad \text{ in } (0, \infty) \times \Omega.
    \end{equation}
    Here, $D = \frac{E h^{3}}{12(1 - \nu^{2})}$ is referred to as the flexural rigidity.
    The nonlinear response $K(\cdot)$ is obtained from $\kappa(\cdot)$ by means of
    \begin{equation}
    	\label{def_K_TS}
    	K(s) = \tfrac{h^{2}}{\sqrt{3}} \big[I \kappa\big]\big(\tfrac{h s}{\sqrt{3}}\big), \quad \text{ where }
    	\big[I f\big](s) = s^{-2} \int_{0}^{s} \xi f(\xi) \mathrm{d}\xi \text{ for } s \in \mathbb{R} \backslash \{0\}.
    \end{equation}
    In contrast to \cite{AmBeMi1984},
    the rotational inertia $\triangle w_{tt}$-term is not neglected here allowing for an adequate description of thicker plates
    than those accounted for by the standard theory.
    
    Multiplying Equation (\ref{EQUATION_BALANACE_OF_THERMAL_ENERGY_3D_TRANSFORMED}) with $\tfrac{12 x_{3}}{h^{3}}$ and integrating over $x_{3}$,
    we obtain
    \begin{equation*}
        \frac{12}{h^{3}} \sum_{i = 1}^{2} \partial_{x_{i}} \int_{-h/2}^{h/2} x_{3} \mathbf{q} \mathrm{d} x_{3}
        + \frac{12}{h^{3}} \int_{-h/2}^{h/2} x_{3} \frac{\partial Q_{3}}{\partial x_{3}} \mathrm{d} x_{3}
        + \frac{12 \rho c}{h^{3}} \partial_{t} \int_{-h/2}^{h/2} x_{3} \tau \mathrm{d} x_{3} - \gamma T_{0} \triangle w_{t}
        = 0
    \end{equation*}
    and, thus,
    \begin{equation}
        \label{EQUATION_THERMAL_EQUATION_AVERAGED}
        \operatorname{div} \mathbf{q} + \frac{12}{h^{3}} \int_{-h/2}^{h/2} x_{3} \frac{\partial Q_{3}}{\partial x_{3}} \mathrm{d} x_{3}
        + \frac{\rho c}{\alpha} \theta_{t} - \gamma T_{0} \triangle w_{t} = 0.
    \end{equation}
    Integrating by parts and using Equation (\ref{EQUATION_CATTANEO_LAW_MODIFIED}), we compute
    \begin{align}
        \notag
        \int_{-h/2}^{h/2} x_{3} \frac{\partial Q_{3}}{\partial x_{3}} \mathrm{d} x_{3} &=
        - \int_{-h/2}^{h/2} Q_{3} \mathrm{d} x_{3} + x_{3} Q_{3} \Big|_{-h/2}^{h/2} \\
        \label{EQUATION_X_3_Q_3_INTEGRATED_BY_PARTS}
        &= \lambda_{0} \int_{-h/2}^{h/2} \frac{\partial \tau}{\partial x_{3}} \mathrm{d} x_{3} + x_{3} Q_{3} \Big|_{-h/2}^{h/2}
        = \big(\lambda_{0} \tau + x_{3} Q_{3}\big) \Big|_{-h/2}^{h/2} \\
        \notag
        &= \lambda_{0} \big(\tau(t, x_{1}, x_{2}, \tfrac{h}{2}) - \tau(t, x_{1}, x_{2}, -\tfrac{h}{2})\big) +    
        \tfrac{h}{2} \big(Q_{3}(t, x_{1}, x_{2}, \tfrac{h}{2}) + Q_{3}(t, x_{1}, x_{2}, -\tfrac{h}{2})\big).
    \end{align}
    
    Following \cite[p. 30]{LaLi1988}, we assume the Newton's cooling law is applied to plate's upper and lower faces:
    \begin{equation}
        \label{EQUATION_Q_3_ON_UPPER_LOWER_FACES}
        Q_{3}(t, x_{1}, x_{2}, \tfrac{h}{2}) = \lambda_{1} \tau(t, x_{1}, x_{2}, \tfrac{h}{2}), \qquad
        Q_{3}(t, x_{1}, x_{2}, -\tfrac{h}{2}) = -\lambda_{1} \tau(t, x_{1}, x_{2}, -\tfrac{h}{2})
    \end{equation}
    for some $\lambda_{1} > 0$.
    Using the Taylor's expansion
    \begin{equation*}
        \tau(t, x_{1}, x_{2}, x_{3}) = \tau_{0}(t, x_{1}, x_{2}) + x_{3} \tau_{1}(t, x_{1}, x_{2})
    \end{equation*}
    and observing $\theta = \alpha \tau_{1}$, we can write
    \begin{equation}
        \label{EQUATION_TAU_FROM_LOWER_TO_UPPER}
        \tau(t, x_{1}, x_{2}, x_{3})\big|_{-h/2}^{h/2} = h \tau_{1}(t, x_{1}, x_{2}) = \tfrac{h}{\alpha} \theta(t, x_{1}, x_{2}).
    \end{equation}
    The combination of Equations (\ref{EQUATION_THERMAL_EQUATION_AVERAGED})--(\ref{EQUATION_TAU_FROM_LOWER_TO_UPPER}) furnishes
    \begin{equation}
        \label{EQUATION_KIRCHHOFF_LOVE_MODELING_DERIVED_EQ_2}
        \tfrac{\rho c}{\alpha} \tau_{t} + \tfrac{12}{\alpha h^{2}} \big(\lambda_{0} + \tfrac{h\lambda_{1}}{2}\big) \theta + \operatorname{div} \mathbf{q} - \gamma T_{0} \triangle w_{t} = 0.
    \end{equation}
    
    Multiplying the equations for $Q_{1}, Q_{2}$ in (\ref{EQUATION_CATTANEO_LAW_MODIFIED}) with $\tfrac{12 x_{3}}{h^{3}}$ and integrating over $x_{3}$, we get
    \begin{equation}
        \label{EQUATION_KIRCHHOFF_LOVE_MODELING_DERIVED_EQ_3}
        \tau_{0} \mathbf{q}_{t} + \mathbf{q} - \tfrac{\lambda_{0}}{\alpha} \nabla \theta = \mathbf{0}.
    \end{equation}
    
    Combining Equations (\ref{EQUATION_KIRCHHOFF_LOVE_MODELING_DERIVED_EQ_1}), (\ref{EQUATION_KIRCHHOFF_LOVE_MODELING_DERIVED_EQ_2}),
    (\ref{EQUATION_KIRCHHOFF_LOVE_MODELING_DERIVED_EQ_3}), we arrive at
    \begin{subequations}
    \begin{align}
        \label{EQUATION_KIRCHHOFF_LOVE_MODELING_FINAL_EQ_1}
    	\rho h w_{tt} - \tfrac{\rho h^{3}}{12} \triangle w_{tt} + \triangle K(\triangle w) 
    	+ D \tfrac{1 + \mu}{2} \triangle \theta &= 0 \quad \text{ in } (0, \infty) \times \Omega, \\
    	\label{EQUATION_KIRCHHOFF_LOVE_MODELING_FINAL_EQ_2}
    	\tfrac{\rho c}{\alpha} \partial_{t} \theta + \operatorname{div} \mathbf{q} + \tfrac{12}{\alpha h^{2}} \big(\lambda_{0} + \tfrac{h\lambda_{1}}{2}\big) \theta + \gamma T_{0} \triangle w_{t} &= 0
    	\quad \text{ in } (0, \infty) \times \Omega, \\
    	\label{EQUATION_KIRCHHOFF_LOVE_MODELING_FINAL_EQ_3}
    	\tau_{0} \mathbf{q}_{t} + \mathbf{q} - \tfrac{\lambda_{0}}{\alpha} \nabla \theta &= \mathbf{0}
        \quad \text{ in } (0, \infty) \times \Omega.
    \end{align}
    \end{subequations}
    
    Various boundary conditions can be adopted for Equations (\ref{EQUATION_KIRCHHOFF_LOVE_MODELING_FINAL_EQ_1})--(\ref{EQUATION_KIRCHHOFF_LOVE_MODELING_FINAL_EQ_3})
    (cf. \cite[Chapter 2]{Am1970}, \cite[Chapter 4]{Il2004}, \cite[Chapter 1]{LaLi1988}, \cite{LaTri1998.1, LaTri1998.2, LaTri1998.3, LaTri1998.4}).
    In this paper, we consider a simply supported plate held at the reference temperature on the boundary $\partial \Omega$:
    \begin{equation}
    	w = \triangle w = \theta = 0 \text{ in } (0, \infty) \times \partial \Omega. \notag
    \end{equation}
    For the sake of convenience, outside of this Section, the constants in Equation
    (\ref{EQUATION_KIRCHHOFF_LOVE_MODELING_FINAL_EQ_1})--(\ref{EQUATION_KIRCHHOFF_LOVE_MODELING_FINAL_EQ_3})
    will be renamed and/or normalized to obtain the mathematically more convenient system
    (\ref{EQUATION_QUASILINEAR_PDE_IN_W_THETA_AND_Q_1})--(\ref{EQUATION_QUASILINEAR_PDE_IN_W_THETA_AND_Q_3}).

	\section{Well-Posedness for Linearized Equations} 
	\label{APPENDIX_SECTION_WELL_POSEDNESS}
	
	The following well-posedness result is based on an extension of Kato's \cite{Ka1985} solution theory for abstract time-dependent evolution equations
	developed by Jiang \& Racke \cite[Appendix A]{JiaRa2000}.
	The arguments presented below are an adaptation of \cite[Appendix A.1]{LaPoWa2017}.
	In contrast to \cite{LaPoWa2017}, Equations (\ref{EQUATION_LINEARIZED_SYSTEM_PDE_1})--(\ref{EQUATION_LINEARIZED_SYSTEM_IC_2})
	comprise a hyperbolic system so that Equation (\ref{EQUATION_LINEARIZED_SYSTEM_PDE_1}) for $z$
	does not decouple from Equations (\ref{EQUATION_LINEARIZED_SYSTEM_PDE_2})--(\ref{EQUATION_LINEARIZED_SYSTEM_PDE_3}) for $\theta, p$.
	This makes the analysis more complicated as all components need to be treated simultaneously.
	This can be explained by the fact that Equations (\ref{EQUATION_LINEARIZED_SYSTEM_PDE_1})--(\ref{EQUATION_LINEARIZED_SYSTEM_PDE_3})
	inherit hyperbolic natural from the original plate system (\ref{EQUATION_QUASILINEAR_PDE_IN_W_THETA_AND_Q_1})--(\ref{EQUATION_QUASILINEAR_PDE_IN_W_THETA_AND_Q_3}).
	
	Let $\Omega \subset \mathbb{R}^{d}$ be a bounded domain with a $C^{s}$-boundary $\partial \Omega$ for some $s \geq \lfloor \tfrac{d}{2}\rfloor + 2$.
	Further, let $T > 0$ be arbitrary, but fixed.
	Throughout this appendix as well as the proof of Theorem \ref{THEOREM_LOCAL_EXISTENCE}, as before, $H^{0}_{0}(\Omega) \equiv H^{0}(\Omega) := L^{2}(\Omega)$
	and $\bar{D}^{n}$ is the time-space gradient defined in Equation (\ref{EQUATION_OPERATOR_D_N}).
	
	Further, let $\phi_{\delta} := \exp\big(\frac{1}{1 - (\cdot/\delta)^{2}}\big) \mathds{1}_{(-\delta, \delta)}$
	denote the one-dimensional Friedrichs' mollifier with a `window size' $\delta > 0$.
	For any $L^{1}$-function $w \colon [0, T] \times \Omega \to \mathbb{R}$, consider the $C^{\infty}(\Omega)$-approximation (cf. \cite[Chapters 8 and 9]{SchuKaHoKa2012}) of $w$:
	\begin{equation}
        \label{EQUATION_APPENDIX_CONVOLUTION}
		w_{\delta}(t, \cdot) = \int_{0}^{T} \phi_{\delta}(t - s) w(s, \cdot) \mathrm{d}s \quad \text{ for } \quad t \in [0, T] \quad \text{ in } \quad \Omega.
	\end{equation}
	The following result by Jiang \& Racke \cite[Lemma A.12]{JiaRa2000} will be used in the sequel.
	\begin{lemma}
		\label{LEMMA_MOLLIFIER_PROPERTIES}
		For arbitrary $a \in C^{1}\big([0, T], L^{\infty}(\Omega)\big)$ and $w \in C^{0}\big([0, T], L^{2}(\Omega)\big)$
		and any sufficiently small $\varepsilon > 0$, there holds
        \begin{equation*}
            \lim_{\delta \to 0} \int_{\varepsilon}^{T - \varepsilon} \big\|\partial_{t}\big((aw)_{\delta}(t, \cdot) - aw_{\delta}(t, \cdot)\big)\big\|_{L^{2}(\Omega)}^{2} \mathrm{d}t = 0.
		\end{equation*}
	\end{lemma}
	
	Consider the following linear system with time- and space-dependent coefficients:
	\begin{subequations}
	\begin{align}
		z_{tt}(t, x) - \bar{a}_{ij}(t, x) \partial_{x_{i}} \partial_{x_{j}} z(t, x) - \tfrac{\alpha}{\gamma} A \theta + B \theta &= \bar{f}(t, x) &&\text{for } (t, x) \in (0, T) \times \Omega,
		\label{EQUATION_LINEARIZED_SYSTEM_PDE_1} \\
		\beta \theta_{t}(t, x) + p(t, x) + \alpha z_{t}(t, x) &= 0 &&\text{for } (t, x) \in (0, T) \times \Omega,
		\label{EQUATION_LINEARIZED_SYSTEM_PDE_2} \\
		\tau p_{t}(t, x) + p(t, x) - \eta A \theta(t, x) &= 0 &&\text{for } (t, x) \in (0, T) \times \Omega,
		\label{EQUATION_LINEARIZED_SYSTEM_PDE_3} \\
		z(t, x) = 0, \quad \theta(t, x) &= 0 && \text{for } (t, x) \in [0, T] \times \partial \Omega,
		\label{EQUATION_LINEARIZED_SYSTEM_BC} \\
		z(0, x) = z^{0}(x), \quad z_{t}(0, x) &= z^{1}(x) &&\text{for } x \in \Omega,
		\label{EQUATION_LINEARIZED_SYSTEM_IC_1} \\
		\theta(0, x) = \theta^{0}(x), \quad p(0, x) &= p^{0}(x) &&\text{for } x \in \Omega.
		\label{EQUATION_LINEARIZED_SYSTEM_IC_2}
	\end{align}
	\end{subequations}
	Here, $B$ is a bounded linear operator on all $H^{s}(\Omega)$, $s \geq 0$, and $H^{s}(\Omega) \cap H^{1}_{0}(\Omega)$, $s \geq 1$, spaces
	and $A$ is the negative Dirichlet-Laplacian, which should not be confused with the generator $\mathcal{A}(t)$ 
	defined in Equation (\ref{EQUATION_OPERATOR_FAMILY_LINEAR_WAVE_EQUATION}) below.
	
	\begin{assumption}[{cf. \cite[Appendix]{LaPoWa2017}}]
		\label{ASSUMPTION_APPENDIX}
		Let $s \geq \lfloor\frac{d}{2}\rfloor + 2$ be a fixed integer
		and let $\gamma_{0}, \gamma_{1}$ be positive numbers.
		Assume the following conditions are satisfied.
		\begin{enumerate}
			\item \emph{Coefficient symmetry}:
			$\bar{a}_{ij}(t, x) = \bar{a}_{ji}(t, x)$ for $(t, x) \in [0, T] \times \bar{\Omega}$.
			
			\item \emph{Coefficient regularity}:
			$\bar{a}_{ij} \in C^{0}\big([0, T] \times \bar{\Omega}\big)$ and
			\begin{equation}
                \notag
				\partial_{x_{k}} \bar{a}_{ij} \in L^{\infty}\big(0, T; H^{s - 1}(\Omega)\big), \quad
				\partial_{t}^{m} \bar{a}_{ij} \in L^{\infty}\big(0, T; H^{s - 1 - m}(\Omega)\big)
				\quad \text{ for } m = 1, 2, \dots, s - 1.
			\end{equation}
			
			\item \emph{Coercivity}: For $z \in H^{1}_{0}(\Omega)$ and $t \in [0, T]$,
			$\|z\|_{H^{1}(\Omega)}^{2} \leq
            \gamma_{0} \big\langle \bar{a}_{ij}(t, \cdot) \partial_{x_{i}} z, \partial_{x_{j}} z\big\rangle_{L^{2}(\Omega)}$.
			
			\item \emph{Elliptic regularity}:
			For $m = 0, 1, \dots, s - 2$, $z(t, \cdot) \in H^{1}_{0}(\Omega)$ and
			$\bar{a}_{ij}(t, \cdot) \partial_{x_{i}} \partial_{x_{j}} z(t, \cdot) \in H^{m}(\Omega)$ for a.e. $t \in [0, T]$ implies
			$u(t, \cdot) \in H^{m + 2}(\Omega)$ and
			\begin{equation}
				\|z(t, \cdot)\|_{H^{m + 2}(\Omega)} \leq \gamma_{1}\Big(\|\bar{a}_{ij}(t, \cdot) \partial_{x_{i}} \partial_{x_{j}} z(t, \cdot)\|_{H^{m}(\Omega)} +
				\|z(t, \cdot)\|_{L^{2}(\Omega)}\Big) \text{ for a.e. } t \in [0, T]. \notag
			\end{equation}
			
			\item \emph{Right-hand side regularity}:
			For $m = 0, 1, \dots, s - 2$,
			\begin{equation}
				\partial_{t}^{m} \bar{f} \in C^{0}\big([0, T], H^{s - 2 - m}(\Omega)\big), \quad
				\partial_{t}^{s - 1} \bar{f} \in L^{2}(0, T; L^{2}(\Omega)\big).
				\notag
			\end{equation}
			
			\item \emph{Compatibility conditions}:
			For $k, l, m = 0, 1, \dots, s - 1$,
			\begin{align*}
				\bar{z}^{m} \in H^{s - m}(\Omega) \cap H^{1}_{0}(\Omega), \quad \bar{z}^{s} \in L^{2}(\Omega), \quad
				\bar{\theta}^{l} &\in H^{s - l}(\Omega) \cap H^{1}_{0}(\Omega), \quad \bar{p}^{k} \in H^{s - 1 - k}(\Omega)
			\end{align*}
			where $\bar{z}^{m}, \bar{\theta}^{l}, \bar{p}^{k}$ are recursively defined via
			\begin{align*}
				\bar{z}^{0}(x) &= z^{0}(x), \quad \bar{z}^{1}(x) = z^{1}(x), \quad
				\bar{\theta}^{0}(x) = \theta^{0}(x), \quad \bar{p}^{0}(x) = p^{0}(x), \\
				\begin{pmatrix}
					\bar{z}^{m} \\
					\bar{\theta}^{m - 1} \\
					\bar{p}^{m - 1}
				\end{pmatrix}(x) 
				&=
				\begin{pmatrix}
					  \Big(\sum\limits_{n = 0}^{m - 2} \binom{m - 2}{n}
					  \partial_{t}^{n} \bar{a}_{ij} \partial_{x_{i}} \partial_{x_{j}} \bar{z}^{m - 2 - n} +
					  \tfrac{\alpha}{\gamma} A\bar{\theta}^{m - 2} - B\bar{\theta}^{m - 2} + \partial_{t}^{m - 2} \bar{f}_{i} \Big)(0, x) \\
					  -\tfrac{1}{\beta} \bar{p}^{m - 2}(x) - \tfrac{\alpha}{\beta} \bar{z}^{m - 1}(x) \\
					  -\tfrac{1}{\tau} \bar{p}^{m - 2}(x) + \tfrac{\eta}{\tau} \big(A \bar{\theta}^{m - 2}\big)(x)
				\end{pmatrix}
			\end{align*}
			for $m \geq 2$ and $x \in \Omega$.
		\end{enumerate}
	\end{assumption}
	
	\noindent
	Note that our Assumption \ref{ASSUMPTION_APPENDIX}
	differs both from \cite[Assumption A.2.1]{JiaRa2000} and \cite[Assumption A.2]{LaPoWa2017}.
	
	\begin{theorem}
		\label{THEOREM_APPENDIX}
		
		Under Assumption \ref{ASSUMPTION_APPENDIX},
		the initial-boundary value problem (\ref{EQUATION_LINEARIZED_SYSTEM_PDE_1})-(\ref{EQUATION_LINEARIZED_SYSTEM_IC_2})
		possesses a unique classical solution $(z, \theta, p)$ such that
		\begin{align*}
			z &\in \bigcap_{m = 0}^{s - 1} C^{m}\big([0, T], H^{s - m}(\Omega) \cap H^{1}_{0}(\Omega)\big) \cap
			C^{s}\big([0, T], L^{2}(\Omega)\big), \\
			\theta &\in \bigcap_{m = 0}^{s - 1} C^{m}\big([0, T], H^{s - m}(\Omega) \cap H^{1}_{0}(\Omega)\big), \quad
			p \in \bigcap_{m = 0}^{s - 1} C^{m}\big([0, T], H^{s - 1 - m}(\Omega)\big).
		\end{align*}
		Moreover, letting
		\begin{equation}
			\begin{split}
				\phi_{0} &= \|\bar{a}_{ij}(0, \cdot)\|_{L^{\infty}(\Omega)} +
				\|\partial_{x_{k}} \bar{a}_{ij}(0, \cdot)\|_{H^{s - 1}(\Omega)}, \\
				\phi &= \max_{0 \leq t \leq T} \Big(\|\bar{a}_{ij}(t, \cdot)\|_{L^{\infty}(\Omega)} +
				\|\partial_{x_{k}} \bar{a}_{ij}(t, \cdot)\|_{H^{s - 1}(\Omega)} +
				\sum_{m = 1}^{s - 1} \|\partial_{t}^{m} \bar{a}_{ij}(t, \cdot)\|_{H^{s - 1 - m}(\Omega)}\Big),
			\end{split}
			\notag
		\end{equation}
		there exists a positive number $K_{1}$,
		which is a continuous function of $\phi_{0}$, $\gamma_{0}$ and $\gamma_{1}$,
		and a positive number $K_{2}$, which continuously depends on $\phi$, $\gamma_{0}$ and $\gamma_{1}$,
		such that
		\begin{align*}
			\max_{0 \leq t \leq T} \Big(\big\|\bar{D}^{s} z(t, \cdot)\big\|_{L^{2}(\Omega)}^{2} &+
			\big\|\bar{D}^{s - 1} \theta(t, \cdot)\big\|_{H^{1}(\Omega)}^{2} +
			\big\|\bar{D}^{s - 1} p(t, \cdot)\big\|_{L^{2}(\Omega)}^{2}\Big) \\
			&\leq
			K_{1} \Lambda_{0} \exp\big(K_{2} T^{1/2}(1 + T^{1/2} + T + T^{3/2})\big),
			\notag
		\end{align*}
		where
		\begin{align*}
			\Lambda_{0} &:=
			\sum_{m = 0}^{s} \|\bar{z}^{m}\|_{H^{s - m}(\Omega)}^{2} +
			\sum_{m = 0}^{s - 1} \|\bar{\theta}^{m}\|_{H^{s - m}(\Omega)}^{2} +
			\sum_{m = 0}^{s - 1} \|\bar{p}^{m}\|_{H^{s - 1 - m}(\Omega)}^{2} \\
			&+ (1 + T) \sup_{0 \leq t \leq T} \big\|\bar{D}^{s - 2} \bar{f}(t, \cdot)\big\|_{L^{2}(\Omega)}
			+ T^{1/2} \int_{0}^{T} \|\partial_{t}^{s - 1} \bar{f}(t, \cdot)\|_{L^{2}(\Omega)}^{2} \mathrm{d}t.
			\notag
		\end{align*}
	\end{theorem}
	
	\begin{proof}
		This proof adopts the abstract solution theory \cite[Theorems A.3 and A.9]{JiaRa2000}.
		When treating the $z$-variable, we follow the streamlines of Lasiecka {\it et al.} \cite[Appendix A.1]{LaPoWa2017}.
		\medskip \\
		{\em Existence and uniqueness at basic regularity level.}
		For $t \in [0, T]$, define a bounded linear operator
		\begin{equation}
			\mathcal{A}(t) :=
			\begin{pmatrix}
				0                                                         & -1                    &  0                           & 0 \\
				-\bar{a}_{ij}(t, \cdot) \partial_{x_{i}} \partial_{x_{j}} &  0                    & -\frac{\alpha}{\gamma} A + B & 0 \\
				0                                                         &  \frac{\alpha}{\beta} &  0                           & \frac{1}{\beta} \\ 
				0                                                         &  0                    & -\frac{\eta}{\tau} A         & \frac{1}{\tau}
			\end{pmatrix}
			\colon Y_{1} \longrightarrow X_{0},
			\label{EQUATION_OPERATOR_FAMILY_LINEAR_WAVE_EQUATION}
		\end{equation}
		where the Hilbert space
		\begin{align*}
            X_{0} := H^{1}_{0}(\Omega) \times L^{2}(\Omega) \times H^{1}_{0}(\Omega) \times L^{2}(\Omega)
		\end{align*}
		is endowed with the inner product
        \begin{equation*}
			\langle V, \bar{V}\rangle_{t} :=
			\big\langle \bar{a}_{ij}(t, \cdot) \partial_{x_{i}} z, \partial_{x_{j}} \bar{z}\big\rangle_{L^{2}(\Omega)} + \langle y, \bar{y}\rangle_{L^{2}(\Omega)} +
			\tfrac{\beta}{\gamma} \langle A^{1/2} \theta, A^{1/2} \bar{\theta}\rangle_{L^{2}(\Omega)} + \tfrac{\tau}{\gamma \eta} \langle p, \bar{p}\rangle_{L^{2}(\Omega)}
		\end{equation*}
		for $(z, y, \theta, p), (\bar{z}, \bar{y}, \bar{\theta}, \bar{p}) \in X_{0}$
		(the bilinear form $\langle \cdot, \cdot\rangle_{t}$ is equivalent with the standard inner product on $X_{0}$ due to coercivity Assumption \ref{ASSUMPTION_APPENDIX}.3)
		and the Hilbert space
		\begin{align*}
			Y_{1} :=
			\big(H^{2}_{0}(\Omega) \cap H^{1}_{0}(\Omega)\big) \times H^{1}_{0}(\Omega) \times
			\big(H^{2}_{0}(\Omega) \cap H^{1}_{0}(\Omega)\big) \times H^{1}_{0}(\Omega)
		\end{align*}
		is equipped with the usual inner product.
		With this notation, letting $V := (z, \partial_{t} z, \theta, p)$,
		Equations (\ref{EQUATION_LINEARIZED_SYSTEM_PDE_1})--(\ref{EQUATION_LINEARIZED_SYSTEM_IC_2})
		can be cast into the form of an abstract Cauchy problem
		\begin{equation}
            \label{EQUATION_LINEARIZED_SYSTEM_CD_SYSTEM}
			\partial_{t} V(t) + \mathcal{A}(t) V(t) = F(t) \text{ in } (0, T), \quad V(0) = V^{0}
		\end{equation}
		with $F = (0, \bar{f}, 0, 0)$ and $V^{0} = (z^{0}, z^{1}, \theta^{0}, p^{0})$.

		We want to show that the triple $\big(\mathcal{A}; X_{0}, Y_{1}\big)$ is a CD-system as defined in \cite[Section A.1]{JiaRa2000}.
		Obviously, $D\big(\mathcal{A}(t)\big) = Y_{1}$. In particular, this means the domain of $\mathcal{A}(t)$ is time-independent, 
		and the operator $\mathcal{A}(t)$ itself is closed. 
		Indeed, suppose 
		\begin{equation}
            \notag
            \mathcal{A}(t)
            \begin{pmatrix}
                z \\ y \\ \theta \\ p
            \end{pmatrix}
            =
            \begin{pmatrix}
                -y \\
                -\bar{a}_{ij}(t, \cdot) \partial_{x_{i}} \partial_{x_{j}} z - \frac{\alpha}{\gamma} A\theta + B \theta \\
                \frac{\alpha}{\beta} y + \frac{1}{\beta} p \\
                -\frac{\eta}{\tau} A \theta + \frac{1}{\tau} p
            \end{pmatrix} \in
            H^{1}_{0}(\Omega) \times L^{2}(\Omega) \times L^{2}(\Omega) \times L^{2}(\Omega).
		\end{equation}
		Inspecting the first component, we get $y \in H^{1}_{0}(\Omega)$. Since $p \in L^{2}(\Omega)$ and $0 \in \rho\big(\mathcal{A}(t)\big)$,
		the fourth inclusion yields $\theta \in H^{2}(\Omega) \cap H^{1}_{0}(\Omega)$.
        Now, exploring the second inclusion, \ref{ASSUMPTION_APPENDIX}.4 suggests $z \in H^{2}(\Omega) \cap H^{1}_{0}(\Omega)$.
		Finally, combining these regularity properties, the third inclusion furnishes $y \in H^{1}_{0}(\Omega)$.
		
		For $t \in [0, T]$, consider the ``elliptic'' problem
		\begin{equation}
            \label{EQUATION_LINEARIZED_SYSTEM_ELLIPTIC}
			\big(\mathcal{A}(t) + \lambda\big) V = F \quad \text{ with } \quad F \equiv (f_{1}, f_{2}, f_{3}, f_{4}) \in X_{0}. \notag
		\end{equation}
		Letting $V = (z, y, \theta, p)$ and expressing Equation (\ref{EQUATION_LINEARIZED_SYSTEM_ELLIPTIC}) in the component form, we get
        \begin{subequations}
        \begin{align}
            \label{EQUATION_LINEARIZED_SYSTEM_ELLIPTIC_COMPONENT_1}
            -y + \lambda z &= f_{1}, \\
            \label{EQUATION_LINEARIZED_SYSTEM_ELLIPTIC_COMPONENT_2}
            -\bar{a}_{ij}(t, \cdot) \partial_{x_{i}} \partial_{x_{j}} z - \tfrac{\alpha}{\gamma} A\theta + B \theta + \lambda y &= f_{2}, \\
            \label{EQUATION_LINEARIZED_SYSTEM_ELLIPTIC_COMPONENT_3}
            \tfrac{\alpha}{\beta} y + \tfrac{1}{\beta} p + \lambda \theta &= f_{3}, \\
            \label{EQUATION_LINEARIZED_SYSTEM_ELLIPTIC_COMPONENT_4}
            -\tfrac{\eta}{\tau} A \theta + \tfrac{p}{\tau} + \lambda p &= f_{4}.
        \end{align}
        \end{subequations}
        Solving Equations (\ref{EQUATION_LINEARIZED_SYSTEM_ELLIPTIC_COMPONENT_1}) and (\ref{EQUATION_LINEARIZED_SYSTEM_ELLIPTIC_COMPONENT_4}) for $y$ and $z$, respectively, 
        \begin{equation}
            \label{EQUATION_LINEARIZED_SYSTEM_ELLIPTIC_ELIMINATION}
            y = \lambda z  - f_{1} \quad \text{ and } p = \big(\lambda + \tfrac{1}{\tau}\big)^{-1} \big(f_{4} + \tfrac{\eta}{\tau} A\theta\big) \notag
        \end{equation}
        and plugging the result into Equations (\ref{EQUATION_LINEARIZED_SYSTEM_ELLIPTIC_COMPONENT_2}), (\ref{EQUATION_LINEARIZED_SYSTEM_ELLIPTIC_COMPONENT_3}), we arrive at
        \begin{subequations}
        \begin{align}
            \label{EQUATION_LINEARIZED_SYSTEM_ELLIPTIC_COMPONENT_TRANSFORMED_1}
            \big(\lambda^{2} - \bar{a}_{ij}(t, \cdot) \partial_{x_{i}} \partial_{x_{j}}\big) z - \big(\tfrac{\alpha}{\gamma} A - B\big) \theta &= f_{2} + \lambda f_{1}, \\
            \label{EQUATION_LINEARIZED_SYSTEM_ELLIPTIC_COMPONENT_TRANSFORMED_2}
            \tfrac{\alpha}{\beta} \lambda z + \Big(\tfrac{\eta}{\beta \tau} \big(\lambda + \tfrac{1}{\tau}\big)^{-1} A + \lambda\Big) \theta &=
            f_{3} + \tfrac{\alpha}{\beta} f_{1} - \tfrac{1}{\beta} \big(\lambda + \tfrac{1}{\tau}\big)^{-1} f_{4}.
        \end{align}
        \end{subequations}
        Multiplying Equations (\ref{EQUATION_LINEARIZED_SYSTEM_ELLIPTIC_COMPONENT_TRANSFORMED_1})--(\ref{EQUATION_LINEARIZED_SYSTEM_ELLIPTIC_COMPONENT_TRANSFORMED_2})
        in $L^{2}(\Omega)$ with $\tfrac{\alpha^{2}}{\beta^{2}} \bar{z}$ and $\lambda \bar{\theta}$, respectively,
        $\bar{z}$, $\bar{\theta}$ being two arbitrary $H^{1}_{0}(\Omega)$-function,
        and integrating by parts using the boundary conditions (\ref{EQUATION_LINEARIZED_SYSTEM_BC}),
        we obtain the variational equation
        \begin{equation}
            \notag
            \mathfrak{a}\big((z, \theta), (\bar{z}, \bar{\theta})\big) = F\big((\bar{z}, \bar{\theta})\big),
        \end{equation}
        with the bilinear form
        \begin{align}
            \label{EQUATION_LINEARIZED_SYSTEM_BILINEAR_FORM}
            \begin{split}
                \mathfrak{a}\big((z, \theta), &(\bar{z}, \bar{\theta})\big) =
                \tfrac{\alpha^{2}}{\beta^{2}} \langle \bar{a}_{ij}(t, \cdot) \partial_{x_{i}} z, \partial_{x_{j}} z\rangle_{L^{2}(\Omega)} 
                + \tfrac{\alpha^{2}}{\beta^{2}} \big\langle \big(\partial_{x_{i}} \bar{a}_{ij}(t, \cdot)\big) z, \partial_{x_{j}} z\rangle_{L^{2}(\Omega)} \\
                &+ \tfrac{\alpha^{2}}{\beta^{2}} \lambda^{2} \langle z, \bar{z}\rangle_{L^{2}(\Omega)}
                - \tfrac{\alpha^{3}}{\gamma \beta^{2}} \langle A^{1/2} \theta, A^{1/2} \bar{z}\rangle_{L^{2}(\Omega)}
                + \tfrac{\alpha^{2}}{\beta^{2}} \langle B \theta, \bar{z}\rangle_{L^{2}(\Omega)}
                + \tfrac{\alpha}{\beta} \lambda^{2} \langle z, \bar{\theta}\rangle_{L^{2}(\Omega)} \\
                &+ \tfrac{\eta}{\beta \tau} \lambda \big(\lambda + \tfrac{1}{\tau}\big)^{-1} \langle A^{1/2} \theta, A^{1/2} \bar{\theta}\rangle_{L^{2}(\Omega)} 
                + \lambda^{2} \langle \theta, \bar{\theta}\rangle_{L^{2}(\Omega)}
            \end{split}
        \end{align}
        and the linear functional
        \begin{align}
            \label{EQUATION_LINEARIZED_SYSTEM_LINEAR_FUNCTIONAL}
            F\big((\bar{z}, \bar{\theta})\big) =
            \tfrac{\alpha^{2}}{\beta^{2}} \langle f_{2} + \lambda f_{1}, \bar{z}\rangle_{L^{2}(\Omega)} +
            \lambda \big\langle f_{3} + \tfrac{\alpha}{\beta} f_{1} - \tfrac{1}{\beta} \big(\lambda + \tfrac{1}{\tau}\big)^{-1} f_{4}, \bar{\theta}\big\rangle_{L^{2}(\Omega)}.
        \end{align}
        Clearly, both $\mathfrak{a}(\cdot, \cdot)$ and $F$ are continuous on $\mathcal{V} \times \mathcal{V}$ and $\mathcal{V}$, respectively,
        with $\mathcal{V} := \big(H^{1}_{0}(\Omega)\big)^{2}$.
        Further, for sufficiently large $\lambda$, applying Young's and the Poincar\'{e} \& Friedrichs inequalities
        and using the boundedness of $\|\partial_{x_{i}} \bar{a}_{ij}\|_{L^{\infty}}$ (viz. Assumption \ref{ASSUMPTION_APPENDIX}.2) and that of $B$,
        we estimate
        \begin{align}
            \label{EQUATION_LINEARIZED_SYSTEM_BILINEAR_FORM_COERCIVITY}
            \begin{split}
                \mathfrak{a}\big((z, \theta), (z, \theta)\big) &\geq
                \tfrac{\alpha^{2}}{\beta^{2}} \langle \bar{a}_{ij}(t, \cdot) \partial_{x_{i}} z, \partial_{x_{j}} z\rangle_{L^{2}(\Omega)} 
                - \varepsilon \|z\|_{H^{1}(\Omega)}^{2} - C_{\varepsilon} \|z\|_{L^{2}(\Omega)}^{2} \\
                &+ \tfrac{\alpha^{2}}{\beta^{2}} \lambda^{2} \|z\|_{L^{2}(\Omega)}^{2} - \varepsilon \|z\|_{H^{1}(\Omega)}^{2} - C_{\varepsilon} \|\theta\|_{H^{1}(\Omega)}^{2}
                - \tfrac{\alpha^{2}}{2 \beta^{2}} \lambda^{2} \|z\|_{L^{2}(\Omega)}^{2} - \tfrac{\lambda^{2}}{2} \|\theta\|_{L^{2}(\Omega)}^{2} \\
                &+ \tfrac{\eta}{\beta \tau} \lambda \big(\lambda + \tfrac{1}{\tau}\big)^{-1} \big\|A^{1/2} \theta\big\|_{L^{2}(\Omega)}^{2} + \lambda^{2} \|\theta\|_{L^{2}(\Omega)}^{2}
            \end{split}
        \end{align}
        for any $\varepsilon > 0$. Now, selecting $\varepsilon$ sufficiently small and, if necessary, increasing $\lambda$,
        we easily see that $B(\cdot, \cdot)$ is coercive, i.e.,
        \begin{equation*}
            \mathfrak{a}\big((z, \theta), (z, \theta)\big) \geq \kappa \big(\|z\|_{H^{1}(\Omega)}^{2} + \|\theta\|_{H^{1}(\Omega)}^{2}\big)
        \end{equation*}
        for some $\kappa > 0$.
        Hence, invoking Lax \& Milgram's Lemma, we obtain a unique solution $(z, \theta) \in \big(H^{1}_{0}(\Omega)\big)^{2}$
        to Equations (\ref{EQUATION_LINEARIZED_SYSTEM_ELLIPTIC_COMPONENT_TRANSFORMED_1})--(\ref{EQUATION_LINEARIZED_SYSTEM_ELLIPTIC_COMPONENT_TRANSFORMED_2}).
        By elliptic regularity (cf. Assumption \ref{ASSUMPTION_APPENDIX}.4), Equation (\ref{EQUATION_LINEARIZED_SYSTEM_ELLIPTIC_COMPONENT_TRANSFORMED_2}) implies $\theta \in H^{2}(\Omega) \cap H^{1}_{0}(\Omega)$.
        Next, plugging in $\theta$ into Equation (\ref{EQUATION_LINEARIZED_SYSTEM_ELLIPTIC_COMPONENT_TRANSFORMED_1}),
        Assumption \ref{ASSUMPTION_APPENDIX}.4 suggests $z \in H^{2}(\Omega) \cap H^{1}_{0}(\Omega)$.
        Substituting into Equation (\ref{EQUATION_LINEARIZED_SYSTEM_ELLIPTIC_ELIMINATION}),
        we further get $y, p \in H^{1}_{0}(\Omega)$.
        Hence, we found a solution $V = (z, y, \theta, p) \in D\big(\mathcal{A}(t)\big)$ to Equation (\ref{EQUATION_LINEARIZED_SYSTEM_ELLIPTIC}).
        The estimate $\|V\|_{X_{0}} \leq C \|F\|_{X_{0}}$ for some $C > 0$
        immediately follows from Lax \& Milgram's Lemma and Equation (\ref{EQUATION_LINEARIZED_SYSTEM_ELLIPTIC_ELIMINATION}).
        Hence, $\big(\mathcal{A}(t) + \lambda\big)$ is a maximal operator on $X_{0}$.

        Further, we prove the operator $\mathcal{A}(t) + \lambda$ for any sufficiently large $\lambda > 0$.
        Using integration by parts and the boundary conditions (\ref{EQUATION_LINEARIZED_SYSTEM_BC}), we estimate
        \begin{align*}
        	\big\langle \mathcal{A}(t) V, V\big\rangle_{X_{0}} &=
        	-\big\langle \bar{a}_{ij}(t, \cdot) \partial_{x_{i}} \partial_{x_{j}} z, y\rangle_{L^{2}(\Omega)}
        	-\big\langle \bar{a}_{ij}(t, \cdot) \partial_{x_{i}} z, \partial_{x_{j}} y\rangle_{L^{2}(\Omega)} \\
        	&-\tfrac{\alpha}{\gamma} \langle A\theta, y\rangle_{L^{2}(\Omega)}
        	+ \tfrac{\alpha}{\gamma} \langle A^{1/2} \theta, A^{1/2} y\rangle_{L^{2}(\Omega)}
        	+ \tfrac{1}{\gamma} \langle A^{1/2} p, A^{1/2} \theta\rangle_{L^{2}(\Omega)} \\
        	&+ \langle B \theta, y\rangle_{L^{2}(\Omega)} - \tfrac{1}{\gamma} \langle A \theta, p\rangle_{L^{2}(\Omega)} + \tfrac{1}{\gamma \eta} \|p\|_{L^{2}(\Omega)}^{2} \\
        	&= 
        	\tfrac{1}{\gamma \eta} \|p\|_{L^{2}(\Omega)}^{2}
        	- \big\|\partial_{x_{i}} \bar{a}_{ij}(t, \cdot)\big)\big\|_{L^{\infty}(\Omega)} \|z\|_{H^{1}(\Omega)} \|y\|_{L^{2}(\Omega)} 
        	+ \langle B \theta, y\rangle_{L^{2}(\Omega)} \\
        	&\geq -\lambda_{0} \|V\|_{X_{0}}^{2}
        \end{align*}
        for some $\lambda_{0} > 0$ depending on $\gamma$, $\eta$, $\tau$, 
        $\|B\|_{L(L^{2}(\Omega))}$ and $\big\|\partial_{x_{i}} \bar{a}_{ij}(t, \cdot)\big\|_{L^{\infty}((0, T) \times \Omega)}$
        (cf. Assumption \ref{ASSUMPTION_APPENDIX}.2 and \ref{ASSUMPTION_APPENDIX}.3).
        Hence, by virtue of Lummer \& Phillips' Theorem and standard perturbation results, 
        $\mathcal{A}(t)$ is a negative generator of a $C_{0}$-semigroup of contractions on $X_{0}$.
        Summarizing, we have shown $\big(\mathcal{A}(t)\big)_{t \in [0, T]}$ is a stable family of infinitesimal negative generators
		of $C_{0}$-semigroups on $X_{0}$ with stability constants $M = 1, \omega = \lambda_{0}$.
		Taking into account regularity conditions from Assumption \ref{ASSUMPTION_APPENDIX}.5,
		we can apply \cite[Theorem A.3]{JiaRa2000} and get a unique classical solution
		\begin{equation}
			V \in C^{0}\big([0, T], Y_{1}\big) \cap C^{1}\big([0, T], X_{0}\big) \notag
		\end{equation}
		at the at basic regularity level, which is equivalent with
		\begin{align*}
			z &\in C^{2}\big([0, T], L^{2}(\Omega)\big) \cap C^{1}\big([0, T], H^{1}_{0}(\Omega)\big) \cap C^{0}\big([0, T], H^{2}(\Omega) \cap H^{1}_{0}(\Omega)\big), \\
			\theta &\in C^{1}\big([0, T], H^{1}_{0}(\Omega)\big) \cap C^{0}\big([0, T], H^{2}(\Omega) \cap H^{1}_{0}(\Omega)\big), \\
			p &\in C^{1}\big([0, T], L^{2}(\Omega)\big) \cap C^{0}\big([0, T], H^{1}_{0}(\Omega)\big).
		\end{align*}
		\noindent {\em Higher regularity.}
		For the proof of higher solution regularity,
		we consider the following increasing double scale $\big((X_{j}, Y_{j})\big)_{j \geq 0}$ of Hilbert spaces
		with $Y_{0} := X_{0}$ and
		\begin{align*}
            X_{j} &= \big(H^{j+1}(\Omega) \cap H^{1}_{0}(\Omega)\big) {\times} H^{j}(\Omega) {\times} \big(H^{j+1}(\Omega) \cap H^{1}_{0}(\Omega)\big) {\times} H^{j}(\Omega), \\
            Y^{j} &= \big(H^{j+1}(\Omega) \cap H^{1}_{0}(\Omega)\big) {\times} \big(H^{j}(\Omega) \cap H^{1}_{0}(\Omega)\big) {\times}
            \big(H^{j+1}(\Omega) \cap H^{1}_{0}(\Omega)\big) {\times} \big(H^{j}(\Omega) \cap H^{1}_{0}(\Omega)\big) \text{ for } j \geq 1.
        \end{align*}
		On the strength of Equation (\ref{EQUATION_OPERATOR_FAMILY_LINEAR_WAVE_EQUATION}), the condition
		\begin{align*}
			\partial_{t} \mathcal{A} \in \mathrm{Lip}\big([0, T], L(Y_{j+s+5}, X_{j})\big)
			\text{ for } j = 0, \dots, s - r - 1 \text{ and }
			r = 0, \dots, s - 2 \notag
		\end{align*}
		reduces to verifying
		\begin{equation}
            \label{EQUATION_PARTIAL_T_A_REGULARITY_REDUCED}
			\partial_{t}^{r} \bar{a}_{ij}(t, \cdot) \partial_{x_{i}} \partial_{x_{j}} \quad \text{ and } \quad
			\partial_{t}^{r} A, \partial_{t}^{r} B \in \mathrm{Lip}\big([0, T], L\big(H^{j+r+2}(\Omega) \cap H^{1}_{0}(\Omega), H^{j}(\Omega)\big)\big)
		\end{equation}
		for $j = 0, \dots, s - r - 1$ and $r = 0, \dots, s - 2$.
		Equation (\ref{EQUATION_PARTIAL_T_A_REGULARITY_REDUCED}) is a direct consequence of Assumption \ref{ASSUMPTION_APPENDIX}.2
		and the Sobolev's embedding $H^{\lfloor d/2\rfloor + 1}(\Omega) \hookrightarrow L^{\infty}(\Omega)$.
		In a similar fashion, exploiting Assumption \ref{ASSUMPTION_APPENDIX}.4,
		we observe for $j = 0, \dots, s-2$ and $\phi \in Y_{1}$ and a.e. $t \in [0, T]$ that $\mathcal{A}(t) \phi \in X_{j}$ implies
		\begin{equation}
			\phi \in Y_{j+1} \text{ and }
			\|\phi\|_{Y_{j+1}} \leq K \big(\|\mathcal{A}(t) \phi\|_{X_{j}} + \|\phi\|_{X_{0}}\big) 
			\text{ for some constant } K > 0, \notag
		\end{equation}
		which does not depend on $\phi$.
		Further, Assumption \ref{ASSUMPTION_APPENDIX}.5 yields
		\begin{equation}
			\partial_{t} F \in C^{0}\big([0, T], X_{s-1-k}\big)
			\text{ for } k = 0, \dots, s-2 \text{ and }
			\partial_{t}^{s-1} F \in L^{1}(0, T; X_{0}). \notag
		\end{equation}
		Finally, Assumption \ref{ASSUMPTION_APPENDIX}.6 implies compatibility conditions
		in sense of \cite[Equations (A.8) and (A.9)]{JiaRa2000}.
		Hence, applying \cite[Theorem A.9]{JiaRa2000} at the energy level $s - 1$,
		we obtain additional regularity for the classical solution satisfying
		\begin{equation}
			V \in \bigcap_{m = 0}^{s-1} C^{m}\big([0, T], Y_{s - 1 - m}\big). \notag
		\end{equation}
		Resubstituting, this yields the desired regularity for $z, \theta, p$. \medskip \\
		{\em Energy estimates.}
		For $n = 1, \dots, s - 1$, applying the $\partial_{t}^{n-1}$-operator to Equations
		(\ref{EQUATION_LINEARIZED_SYSTEM_PDE_1})--(\ref{EQUATION_LINEARIZED_SYSTEM_PDE_3}) and recalling the compatibility conditions from Assumption \ref{ASSUMPTION_APPENDIX}.6, obtain
        \begin{subequations}
        \begin{align}
            \partial_{t}^{2} \big(\partial_{t}^{n - 1} z\big) - \bar{a}_{ij} \partial_{x_{i}} \partial_{x_{j}} \big(\partial_{t}^{n - 1} z\big) 
            - \big(\tfrac{\alpha}{\gamma} A - B\big) \big(\partial_{t}^{n - 1} \theta\big) &= h^{n} & &\text{in } (0, T) \times \Omega,
            \label{EQUATION_LINEARIZED_SYSTEM_PDE_DIFFERENTIATED_WRT_TIME_1} \\
            \beta \partial_{t} \big(\partial_{t}^{n - 1} \theta\big) + \big(\partial_{t}^{n - 1} p\big) +
            \alpha \partial_{t} \big(\partial_{t}^{n - 1} z\big) &= 0 & &\text{in } (0, T) \times \Omega,
            \label{EQUATION_LINEARIZED_SYSTEM_PDE_DIFFERENTIATED_WRT_TIME_2} \\
            \tau \partial_{t} \big(\partial_{t}^{n - 1} p\big) + \big(\partial_{t}^{n - 1} p\big) - \eta A \big(\partial_{t}^{n - 1} \theta\big) 
            &= 0 & &\text{in } (0, T) \times \Omega,
            \label{EQUATION_LINEARIZED_SYSTEM_PDE_DIFFERENTIATED_WRT_TIME_3} \\
            \partial_{t}^{n - 1} z = 0, \quad \partial_{t}^{n - 1} \theta  &= 0 & &\text{in } (0, T) \times \partial \Omega,
            \label{EQUATION_LINEARIZED_SYSTEM_DIFFERENTIATED_WRT_TIME_BC} \\
            \partial_{t}^{n - 1} z(0, \cdot) = z^{n - 1}, \quad \partial_{t} \big(\partial_{t}^{n - 1} z\big)(0, \cdot) &= z^{n} & &\text{in } \Omega,
            \label{EQUATION_LINEARIZED_SYSTEM_DIFFERENTIATED_WRT_TIME_IC_1} \\
            \partial_{t}^{n - 1} \theta(0, \cdot) = \theta^{n - 1}, \quad \partial_{t}^{n - 1} p(0, \cdot) &= p^{n - 1} & &\text{in } \Omega.
            \label{EQUATION_LINEARIZED_SYSTEM_DIFFERENTIATED_WRT_TIME_IC_2}
        \end{align}
        \end{subequations}
        with
        \begin{equation}
			h^{n - 1} = \partial_{t}^{n - 1} \bar{f} +
			\sum_{m = 1}^{n - 1} \binom{n - 1}{m} \big(\partial_{t}^{m} \bar{a}_{ij}\big)
			\partial_{x_{i}} \partial_{x_{j}} \partial_{t}^{n - 1 - m} z.
			\label{EQUATION_FUNCTION_N_MINUS_ONE_DEFINITION}
		\end{equation}
		For $t \in [0, T]$, multiplying Equations (\ref{EQUATION_LINEARIZED_SYSTEM_PDE_DIFFERENTIATED_WRT_TIME_2}) and (\ref{EQUATION_LINEARIZED_SYSTEM_PDE_DIFFERENTIATED_WRT_TIME_3}) 
		in $L^{2}\big((0, t) \times \Omega\big)$ with $\frac{1}{\gamma} A \partial_{t}^{n - 1} \theta$ and $\frac{1}{\gamma \eta} \partial_{t}^{n - 1} p$, respectively,
		integrating by parts while taking into account the boundary conditions (\ref{EQUATION_LINEARIZED_SYSTEM_DIFFERENTIATED_WRT_TIME_BC}) and
		adding up the resulting identities, we get
		\begin{align}
            \notag
            \frac{\beta}{2 \gamma} \big\|A^{1/2} \partial_{t}^{n - 1} &\theta\big\|_{L^{2}(\Omega)}^{2} \big|_{\tau = 0}^{\tau = t}
            + \frac{1}{\gamma} \int_{0}^{t} \langle A^{1/2} \partial_{t}^{n - 1} p, A^{1/2} \partial_{t}^{n - 1} \theta\rangle_{L^{2}(\Omega)} \mathrm{d}\tau \\
            \label{EQUATION_APPENDIX_ENERGY_ESTIMATE_THETA_AND_P_AT_LEVEL_N_MINUS_ONE}
            &+ \frac{\alpha}{\gamma} \int_{0}^{t} \langle A^{1/2} \partial_{t}^{n} z, A^{1/2} \partial_{t}^{n - 1} \theta\rangle_{L^{2}(\Omega)} \mathrm{d}\tau
            + \frac{\tau}{2 \gamma \eta} \big\|\partial_{t}^{n - 1} p\big\|_{L^{2}(\Omega)}^{2} \big|_{\tau = 0}^{\tau = t} \\
            &+
            \notag
            \frac{1}{\gamma \eta} \int_{0}^{t} \big\|\partial_{t}^{n - 1} p\big\|_{L^{2}(\Omega)}^{2} \mathrm{d}\tau -
            \frac{1}{\gamma} \int_{0}^{t} \langle A^{1/2} \partial_{t}^{n - 1} \theta, A^{1/2} \partial_{t}^{n - 1} p\rangle_{L^{2}(\Omega)} \mathrm{d}\tau \equiv 0.
		\end{align}
        Summing up over $n = 1, \dots, s - 1$, we find
        \begin{align}
            \label{EQUATION_APPENDIX_ENERGY_ESTIMATE_THETA_AND_P_SUMMED_OVER_LEVEL_N}
            \begin{split}
                \frac{1}{C}  \sum_{n = 0}^{s -2} \Big(\big\|\partial_{t}^{n} \theta(t, \cdot)\big\|_{H^{1}(\Omega)}^{2} &+ \big\|\partial_{t}^{n} p(t, \cdot)\big\|_{L^{2}(\Omega)}^{2}\Big) 
                + \frac{\alpha}{\gamma} \sum_{n = 1}^{s} \int_{0}^{t} \langle A^{1/2} \partial_{t}^{n} z, A^{1/2} \partial_{t}^{n - 1} \theta\rangle_{L^{2}(\Omega)} \mathrm{d}\tau \\
                &\leq C \Lambda_{0} + \sum_{n = 0}^{s -2} \int_{0}^{t} 
                \Big(\big\|\partial_{t}^{n} \theta(\tau, \cdot)\big\|_{H^{1}(\Omega)}^{2} + \big\|\partial_{t}^{n} p(\tau, \cdot)\big\|_{L^{2}(\Omega)}^{2}\Big) \mathrm{d}\tau
            \end{split}
        \end{align}
        for a large generic constant $C > 0$.
        To derive an estimate for $\partial_{t}^{s - 1} \theta$ and $\partial_{t}^{s - 1} p$, we use the mollifier from Equation (\ref{EQUATION_APPENDIX_CONVOLUTION}).
        Convolving Equations (\ref{EQUATION_LINEARIZED_SYSTEM_PDE_DIFFERENTIATED_WRT_TIME_2}), (\ref{EQUATION_LINEARIZED_SYSTEM_PDE_DIFFERENTIATED_WRT_TIME_3}) for $n = s - 1$ 
        with the Friedrichs' kernel $\phi_{\delta}$, we obtain for any $t \in [\varepsilon, T - \varepsilon]$:
        \begin{align}
            \big(\partial_{t}^{s - 1} \theta\big)_{\delta} + \big(\partial_{t}^{s - 2} p\big)_{\delta} +
            \alpha \big(\partial_{t}^{s - 1} z\big)_{\delta} &= 0 & &\text{in } (0, T) \times \Omega,
            \label{EQUATION_LINEARIZED_SYSTEM_PDE_DIFFERENTIATED_WRT_TIME_MOLLIFIED_2} \\
            \tau \big(\partial_{t}^{s - 1} p\big)_{\delta} + \big(\partial_{t}^{s - 2} p\big)_{\delta} - \eta A \big(\partial_{t}^{s - 2} \theta\big)_{\delta}
            &= 0 & &\text{in } (0, T) \times \Omega,
            \label{EQUATION_LINEARIZED_SYSTEM_PDE_DIFFERENTIATED_WRT_TIME_MOLLIFIED_3}
        \end{align}
        where we used fact that $\big(\partial_{t} w\big)_{\delta} = \partial_{t} \big(w_{\delta}\big)$
        and $\big(\partial^{n} \theta\big)_{\delta}$, $\big(\partial^{n} p\big)_{\delta}$ satisfy the 
        same homogeneous Dirichlet boundary conditions as for $\partial_{t}^{n} \theta$ and $\partial_{t}^{n} p$.
        Applying the $\partial_{t}$-operator to Equations (\ref{EQUATION_LINEARIZED_SYSTEM_PDE_DIFFERENTIATED_WRT_TIME_MOLLIFIED_2}), (\ref{EQUATION_LINEARIZED_SYSTEM_PDE_DIFFERENTIATED_WRT_TIME_MOLLIFIED_3})
        and multiplying in $L^{2}\big((0, t) \times \Omega)$ with $\frac{1}{\gamma} A \big(\partial_{t}^{s - 1} \theta\big)_{\delta}$ 
        and $\frac{1}{\gamma \eta} \big(\partial_{t}^{s - 1} p\big)_{\delta}$, respectively,
        Similar to Equations (\ref{EQUATION_APPENDIX_ENERGY_ESTIMATE_THETA_AND_P_AT_LEVEL_N_MINUS_ONE}), (\ref{EQUATION_APPENDIX_ENERGY_ESTIMATE_THETA_AND_P_SUMMED_OVER_LEVEL_N}), 
        we get for $t \in [\varepsilon, T - \varepsilon]$:
        \begin{align}
            \notag
            \frac{1}{C} \Big(\big\|(\partial_{t}^{s - 1} \theta)_{\delta}(t, \cdot)\big\|_{H^{1}(\Omega)}^{2} &+ \big\|(\partial_{t}^{s - 1} p)_{\delta}(t, \cdot)\big\|_{L^{2}(\Omega)}^{2}\Big) 
            + \frac{\alpha}{\gamma} 
            \int_{0}^{t} \big\langle \partial_{t}^{s} (A^{1/2} z)_{\delta}, \partial_{t}^{s - 1} (A^{1/2} \theta)_{\delta}\big\rangle_{L^{2}(\Omega)} \mathrm{d}\tau \\
            \label{EQUATION_APPENDIX_ENERGY_ESTIMATE_THETA_AND_P_HIGHEST_TIME_DERIVATIVE}
            &\leq C \Big(\big\|(\partial_{t}^{s - 1} \theta)_{\delta}(\varepsilon, \cdot)\big\|_{H^{1}(\Omega)}^{2} 
            + \big\|(\partial_{t}^{s - 1} p)_{\delta}(\varepsilon, \cdot)\big\|_{L^{2}(\Omega)}^{2}\Big) \\
            \notag
            &+ \int_{0}^{t} 
            \Big(\big\|(\partial_{t}^{n} \theta)_{\delta}(\tau, \cdot)\big\|_{H^{1}(\Omega)}^{2} + \big\|(\partial_{t}^{n} p)_{\delta}(\tau, \cdot)\big\|_{L^{2}(\Omega)}^{2}\Big) \mathrm{d}\tau.
        \end{align}

		Repeating the same procedure for the $z$-component (cf. \cite[pp. 216--218]{LaPoWa2017}), we get
        \begin{align}
            \label{EQUATION_APPENDIX_ENERGY_ESTIMATE_Z_LOWER_TIME_DERIVATIVES}
			\begin{split}
				\frac{1}{C} \sum_{n = 1}^{s - 1} \Big(\big\|\partial_{t}^{n} &z(t, \cdot)\big\|_{L^{2}(\Omega)}^{2} + \big\|\partial_{t}^{n - 1} z(t, \cdot)\big\|_{H^{1}(\Omega)}^{2}\Big)
				- \frac{\alpha}{\gamma} \sum_{n = 1}^{s - 1} \int_{0}^{t} \langle A^{1/2} \partial_{t}^{n} z, A^{1/2} \partial_{t}^{n - 1} \theta\rangle_{L^{2}(\Omega)} \mathrm{d}\tau \\
				&\leq
				C(\phi_{0}, \gamma_{0}) \Lambda_{0} + C(\phi, \gamma_{0}) \int_{0}^{t} \big\|\bar{D}^{s} z\big\|_{L^{2}(\Omega)}^{2} \mathrm{d}\tau
				+ C \sum_{n = 1}^{s - 1} \int_{0}^{t} \big\|\partial_{t}^{n - 1} \theta\big\|_{L^{2}(\Omega)}^{2} \mathrm{d}\tau
			\end{split}
		\end{align}
		as well as
		\begin{align}
            \label{EQUATION_APPENDIX_ENERGY_ESTIMATE_Z_HIGHEST_TIME_DERIVATIVE}
			\begin{split}
				\frac{1}{C} \Big(\big\|(\partial_{t}^{s} z)_{\delta}(t, &\cdot)\big\|_{L^{2}(\Omega)}^{2} + \big\|(\partial_{t}^{s - 1} z)_{\delta}(t, \cdot)\big\|_{H^{1}(\Omega)}^{2}\Big)
				- \frac{\alpha}{\gamma} \int_{0}^{t} \big\langle \partial_{t}^{s} (A^{1/2} z)_{\delta}, \partial_{t}^{s - 1} (A^{1/2} \theta)_{\delta}\big\rangle_{L^{2}(\Omega)} \mathrm{d}\tau \\
				&\leq
				C(\phi_{0}, \gamma_{0}) \Big(\big\|(\bar{D}^{s} z)_{\delta}(\varepsilon, \cdot)\big\|_{L^{2}(\Omega)}^{2} +
				\big\|(\partial_{t}^{s - 1} z)_{\delta}(t, \cdot)\big\|_{L^{2}(\Omega)}^{2}\Big) \\
				&+ C(\gamma_{0}, \gamma_{0}) (1 + T^{-1/2}) \int_{0}^{t} \big\|(\bar{D}^{s} z)_{\delta}(\tau, \cdot)\big\|_{L^{2}(\Omega)}^{2} \mathrm{d}\tau \\
				&+ T^{1/2} \int_{\varepsilon}^{t} \big\|\partial_{t} (h^{s - 2})_{\delta}\big\|_{L^{2}}^{2} \mathrm{d}\tau
				+ \int_{\varepsilon}^{t} \|\eta_{\delta}\|_{L^{2}(\Omega)}^{2} \mathrm{d}\tau
				+ C \sum_{n = 1}^{s - 1} \int_{0}^{t} \big\|(\partial_{t}^{s - 1} \theta)_{\delta}\big\|_{L^{2}(\Omega)}^{2} \mathrm{d}\tau
			\end{split}
		\end{align}
		for $t \in [\varepsilon, T - \varepsilon]$ with the `commutator'
		\begin{equation*}
            \eta_{\delta}(t, \cdot) =
            \Big(\bar{a}_{ij}(t, \cdot) \partial_{t}^{s - 2} \partial_{x_{i}} \partial_{x_{j}} z(t, \cdot)\Big)_{\delta} -
            \bar{a}_{ij}(t, \cdot) \Big(\partial_{t}^{s - 2} \partial_{x_{i}} \partial_{x_{j}} z(t, \cdot)\Big)_{\delta}.
		\end{equation*}
		Adding up Equations (\ref{EQUATION_APPENDIX_ENERGY_ESTIMATE_THETA_AND_P_SUMMED_OVER_LEVEL_N}), 
		(\ref{EQUATION_APPENDIX_ENERGY_ESTIMATE_THETA_AND_P_HIGHEST_TIME_DERIVATIVE})--(\ref{EQUATION_APPENDIX_ENERGY_ESTIMATE_Z_HIGHEST_TIME_DERIVATIVE}),
		invoking the the regularity of $z$, $\theta$ and $p$ from the previous step of the proof
		and using \cite[Equation (A.17)]{LaPoWa2017} and Lemma \ref{LEMMA_MOLLIFIER_PROPERTIES} 
		to ``eliminate'' $(\partial_{t}^{s - 1} z)_{\delta}(t, \cdot)$ and $\eta_{\delta}$ on the right-hand side
		of Equation (\ref{EQUATION_APPENDIX_ENERGY_ESTIMATE_Z_HIGHEST_TIME_DERIVATIVE}),
        we send $\varepsilon \to 0$ and then $\delta \to 0$ to arrive at the estimate
		\begin{align}
            \label{EQUATION_APPENDIX_ENERGY_ESTIMATE_TIME_DERIVATIVES_COMBINED}
            \begin{split}
            	\sum_{n = 1}^{s} \Big(\big\|\partial_{t}^{n} &z(t, \cdot)\big\|_{L^{2}(\Omega)}^{2} + \big\|\partial_{t}^{n - 1} z(t, \cdot)\big\|_{H^{1}(\Omega)}^{2}\Big)
            	+ \sum_{n = 1}^{s - 1} \Big(\big\|\partial_{t}^{n} \theta(t, \cdot)\big\|_{H^{1}(\Omega)}^{2} + \big\|\partial_{t}^{n} \theta(t, \cdot)\big\|_{L^{2}(\Omega)}^{2}\Big) \\
				&\leq C(\phi_{0}, \gamma_{0}) \Lambda_{0} + C(\phi, \gamma_{0}) (1 + T^{1/2} + T^{-1/2}) \\
				&\times \int_{0}^{t} \Big(\|\bar{D}^{s} z(\tau, \cdot)\|_{L^{2}(\Omega)}^{2} +
				\|\bar{D}^{s - 1} \theta(\tau, \cdot)\|_{H^{1}(\Omega)}^{2} + \|\bar{D}^{s - 1} p(\tau, \cdot)\|_{L^{2}(\Omega)}^{2}\Big) \mathrm{d}\tau.
            \end{split}
		\end{align}
		To close the estimate in Equation (\ref{EQUATION_APPENDIX_ENERGY_ESTIMATE_TIME_DERIVATIVES_COMBINED}),
        respective space-derivatives on the left-hand side need to be reconstructed.
        To this end, we use Equations (\ref{EQUATION_LINEARIZED_SYSTEM_PDE_DIFFERENTIATED_WRT_TIME_1})--(\ref{EQUATION_LINEARIZED_SYSTEM_PDE_DIFFERENTIATED_WRT_TIME_3}) to write
        \begin{align*}
            \eta A \big(\partial_{t}^{n - 1} \theta\big) &= \tau \partial_{t} \big(\partial_{t}^{n - 1} p\big) + \big(\partial_{t}^{n - 1} p\big), \\
            \bar{a}_{ij} \partial_{x_{i}} \partial_{x_{j}} \big(\partial_{t}^{n - 1} z\big) &= \partial_{t}^{2} \big(\partial_{t}^{n - 1} z\big)
            - \tfrac{\alpha}{\gamma} A \big(\partial_{t}^{n - 1} \theta\big) + B \big(\partial_{t}^{n - 1} \theta\big) - h^{n}, \\
            \big(\partial_{t}^{n - 1} p\big) &= -\beta \partial_{t} \big(\partial_{t}^{n - 1} \theta\big) - \alpha \partial_{t} \big(\partial_{t}^{n - 1} z\big).
        \end{align*}
        Starting at $n = s - 1$ and iteratively going down to $n = 1$,
        while exploiting the elliptic regularity of $A$ and $\bar{a}_{ij} \partial_{x_{i}} \partial_{x_{j}}$ from Assumption \ref{ASSUMPTION_APPENDIX}.4
        as well as regularity of $h^{n}$ from Assumption \ref{ASSUMPTION_APPENDIX}.5, repeating the arguments of our closedness proof for $\mathcal{A}(t)$
        at the basic energy level and the streamlines of \cite[pp. 217--218]{LaPoWa2017}, get
        \begin{equation}
			\begin{split}
                \big\|\bar{D}^{s} z(t, \cdot)\big\|_{L^{2}(\Omega)}^{2} &+
                \big\|\bar{D}^{s - 1} \theta(t, \cdot)\big\|_{H^{1}(\Omega)}^{2} +
                \big\|\bar{D}^{s - 1} p(t, \cdot)\big\|_{L^{2}(\Omega)}^{2} \\
                &\leq C(\phi_{0}, \gamma_{0}, \gamma_{1}) \Lambda_{0}
                + C(\phi, \gamma_{0}, \gamma_{1}) (1 + T^{1/2} + T + T^{-1/2}) \\
                &\times \int_{0}^{t} \Big(\big\|\bar{D}^{s} z(\tau, \cdot)\big\|_{L^{2}(\Omega)}^{2} +
                \big\|\bar{D}^{s - 1} \theta(\tau, \cdot)\big\|_{H^{1}(\Omega)}^{2} +
                \big\|\bar{D}^{s - 1} p(\tau, \cdot)\big\|_{L^{2}(\Omega)}^{2}\Big) \mathrm{d}\tau.
			\end{split}
			\notag
		\end{equation}
		The assertion of Theorem \ref{THEOREM_APPENDIX} is now a direct consequence of Gronwall's inequality.
	\end{proof}
\end{appendix}

\section*{Acknowledgment}
IL's research was partially funded by the NSF-DMS Grant \# 1713506.
MP's work on this project was supported by the Deutsche Forschungsgemeinschaft (DFG) through CRC 1173.
MP is also grateful to Dinstinguished Professors Irena Lasiecka and Roberto Triggiani for their hospitality during his research visit to the University of Memphis, TN.
XW thanks Professor Roland Schnaubelt for hosting him during his research stay at Karlsruhe Institute of Technology, Germany.

\bibliographystyle{plain}
\bibliography{bibliography}
\end{document}